\documentclass[12pt,reqno]{article}
\pdfoutput=1
\usepackage[a4paper]{geometry}
\usepackage{color}
\usepackage{amsmath}
\usepackage{amsthm}
\usepackage{amssymb}
\usepackage{graphicx}
\usepackage{setspace}
\usepackage[authoryear]{natbib}
\numberwithin{equation}{section}
\numberwithin{table}{section}
\numberwithin{figure}{section}

\usepackage[unicode=true,
 bookmarks=true,bookmarksnumbered=true,bookmarksopen=true,bookmarksopenlevel=2,
 breaklinks=false,pdfborder={0 0 0},pdfborderstyle={},backref=false,colorlinks=true]
 {hyperref}
\hypersetup{
 pdfborderstyle={},urlcolor=blue,linkcolor=red,citecolor=blue}

\makeatletter

\providecommand{\tabularnewline}{\\}

  \theoremstyle{plain}
  \newtheorem{prop}{\protect\propositionname}
\theoremstyle{plain}
\newtheorem{thm}{\protect\theoremname}
  \theoremstyle{plain}
  \newtheorem{lem}{\protect\lemmaname}
  \theoremstyle{plain}
  \newtheorem{assumption}{\protect\assumptionname}
 \theoremstyle{definition}
  \newtheorem{example}{\protect\examplename}
    \newtheorem*{prop*}{\protect\propositionname}
  \theoremstyle{plain}

\usepackage{paralist}
\usepackage{caption}
\usepackage{subcaption}
\usepackage{tikz}
\usetikzlibrary{arrows,decorations,decorations.pathreplacing,calc,matrix}
\usetikzlibrary{patterns}

\definecolor{Red}{rgb}{1,0,0}
\definecolor{Blue}{rgb}{0,0,1}
\definecolor{Green}{rgb}{0,1,0}
\definecolor{magenta}{rgb}{1,0,.6}
\definecolor{lightblue}{rgb}{0,.5,1}
\definecolor{lightpurple}{rgb}{.6,.4,1}
\definecolor{gold}{rgb}{.6,.5,0}
\definecolor{orange}{rgb}{1,0.4,0}
\definecolor{hotpink}{rgb}{1,0,0.5}
\definecolor{newcolor2}{rgb}{.5,.3,.5}
\definecolor{newcolor}{rgb}{0,.3,1}
\definecolor{newcolor3}{rgb}{1,0,.35}
\definecolor{darkgreen1}{rgb}{0, .35, 0}
\definecolor{darkgreen}{rgb}{0, .6, 0}
\definecolor{darkred}{rgb}{.75,0,0}

\allowdisplaybreaks[4]

\def\half{\frac{1}{2}}

\def\r{\rho}

\def\a{\alpha}

\def\bq{\begin{equation}}
\def\eq{\end{equation}}
\def\ba{\begin{eqnarray}}
\def\ea{\end{eqnarray}}
\def\bas{\begin{eqnarray*}}
\def\eas{\end{eqnarray*}}

\usepackage{paralist}
\renewenvironment{itemize}[1]{\begin{compactitem}#1}{\end{compactitem}}
\renewenvironment{enumerate}[1]{\begin{compactenum}#1}{\end{compactenum}}

  \providecommand{\assumptionname}{Assumption}
  \providecommand{\examplename}{Example}
  \providecommand{\lemmaname}{Lemma}
  \providecommand{\propositionname}{Proposition}

\providecommand{\theoremname}{Theorem}

  \providecommand{\assumptionname}{Assumption}
  \providecommand{\examplename}{Example}
  \providecommand{\lemmaname}{Lemma}
  \providecommand{\propositionname}{Proposition}

\providecommand{\theoremname}{Theorem}

  \providecommand{\assumptionname}{Assumption}
  \providecommand{\examplename}{Example}
  \providecommand{\lemmaname}{Lemma}
  \providecommand{\propositionname}{Proposition}

\providecommand{\theoremname}{Theorem}

  \providecommand{\assumptionname}{Assumption}
  \providecommand{\examplename}{Example}
  \providecommand{\lemmaname}{Lemma}
  \providecommand{\propositionname}{Proposition}

\providecommand{\theoremname}{Theorem}

  \providecommand{\assumptionname}{Assumption}
  \providecommand{\examplename}{Example}
  \providecommand{\lemmaname}{Lemma}
  \providecommand{\propositionname}{Proposition}

\providecommand{\theoremname}{Theorem}

  \providecommand{\assumptionname}{Assumption}
  \providecommand{\examplename}{Example}
  \providecommand{\lemmaname}{Lemma}
  \providecommand{\propositionname}{Proposition}
\providecommand{\theoremname}{Theorem}

\makeatother

  \providecommand{\assumptionname}{Assumption}
  \providecommand{\examplename}{Example}
  \providecommand{\lemmaname}{Lemma}
  \providecommand{\propositionname}{Proposition}
\providecommand{\theoremname}{Theorem}

\begin{document}

\title{Optimal Dynamic Allocation of Attention\thanks{ We thank Martino Bardi, Guy Barles, Hector Chade, Martin Cripps,
Mark Dean, Mira Frick, Duarte Goncalves, Johannes H\"{o}rner, Han Huynh, Philippe Jehiel,
Ian Krajbich, Botond K\H{o}szegi, Suehyun Kwon, George Mailath, Filip
Mat\v{e}jka, Pietro Ortoleva, Andrea Prat, Ed Schlee, Sara Shahanaghi, Philipp Strack,
Tomasz Strzalecki, Glen Weyl, Michael Woodford, Xingye Wu, Weijie
Zhong, and audiences at numerous conferences and seminars for helpful
comments; and Timur Abbyasov for excellent research assistance.\protect \\
 Che: Department of Economics, Columbia University, 420 W. 118th Street,
1029 IAB, New York, NY 10025, USA; email: \protect\protect\href{mailto:yeonkooche@gmail.com}{yeonkooche@gmail.com};
web: \protect\protect\href{http://blogs.cuit.columbia.edu/yc2271/}{http://blogs.cuit.columbia.edu/yc2271/}.\protect \\
 Mierendorff: Department of Economics, University College London,
30 Gordon Street, London WC1H 0AX, UK; email: \protect\protect\href{mailto:k.mierendorff@ucl.ac.uk}{k.mierendorff@ucl.ac.uk};
web: \protect\protect\href{http://www.homepages.ucl.ac.uk/\~uctpkmi/}{http://www.homepages.ucl.ac.uk/$\sim$uctpkmi/}. }
}

\author{Yeon-Koo Che\qquad{}Konrad Mierendorff}

\date{This draft: December 15, 2018\\First draft: April 18, 2016}
\maketitle
\begin{abstract}
We consider a decision maker (DM) who, before taking an action, seeks information by allocating her limited attention dynamically over different news sources that are biased toward alternative actions. Endogenous choice of information generates rich dynamics: The chosen news source either reinforces or weakens the prior, shaping subsequent attention choices, belief updating, and the final action. The DM adopts a learning strategy biased toward the current belief when the belief is extreme and  against that belief when it is moderate.  Applied to consumption of news media, observed behavior exhibits an ``echo-chamber'' effect for partisan voters and a novel ``anti echo-chamber'' effect for moderates.
\medskip{}

\end{abstract}

\section{Introduction}

Information is central to decision making. Individuals, firms, and
government agencies often expend significant resources to evaluate
their choices in consumption, investment, and public projects.   This is particularly so in high-stakes deliberation; e.g.,
when a voter deliberates on alternative candidates, when a researcher investigates a hypothesis, or when a judge or a juror weighs a defendant's guilt in a criminal case.
In such situations and others, decision makers have access to different
``news'' sources or diverse views on their actions. For instance, voters may expose themselves to like-minded news channels or to opposite-minded ones.
Jurors may hear adversarial lawyers advocating opposing views
on the case. Individuals also make attention choices in their internal processes of deliberation and thinking. For instance,
a researcher may expend efforts to either prove or disprove a hypothesis.

In this paper, we ask how a decision maker (DM) should allocate her
limited attention across different news sources  or different
deliberation strategies dynamically over time, and how that process shapes her choice of action.  An important aspect of this problem is how long the DM should search
for information before stopping to take an action. This \emph{stopping
problem} has been studied extensively by many authors, starting with
\citet{wald:47} and \citet{arrow/blackwell/girshick:49}. While sharing
the premise that information is costly and takes time to
arrive, the contribution of the current paper is to study how a DM allocates her attention over different \emph{types} of news sources.

In our model the DM faces binary actions, $r$ and $\ell$,
which are optimal in states $R$ and $L$, respectively. The state
is initially unknown, and the DM has a prior belief. At each point
in time, the DM may stop and take an action which is irreversible,
or she may acquire more information about the state. In the latter case,
she incurs a flow cost and payoffs are discounted.

Information can be received from of two sources: an $L$-biased or an $R$-biased news source.%
\footnote{\label{fn:bias_intro}We assume that the DM is Bayesian, so that a news source cannot systematically bias her belief.  The term ``bias'' here refers to the frequency of a signal favoring one state, which will be clarified in detail later.} %
The $L$-biased news source always sends an
$L$-signal in state $L$ and sometimes also in state $R$. Otherwise, it sends an $R$-signal. Since the $R$-signal is sent only in state $R$, it fully reveals the state.
Symmetrically, the $R$-biased news source is biased toward sending
an $R$-signal, except that in state $L$ it  occasionally reveals the state to be
$L$. At each point in time, the DM has a unit budget of attention to allocate between these
two news sources, and she may ``multi-home'' by arbitrarily dividing
her attention between the two sources. In our continuous time model, these two news sources reduce to
two Poisson processes that each generate breakthrough news revealing one
state. In the absence of a breakthrough, each source leads to continuous updating of the belief in the direction of the source's bias. We show that these Poisson processes
can be justified as optimal within a class of experiments which encompass general non-conclusive Poisson signals.

The main trade-off in our model is the decision between news sources that are biased in favor of, or against one's current belief.
We obtain a novel characterization of the optimal learning/attention strategy. 
While our model allows for general strategies, we show that for each
prior belief, the DM optimally uses one of three simple heuristics:
(i) \emph{immediate action}, (ii) \emph{own-biased learning}, and
(iii) \emph{opposite-biased learning}; and she never switches between these different
modes of learning.

\textbf{Immediate action} is a simple strategy where the DM takes
an optimal action given her prior without acquiring any information.
\textbf{Own-biased learning} focuses attention on the news source
that is biased toward the state that the DM finds \emph{relatively likely}.
An example is to focus on the $R$-biased news source when state $R$
is relatively likely. Given this strategy, the DM will take action
$\ell$ if breakthrough news reveals the state to be $L$. Otherwise, and more
likely, the DM becomes more convinced of state $R$, which leads to
further own-biased learning. Eventually,
she becomes sufficiently certain that the state is $R$---her belief reaches a stopping boundary, and she chooses action $r$ without fully learning the state. 
\textbf{Opposite-biased learning} focuses attention on the
news source biased toward the state she finds \emph{relatively unlikely}.
An example is to focus initially on the $L$-biased news source when
state $R$ is relatively likely. When following this strategy, the DM becomes less confident about
state $R$ when no breakthrough news arrives. Eventually, she becomes
so uncertain that she switches to a second phase where she divides
her attention equally between the $R$-biased and the $L$-biased news source.
She continues to acquire information and never stops until a breakthrough reveals the true state.

The optimality of the alternative learning strategies depends on the
parameters as well as the DM's prior. In particular, the cost of information
as measured by the flow cost and discounting is important. Not surprisingly,
if information is very costly, the DM takes an immediate action for
all beliefs. For moderate information acquisition costs, we show that
the DM optimally takes an immediate action when she is extremely certain,
while she employs own-biased learning when she is more uncertain.
Finally, if information acquisition costs are low, immediate action
is again optimal for extreme priors, and own-biased learning is optimal
for less extreme priors. For very uncertain priors, however, opposite-biased
learning becomes optimal.

The intuition behind the optimal policy is explained by a trade-off
between \emph{accuracy} and \emph{delay}. With an extreme belief,
a fairly accurate decision can be made even without evidence, so further information acquisition
is worth relatively little compared to the delay it causes. Conversely,
with a less extreme belief, information acquisition is more valuable.
This explains why the \emph{experimentation region} contains moderate
beliefs and the \emph{stopping region} is located at the extreme ends
of the belief space. This trade-off also explains which strategy is
optimal inside the experimentation region.  Opposite-biased learning will
lead to a fully accurate decision because the DM never takes an action
before learning the state, but this could take a long time. By contrast, own-biased learning is likely to produce a quick decision, because when no breakthrough arrives, it takes only a finite period of time for the DM to reach the stopping boundary and take an action. The price of a quick decision is lower accuracy, since the DM sometimes takes an incorrect action when she reaches the stopping region without breakthrough news.
When the DM is already quite certain,
under own-biased learning
the time needed to reach the stopping boundary is very short, 
and the higher accuracy of opposite-biased learning is less valuable.
This explains  why the DM chooses own-biased learning when she is more certain and
opposite-biased learning only when she is more uncertain.  An implication is that
\emph{a ``skeptic'' is more likely to make an accurate decision
but with a longer delay than a ``believer}.''  This prediction---particularly the dependence of decision accuracy on the prior beliefs---constitutes an important difference 
relative to the existing literature, as will be explained in Section \ref{sec:optimal-strategy}. 

Our model yields rich implications in terms of dynamic feedback between the
DM's selective exposure to a news source and her belief updating.
This feedback is particularly relevant for voters who  consult media, 
say before an election.
Our results imply that optimal media choice leads to an ``echo-chamber
effect,'' where voters with relatively extreme beliefs subscribe
to own-biased media that are likely to reinforce their prior beliefs.
With their beliefs reinforced, such voters further subscribe to own-biased
news media, repeating the same process until they become sufficiently
convinced. The resulting feedback loop results in polarization of beliefs. Interestingly, with sufficiently informative media, this
effect is reversed for voters with moderate beliefs. They optimally
seek opposite-biased outlets. As a result, they are likely to become
more skeptical about their initial beliefs. With growing skepticism,
such voters will eventually multi-home both types of media outlets
until they receive conclusive news that leads them to make up their
minds. Their behavior thus exhibits an ``anti-echo chamber effect.''

\subsection*{Relation to the Literature}
Our model incorporates endogenous choice of information in an optimal stopping framework \`a la \cite{wald:47}. Optimal stopping problems with exogenous information have been analyzed in a Poisson framework \citep[see][ch.~VI]{Peskir2006}, but economic applications have focused on drift-diffusion models (DDM) in which the signal follows a Brownian motion with a drift determined by the state. \cite{moscarini/smith:01} extend the stopping problem by allowing for an endogenous choice of signal precision. Other applications of DDMs include \cite{Chan2016}, \cite{fudenberg:15}, \cite{Henry2017}, and \cite{McClellan2017}.
An exception is \citet{nikandrova/pancs:15} who consider the problem
of selectively learning about different investment projects in a Poisson framework. In their model, the payoffs of final actions are uncorrelated whereas our model assumes negatively correlated payoffs.\footnote{See also the recent paper by \citet{Mayskaya2016}, and \citet{Ke2016}, which is concurrent to our paper.} 

Ultimately, whether a DDM or a Poisson model is appropriate depends on the specific application at hand.%
\footnote{DDMs, which lead to a continuous belief process, are more suitable for the problem of learning the properties of a data-generating process from a sequence of samples, as in clinical trials, or for the analysis of statistical information. Poisson models, which lead to discontinuous updating, are useful to model the discovery of individual pieces of information that are very informative, as is common in the political sphere, in criminal investigations, or when a scientist searches for a breakthrough insight that will prove or disprove a hypothesis.} %
From a theoretical point of view, \cite{Zhong2017} justifies Poisson learning as  an optimal information choice in a model that also allows for learning from a diffusion process.\footnote{His model differs from ours in that it assumes a posterior separable cost function \citep[see e.g.][]{Caplin2018}. Much like the rational inattention model described below,  the cost of a Blackwell experiment depends on the DM's belief in his model. By contrast, feasible experiments  and their costs do not depend on the DM's belief in our model.  Due to this difference, our models are not directly comparable; his framework cannot be used to formulate our model (see also footnote \ref{fn:zhong} in Section \ref{sec:Micro-foundation}).}%

Our model shares a common theme with the rational inattention model (henceforth, RI) introduced by \citet{sims:03} and further developed by \citet{matejka/mckay:15} and \citet{steiner/stewart/matejka:15}.  Like our paper, the RI models explain individual choice as resulting from the optimal allocation of limited attention over diverse information.   However, our model differs from the RI models in two respects.  One feature of the RI model is that  the same Blackwell experiment entails different costs for different beliefs. 
We do not allow for such belief-dependence of the information technology. Thus any dependence of information choice on beliefs arises from the DM's incentives, rather than the technology she faces.   Second and more important,  the RI model abstracts from the precise dynamic process of allocating one's attention.  By contrast, our objective is to unpack the ``black box'' and explicitly characterize the DM's dynamic attention choice.  \citet{steiner/stewart/matejka:15}
extend the RI model to a dynamic setting in which a DM chooses a sequence of multiple actions.  However, their main objective is to characterize stochastic choice rather than the dynamic information acquisition, the main focus of our paper.  As they acknowledge, the information acquisition strategy implementing the optimal stochastic choice is not uniquely pinned down (see p. 527 of  \citet{steiner/stewart/matejka:15}).
 \cite{Hebert-Woodford:17} and \cite{morris-strack:17} provide a link between RI and DDM models by showing that a class of static reduced-form cost functions, including RI, can be microfounded by dynamic DDM models.
 
Our model leads to rich predictions about the stochastic choice function that are not obtained in the DDM or RI framework. In the latter two, the accuracy of the DM's decision is independent of the prior (conditional on information acquisition). In our model, the accuracy varies with the prior.%
\footnote{Note however, that this dichotomy is less clear in decision problems with a richer state-space such as \cite{fudenberg:15} and \cite{Ke2016}. A complete comparison of the implications of continuous versus discontinuous learning is beyond the current state of the literature.} %
There are two reasons for this difference. First, the endogenous choice of information sources leads to different modes of learning depending on the prior, which results in significant differences in the accuracy and delay. Second, in our Poisson model, the DM may reach a decision either following a breakthrough, or after the belief drifts to a stopping bound. In the DDM framework, all decisions are reached after drifting to a decision boundary.\footnote{Note that we are comparing our model with discontinuous learning and endogenous information choice to a DDM model with continuous learning and exogenous information. In our two-state model, it is not possible to formulate a DDM with multiple information sources making it impossible to describe separately the effect of discontinuous learning and endogenous information.} (See Section \ref{sec:optimal-strategy} for further discussion.) 

The Poisson signal structure was introduced in bandit problems by \citet{keller:05}. The negatively-correlated
bandits model of \citet{klein/rady:11} 
parallels the choice between two biased Poisson processes studied in the current paper.
However, there are two fundamental differences. First, there is difference in timing. In our model, exploiting a payoff requires the DM to stop learning, whereas in bandit problems, she can exploit payoffs while learning.
Second, in bandit models, information sources are linked to exploitation of specific arms. For these reasons, a distinct characterization
emerges; for instance, there is no analogue to our ``own-biased learning'' in the bandit literature.\footnote{This is also the case in \citet{Damiano2017} who add additional learning to a Poisson bandit model.}

Finally, the current paper is related to the media choice literature.
It has been observed before that a Bayesian voter may find it optimal
to consume news from a biased outlet (see \citet{Calvert1985}, \citet{Suen2004},
and \citet{burke:08}). In particular, \citet{Suen2004} observes self-perpetuation
and polarization of beliefs. As we elaborate later, however, these
models consider very special cases that prevent nontrival dynamics
 from emerging.

The paper is organized as follows. Section \ref{sec:Model}
presents the model. Section \ref{sec:Example} presents an example
to illustrate the main results. Section \ref{sec:Analysis_Baseline}
characterizes the optimal policy. Section \ref{sec:Application} applies
the model to media choice. Section \ref{sec:Extensions} extends the
model in several directions.   Section
\ref{sec:Conclusion} concludes. Proofs are defered to Appendix
\ref{sec:Main_Appendix} and the Supplemental Material.

\section{Model\label{sec:Model}}

\paragraph*{States, Actions and Payoffs.}

A DM must choose from two \emph{actions,} $r$ or $\ell$, whose payoffs
depend on the unknown \emph{state} $\omega\in\{R,L\}$. The \emph{payoff}
of taking action $x\in\{r,\ell\}$ in state $\omega$ is denoted by $u_{x}^{\omega}\in\mathbb{R}$.
We label states and actions such that it is optimal to match the state,
and assume that the optimal action yields a positive payoff\textemdash that
is, $u_{r}^{R}>\max\left\{ 0,u_{\ell}^{R}\right\} $ and $u_{\ell}^{L}>\max\left\{ 0,u_{r}^{L}\right\} $.\footnote{Note that we allow for $u_{x}^{R}=u_{x}^{L}$ so that one action $x\in\{r,\ell\}$
can be a \emph{safe action}. We rule out the trivial case in which
$u_{x}^{\omega}\ge u_{y}^{\omega}$ for $x\neq y$, in both states
$\omega=R,L$.} The DM may delay her action and acquire information. In this case,
she incurs a flow cost of $c\ge0$ per unit of time. In addition,
her payoffs (and the flow cost) are discounted exponentially at rate
$\rho\ge0$. Either $c$ or $\rho$ may be zero, but not both.

The DM's \emph{belief} is denoted by the probability $p\in[0,1]$
that the state is $R$. Her \emph{prior belief} at time $t=0$ is
denoted by $p_{0}$. If the DM chooses her action optimally without
information acquisition, then given belief $p$, she will realize
an expected payoff of $U(p):=\max\{U_{r}(p),U_{\ell}(p)\}$, where
$U_{x}(p):=pu_{x}^{R}+(1-p)u_{x}^{L}$ is the expected payoff of taking
action $x$. $U(p)$ takes the piece-wise linear form depicted in
Figure \ref{fig:examples_optimal_solution_B} on page \pageref{fig:examples_optimal_solution_B}.

\paragraph*{Information Acquisition and Attention.}

We model information acquisition in continuous time. At each point
in time, the DM may allocate one unit of \emph{attention} across an \emph{$L$-biased} and an \emph{$R$-biased} news source.  
The $L$-biased source sends Poisson breakthrough news only in state $R$ with an arrival rate of $\lambda>0$, and the $R$-biased source sends breakthrough news only in state $L$ with an arrival rate of $\lambda$. Since, for a given source, a breakthrough arrives only in one state, it conclusively reveals the true state.  
In this sense, paying attention to the $L$-biased source say can be interpreted as ``looking for  $R$-evidence.''  Indeed, this interpretation describes some circumstances well and is hence useful to keep in mind. %
For instance, one can interpret this as a judge or juror focusing attention to evidence proving a certain state---guilt or innocence of a suspect, or a scientist seeking firm evidence that either proves or disproves a certain hypothesis. %
For the media choice application, however, interpreting news-sources according to their ``biases'' is more useful.

To get a better understanding of this interpretation, it is useful to study the \emph{statistical experiments} induced by different attention choices.  Suppose the DM pays full attention to the $L$-biased news source for a short duration $dt>0$.
Then, she receives one of two signals---breakthrough news (``signal $\sigma_{R}^{L}$''), or its absence (``signal $\sigma_{L}^{L}$'').  Panel (a) of Figure \ref{fig:two-biased-experiments} describes the probabilities of receiving these two signals in each state. 
 \medskip
\begin{figure}[htp]
\hfill
\begin{tabular}{ccc}
		\multicolumn{3}{c}{(a) $L$-biased experiment: $\sigma^{L}$}\tabularnewline
		\hline 
		\hline 
		state/signal  & $\sigma_{L}^{L}$  & $\sigma_{R}^{L}$ \tabularnewline
		\hline 
		$L$  & $1$  & $0$\tabularnewline
		$R$  & $1-\lambda dt$  & $\lambda dt$ \tabularnewline
		\hline 
\end{tabular}
\hfill
\begin{tabular}{ccc}
 \multicolumn{3}{c}{(b) $R$-biased experiment: $\sigma^{R}$}\tabularnewline
	\hline 
	\hline 
	state/signal  & $\sigma_{L}^{R}$  & $\sigma_{R}^{R}$ \tabularnewline
	\hline 
	$L$  & $\lambda dt$  & $1-\lambda dt$\tabularnewline
	$R$  & $0$  & $1$\tabularnewline
	\hline 
\end{tabular}\hfill
\medskip
\caption{Experiments induced by two Poisson signals.} \label{fig:two-biased-experiments}
\end{figure}

 Experiment $\sigma^{L}$ is ``biased'' toward $L$ in the sense
of sending the $L$-favoring signal $\sigma_{L}^{L}$ excessively\textemdash always
in state $L$ but even in state $R$ with some probability.\footnote{As mentioned in Footnote \ref{fn:bias_intro}, we use the term ``bias'' only in this sense. This terminology is in keeping with the media choice literature as we will discuss in Section \ref{sec:Micro-foundation}.} An implication of this is that $\sigma_{L}^{L}$ is a relatively
weak signal that moves the belief toward $L$ but not by much. By contrast, the signal $\sigma_{R}^{L}$\textemdash the $R$-favoring
signal from source $\sigma^{L}$\textemdash reveals conclusively that
the state is $R$.  Similarly, if the DM chooses the $R$-biased news source for a duration $dt>0$,
 then this induces the experiment  $\sigma^{R}$ depicted in Panel (b).  This experiment
   is biased toward $R$ in
the sense of sending the $R$-favoring signal excessively.  
While our model focuses on conclusive Poisson experiments, we show in Section \ref{sec:Micro-foundation} that this focus can be justified in a more general class of information technologies.

We assume that the DM may allocate any fraction $\alpha\in[0,1]$ of her attention to $L$-biased news and the remaining fraction $\beta=1-\a$ to $R$-biased news. In this case, she receives $R$-evidence with
arrival rate $\alpha\lambda$ in state $R$, and $L$-evidence with
arrival rate $\beta\lambda$ in state $L$. An interior attention
choice can be interpreted as ``multi-homing''---or switching back and forth arbitrarily frequently---between the two news sources.
We denote the DM's attention strategy by $(\alpha_{t})=(\alpha_{t})_{t\ge0}$,
and assume that $\alpha_{t}$ is a measurable function of $t.$

Suppose the DM uses the attention strategy $(\alpha_{t})$. In the
absence of a breakthrough, the DM's belief will evolve according
to Bayes rule:\footnote{Since the belief is a martingale, we have $\lambda\alpha_{t}p_{t}dt+\left(1-\lambda\alpha_{t}p_{t}dt-\lambda\beta_{t}\left(1-p_{t}\right)dt\right)\left[p_{t}+\dot{p}_{t}dt\right]=p_{t}.$
Dividing by $dt$ and letting $dt\rightarrow0$ yields \eqref{eq:pdot}.} 
\begin{equation}
\dot{p}_{t}=-\lambda(\alpha_{t}-\beta_{t})p_{t}(1-p_{t})=-\lambda(2\alpha_{t}-1)p_{t}(1-p_{t}).\label{eq:pdot}
\end{equation}
The bias of a news source corresponds to the direction of updating
in the absence of breakthrough news. The higher $\alpha$, the more
$L$-biased is the mix of news sources that the DM pays attention
to. For instance, full attention to $L$-biased news ($\alpha=1$)
makes the DM pessimistic about state $R$ if no breakthrough
arrives. Consequently, her belief drifts towards $L$. Finally, note
that if the DM divides her attention equally between both news sources,
she never updates her belief in case of no breakthrough news: $\dot{p}_{t}=0$
if $\a_{t}=1/2$.

\section{Illustrative Examples \label{sec:Example}}

Before proceeding, we use simple examples to illustrate the main insights
of our results. The examples will highlight the role of dynamics by
considering one-period and two-period versions of our model.\footnote{These examples are similar in spirit to \citet{Suen2004} and \citet{burke:08},
except that information is not costly in these models, which leads
to a different characterization. In particular, dynamics has no significant
effect in these models, in contrast to the point made here.} For simplicity, assume the DM enjoys a payoff of $1$ when matching
the state and $-1$ otherwise ($u_{r}^{R}=u_{\ell}^{L}=1$, and
$u_{\ell}^{R}=u_{r}^{L}=-1$). Time is discrete with a period length
of $dt=1$, and the DM incurs a cost $c$ per period if she acquires
information. There is no discounting. 
In each period, the DM can either stop and take an irreversible action, or pay the cost $c$ to observe one of the experiments $\sigma^L$ and $\sigma^R$ in Figure \ref{fig:two-biased-experiments}.

\paragraph*{One-Period Problem.}

Suppose the DM can experiment only once before taking an action. For a prior
$p_{0}$ sufficiently high or low, the DM will optimally choose to
take an action immediately, since experimentation is costly. For less extreme priors
$p_{0}$, she will experiment. The question is which news source she
will pay attention to. We argue that it is optimal to pick the ``own-biased news source''---namely the experiment biased toward one's prior:  $\sigma^{R}$  if $p_{0}>1/2$ and $\sigma^{L}$ if $p_{0}<1/2$. 
 
The reason is that own-biased news lowers the chance of mistakes compared to opposite-biased news. Specifically, suppose $p_{0}>1/2$, so that the prior indicates that action $r$ is optimal, and assume that the DM
picks the opposite-biased experiment, $\sigma^L$.
There are two cases.  Suppose first $p_0$ is very high.  Then even when the DM receives the $L$-favoring signal $\sigma_{L}^{L}$, this will not move her belief enough to ``change'' her action to $\ell$, meaning she will choose $r$ regardless of the signal.  Hence, $\sigma^L$ is worthless to her.  By contrast, a DM choosing  $\sigma^R$ will be ``convinced'' by an $L$-signal to change her action to $\ell$, which makes $\sigma^R$ valuable. 

Next suppose that $p_{0}>1/2$ but lower, so that an $L$-signal
from $\sigma^{L}$ also leads the DM to choose $\ell$. Why is $\sigma^{L}$
still inferior to $\sigma^{R}$?  The reason is that  $\sigma^{L}$ chooses
 $\ell$ as the ``default action'' which is taken unless conclusive $R$-evidence is observed. 
 Since the prior $p_{0}>1/2$ favors state $R$, action $r$ is a better default; so the own-biased  experiment $\sigma^R$ is better because it leads to action $r$ unless conclusive $L$-evidence is observed.\footnote{\label{fn:one-period} To see precisely why $\sigma^{R}$ is more
valuable than $\sigma^{L}$ if $p_{0}>1/2$, we consider a (hypothetical)
experiment $S$ with three possible signals $S_L$, $S_0$, and $S_R$: 
\begin{center}
\begin{tabular}{cccc}
\hline 
state/signal  & $S_{L}$  & $S_{0}$  & $S_{R}$ \tabularnewline
\hline 
$L$  & $\lambda$  & $1-\lambda$  & $0$\tabularnewline
$R$  & $0$  & $1-\lambda$  & $\lambda$ \tabularnewline
\hline 
\end{tabular}
\par\end{center}
Clearly, $S$ is Blackwell more informative than $\sigma^{L}$ and
$\sigma^{R}$, respectively. $\sigma^{L}$ can be obtained from $S$
by observing only the partition $\left\{ \left\{ S_{L},S_{0}\right\} ,\left\{ S_{R}\right\} \right\} $,
and $\sigma^{R}$ can be obtained by observing only the partition
$\left\{ \left\{ S_{L}\right\} ,\left\{ S_{0},S_{R}\right\} \right\} $.
What is the optimal decision rule following experiment $S$? If $p_{0}>1/2$,
the DM should choose action $r$ if she observes a signal in $\left\{ S_{0},S_{R}\right\} $
and $\ell$ only if she observes $S_{L}$. Action $r$ is the ``default'' chosen after the uninformative signal $S_0$. This corresponds to the decision implemented by the partition corresponding to $\sigma^R$.  By contrast the partition corresponding
to $\sigma^{L}$ leads to action $\ell$ as the ``wrong default,'' which reduces the value of this experiment compared to the value of $\sigma^{R}$
(or of $S$).}

\paragraph*{Two-Period Problem.}

Now suppose the DM can experiment for up to two periods. After observing the
first signal, the DM may stop and take an action, or experiment
for another period. The problem facing the DM in the second period
is precisely the one-period problem above: Depending on her posterior
belief after the first experiment, she will either take an immediate
action, or choose an own-biased experiment. But what should she do
in the first period?

\begin{figure}
\begin{center}
\includegraphics[height=0.2\textheight]{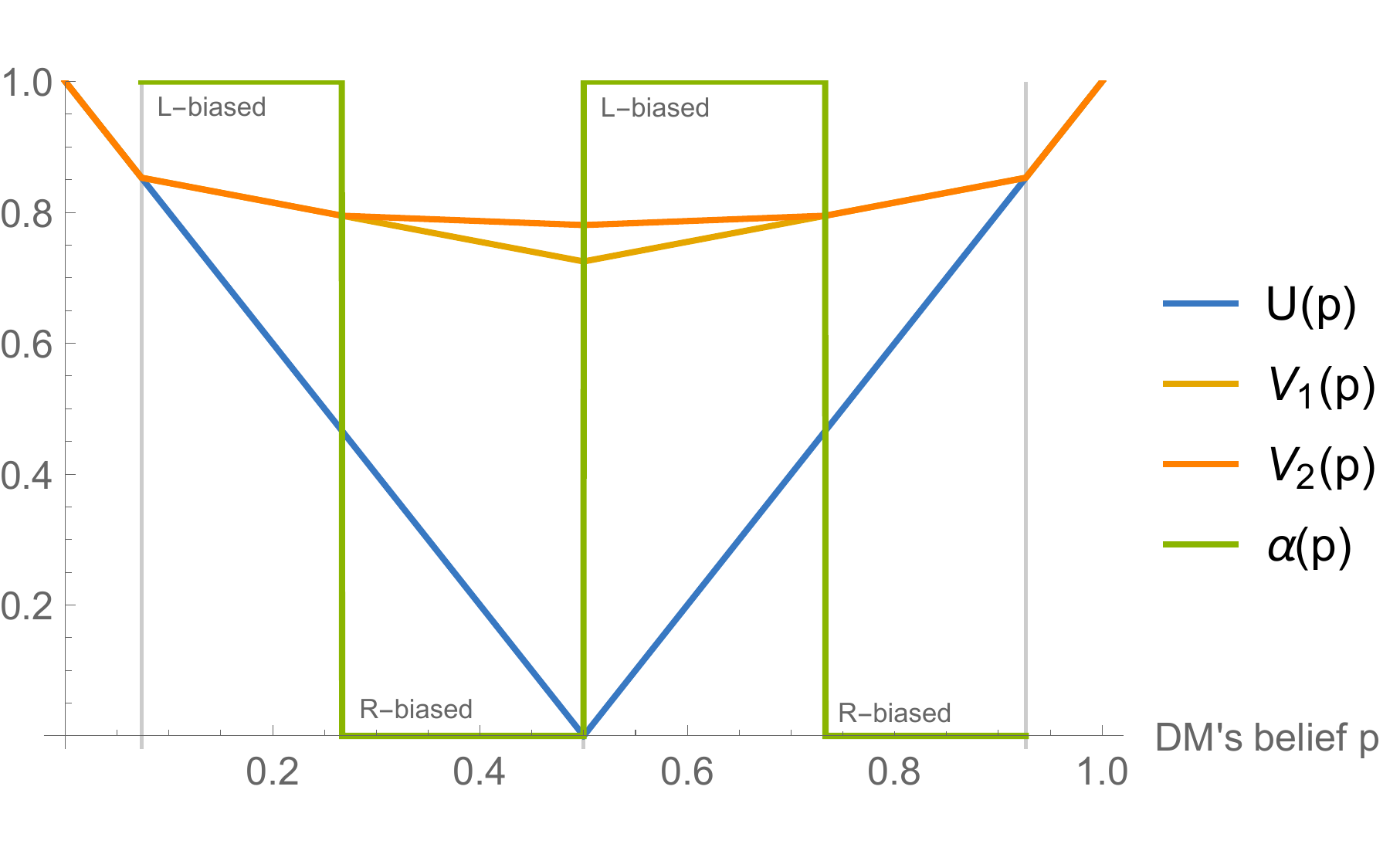}
\end{center}
\vspace{-20pt}
\caption{Optimal policy in two-period learning}
\medskip
\begin{footnotesize}\begin{spacing}{1} Note:   $\lambda=.85,c=.125$; $U(p)$ is the payoff from immediate action, $V_1(p)$ is the optimal payoff in the one-period problem, $V_2(p)$ is the optimal payoff in the two-period problem, and $\alpha(p)$ is the choice of experiment in the first (of two) periods, $\alpha=1$ corresponds to $\sigma^L$ and $\alpha=0$ corresponds to $\sigma^R$.\end{spacing}\end{footnotesize}
\label{fig:2-period} 
\end{figure}
Figure \ref{fig:2-period} illustrates a case where the optimal choice of experiment in the first period differs from that of the one-period problem.
For extreme values of $p_0$,
an immediate action is optimal as before. For less extreme priors ($p_0\in[.07,.93]$), the DM chooses to experiment.  When the belief is close
to the boundaries of the experimentation region ($p_0 \in [.07,.27]$ or $p_0\in [.73,.93]$), the DM chooses own-biased news in period one. Afterwards, she immediately
takes an action without any further experimentation. 
The intuition is that when her belief is already quite extreme, the
marginal value of increasing the accuracy for the decision is small,
so she experiments only once. Conditional on experimenting once, she
prefers own-biased news for the same reason as in the one-period problem.

 If the belief is moderate ($p_0\in [.27,.73]$), the DM picks an opposite-biased
news source. For instance, for $p_{0}=.7$, she chooses $\sigma^{L}$. Recall that this was not a good strategy in the one-period problem since an action had to be taken even after a ``weak''  
$L$-signal. In the two-period problem, however, the DM can continue to experiment.
This option value can make both opposite-biased and own-biased learning more attractive.
However, for own-biased learning, the posterior after the first experiment becomes more extreme, which limits the option value so that it is optimal to stop after period one.\footnote{It is in fact optimal to stop after period one for $p_0>.65$. We focus on this range of beliefs because it illustrates the difference between a static and a dynamic model most clearly.}
By contrast, opposite-biased learning leads to a second round of experimentation, so the option value is more significant in particular for moderate beliefs where the value of accuracy is high. 
This explains the optimality of the opposite-biased learning with longer delay for the moderate beliefs.

These examples suggest that dynamics matters for the optimal allocation of attention. Different learning strategies are optimal for different beliefs, and a learning strategy that is not myopically optimal becomes optimal when a DM can experiment for more than one period.  The pattern of the optimal strategy as well as the insights illustrated in these examples generalize to our continuous time model, which we now turn to.

\section{Analysis of the Optimal Strategy\label{sec:Analysis_Baseline}}

We now analyze the DM's optimal strategy in the continuous time model
introduced in Section \ref{sec:Model}.

\subsection{Formulation of the Problem}

The DM chooses an attention strategy $(\alpha_{t})=(\alpha_{t})_{t\ge0}$,
and a stopping time $T\in[0,\infty]$ at which a decision will be
 made if no   breakthrough news is received by then.\footnote{Given the linearity of the arrival rates in $\alpha$, the DM cannot benefit from
randomization in the continuous time model. For this reason, we only
consider deterministic strategies $(\alpha_{t})_{t\in\mathbb{R}_{+}}$.
Moreover, it suffices to consider strategies that specify $\alpha$
as a function of $t$ since the attention choice at time $t$ is only
relevant if the DM has not received any conclusive signals until time
$t$.} Her problem is thus given by 
\begin{align}
V^{*}(p_{0})= & \max_{(\alpha_t),T}\int_{0}^{T}e^{-\rho t}P_{t}\left[p_{t}\lambda\alpha_{t}u_{r}^{R}+(1-p_{t})\lambda\beta_{t}u_{\ell}^{L}-c\right]dt+e^{-\rho T}P_{T}\,U(p_{T}),\label{eq:DMs_probelm}
\end{align}
where $p_{t}$ satisfies \eqref{eq:pdot}, $\beta_{t}=1-\alpha_{t}$,
and $P_{t}=p_{0}e^{-\lambda\int_{0}^{t}\alpha_{s}ds}+(1-p_{0})e^{-\lambda\int_{0}^{t}\beta_{s}ds}$
is the probability that no signal is received by time $t$ given strategy
$(\alpha_{\tau})$. The integrand in the objective function captures
the payoffs from taking an action following discovery of evidence,
and the flow cost incurred until the DM stops. Specifically,
at each time $t$, conditional on no discovery so far (which occurs
with probability $P_{t}$), the strategy $\a_{t}$ leads to discovery
of $R$-evidence with probability $p_{t}\lambda\alpha_{t}$, and of
$L$-evidence with probability $(1-p_{t})\lambda\beta_{t}$, per unit time. The second
term accounts for the payoff from the optimal decision in case of
no discovery by $T$.\footnote{ For a given $(\alpha_{t})_{t\in\mathbb{R}_{+}}$, conditional on
no discovery, the posterior belief evolves according to a deterministic
rule (\ref{eq:pdot}). Since stopping matters only when there is no
discovery, it is without loss to focus on a \emph{deterministic} stopping
time $T$.}

The Hamilton-Jacobi-Bellman (HJB) equation for this problem is 
\begin{align}
c+\rho V(p) & =\max_{\alpha\in[0,1]}\begin{Bmatrix}\lambda\alpha p\,\left(u_{r}^{R}-V(p)\right)+\lambda(1-\alpha)(1-p)\,\left(u_{\ell}^{L}-V(p)\right)\\
-\lambda(2\alpha-1)p(1-p)V'(p)
\end{Bmatrix},\label{eq:HJB}
\end{align}
if $V(p)>U(p)$. If $V(p)=U(p)$, the LHS of \eqref{eq:HJB} should be no less than the RHS---in this case
$T(p)=0$ and immediate action is optimal.   The objective in \eqref{eq:HJB} is linear in $\alpha$,
which implies that the optimal policy is a bang-bang solution $\alpha^{*}(p)\in\{0,1\},$
except when the derivative of the objective vanishes.  While this observation narrows down our search, there is a large class of strategies consistent with bang-bang choices.  Ultimately, one needs to characterize the attention choice for each belief, which  we now turn to.

\subsection{Learning Heuristics}

We begin with several intuitive learning heuristics that the DM could
employ. These heuristics form basic building blocks for the
DM's optimal strategy. Specifically, it will be seen that for each prior, the optimal strategy employs the heuristic with the highest value and never switches even after the belief has moved away from the prior. The details
of the formal construction are presented in Appendix \ref{sec:Main_Appendix}.

\paragraph*{Immediate action (without learning).}

A simple strategy is to take an immediate action
and realize $U(p)$ without any information acquisition. Since information
acquisition is costly, this can be optimal if the DM is sufficiently
confident in her belief\textemdash that is, if $p$ is either sufficiently
high or sufficiently low.

\paragraph*{Own-Biased Learning.}

When the DM decides to experiment, one natural strategy is to focus attention 
on the news source that is biased toward the more likely state.  
 Formally, own-biased learning prescribes 
\begin{equation}
\alpha(p)=\begin{cases} 1 & \mbox{ if }p\in\left(\underline{p}^{*},\check{p}\right),\\
0 & \mbox{ if }p\in\left[\check{p},\overline{p}^{*}\right),
\end{cases}\label{eq:contr_alpha_thm1}
\end{equation}
for some reference belief $\check{p}$ and boundaries of the \emph{experimentation
	region} $\left(\underline{p}^{*},\overline{p}^{*}\right)$, which
will each be chosen optimally.\footnote{When payoffs are symmetric between $\ell$ and $r$, then $\check{p}$ equals  $1/2$. In general,  $\check p$ may not equal $1/2$.}   For instance, if $L$ is relatively likely, the DM  chooses the $L$-biased news source,  or equivalently looks for conclusive $R$-evidence.   In the absence of such ``contradictory'' evidence, the DM's belief drifts toward $L$.  The belief updating is illustrated by the direction of the arrows in Panel (a) of Figure \ref{fig:structure_of_optimal_solution}.  Eventually, the DM's belief will reach one of the boundary points
	$\underline{p}^{*}$ or $\overline{p}^{*}$, at which she is sufficiently
	certain to take an immediate action without conclusive evidence. 
	Since contradictory evidence is unlikely, the DM adopting this strategy can be seen as seeking to gradually rule  out the unlikely state. For example,  a   juror or a judge sympathetic to a defendant's innocence may try to rule out incriminating evidence by actively looking for it, or a mathematician convinced of her proof may try to rule out ``errors'' by actively searching for them.

\paragraph*{Opposite-Biased Learning.}

Alternatively, the DM could focus attention on a news source  biased toward the state she finds relatively unlikely.  Formally, opposite-biased
learning prescribes:
\begin{equation}
 \a(p)=\begin{cases}
 0 & \mbox{ if }p<p^{*},\\
 \half & \mbox{ if }p=p^{*},\\
 1 & \mbox{ if }p>p^{*},
\end{cases}\label{eq:alpha_conf} 
\end{equation}
for some  reference belief $p^{*}\in(0,1)$, which will be chosen
optimally.\footnote{In our notation, an asterisk ($*$) indicates an absorbing point, as in $p^{*},\underline{p}^{*},\overline{p}^{*}$. Cutoff beliefs without an asterisk indicate
points were the belief diverges (e.g. $\check{p}$, $\underline{p}$,
$\overline{p}$). An overline (underline) is used to denote cutoff
beliefs to the right (left).}  %
At the optimal $p^{*}$, the DM is indifferent between all $\alpha\in[0,1]$,
so that the bang-bang result from the previous section does not apply. %
For beliefs below $p^{*}$, the DM subscribes to the $R$-biased news source, or equivalently, she seeks ``confirmatory'' evidence supporting the likely state $L$.  An example of such a strategy could be that  a mathematician tries to prove a promising hypothesis and to disprove unpromising one.  
In the absence of such affirmative proof, the DM becomes more uncertain in her belief, which thus drifts inward.  The belief updating is illustrated in Panel (b) of Figure \ref{fig:structure_of_optimal_solution}, with arrows indicating the direction of Bayesian updating.  
As is clear from equation \eqref{eq:pdot} and \eqref{eq:alpha_conf}, her belief drifts from
both extremes towards the absorbing point $p^{*}$. Once $p^{*}$
is reached, the DM divides her attention equally between both news
sources. In this case, no further updating occurs, and she repeats
the same strategy until she obtains evidence that reveals the true
state.  

\begin{figure}
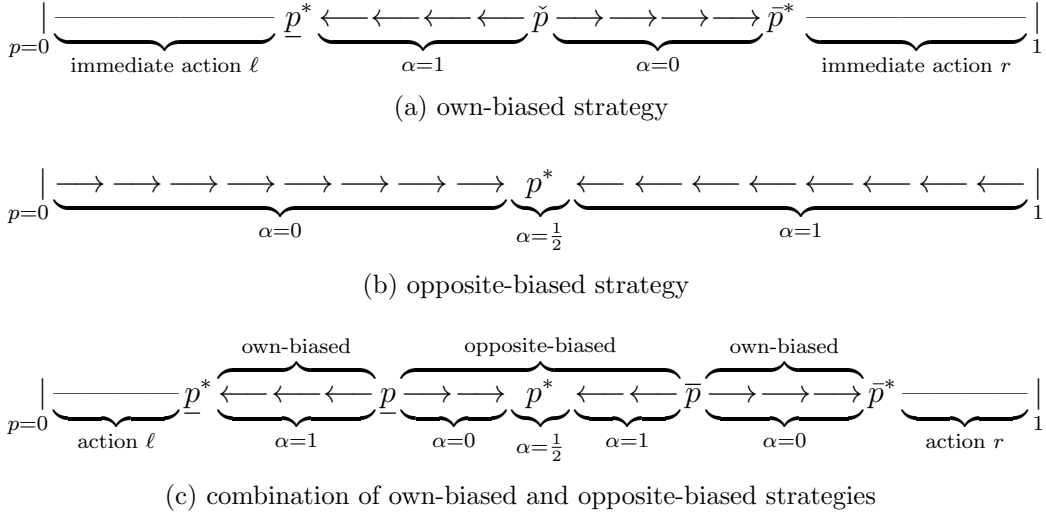


\noindent \begin{centering}
\[
\underset{p=}{\phantom{|}}\underset{0}{|}\underbrace{\text{\ensuremath{\vphantom{p^{*}}}---------------------}}_{\text{immediate action }\ell}\,\underline{p}^{*}\underbrace{\longleftarrow\longleftarrow\longleftarrow\vphantom{p^{*}}\negthickspace\negmedspace\longleftarrow}_{\alpha=1}\check{p}\underbrace{\longrightarrow\vphantom{p^{*}}\negthickspace\negmedspace\longrightarrow\longrightarrow\longrightarrow}_{\alpha=0}\bar{p}^{*}\,\underbrace{\text{\ensuremath{\vphantom{p^{*}}}---------------------}}_{\text{immediate action }r}\underset{1}{|}
\]
\par\end{centering}
\noindent \begin{centering}
{\footnotesize{}{}{}{}{}{}{}(a) own-biased strategy} 
\par\end{centering}

\noindent \begin{centering}
	\[
	\underset{p=}{\phantom{|}}\underset{0}{|}\underbrace{\longrightarrow\,\longrightarrow\,\longrightarrow\,\longrightarrow\,\longrightarrow\,\longrightarrow\,\longrightarrow\vphantom{p^{*}}\negthickspace\negmedspace\,\longrightarrow}_{\alpha=0}\underbrace{p^{*}}_{\alpha=\frac{1}{2}}\underbrace{\longleftarrow\vphantom{p^{*}}\negthickspace\negmedspace\,\longleftarrow\,\longleftarrow\,\longleftarrow\,\longleftarrow\,\longleftarrow\,\longleftarrow\,\longleftarrow}_{\alpha=1}\underset{1}{|}
	\]
	{\footnotesize{}{}{}{}{}{}{}(b) opposite-biased strategy} 
	\par\end{centering}

\noindent \begin{centering}
\[
\underset{p=}{\phantom{|}}\underset{0}{|}\underbrace{\text{\ensuremath{\vphantom{p^{*}}}------------}}_{\text{action }\ell}\underline{p}^{*}\overbrace{\underbrace{\longleftarrow\vphantom{p^{*}}\negthickspace\negmedspace\longleftarrow\longleftarrow}_{\alpha=1}}^{\text{own-biased}}\underline{p}\overbrace{\underbrace{\longrightarrow\vphantom{p^{*}}\negthickspace\negmedspace\longrightarrow}_{\alpha=0}\underbrace{p^{*}}_{\alpha=\frac{1}{2}}\underbrace{\longleftarrow\vphantom{p^{*}}\negthickspace\negmedspace\longleftarrow}_{\alpha=1}}^{\text{opposite-biased}}\overline{p}\overbrace{\underbrace{\longrightarrow\vphantom{p^{*}}\negthickspace\negmedspace\longrightarrow\longrightarrow}_{\alpha=0}}^{\text{own-biased}}\bar{p}^{*}\underbrace{\text{\ensuremath{\vphantom{p^{*}}}------------}}_{\text{action }r}\underset{1}{|}
\]
{\footnotesize{}{}{}{}{}{}{}{}(c) combination of own-biased
and opposite-biased strategies} 
\end{centering}
\centering{}\protect\caption{\label{fig:structure_of_optimal_solution}Structure of Heuristic Strategies
and Optimal Solution.}
\end{figure}
\subsection{Optimal Strategy}\label{sec:optimal-strategy}

The structure of the optimal policy depends on the \emph{cost of information}
$c$. Intuitively, the higher the flow cost, the lower  the net
value of experimenting. As will be seen, the experimentation region
expands as the cost of information falls. More interestingly, the
type of learning strategy employed also changes in a nontrivial way.
The following theorem shows that there are three cases. If $c$ is
very high, immediate action is always optimal (case (a)). For intermediate
values of $c$, the optimal strategy involves only own-biased learning
(case (b)). For low values of $c$, both own-biased and opposite-biased learning occur (case (c)). For the theorem, we denote the optimal
immediate action by $x^*(p)\in\arg\max_{x\in\left\{ r,\ell\right\} }U_{x}(p)$,
which is unique almost everywhere. 
\begin{thm}
\label{thm:optimal_attention_strategy} For given utilities $u_{x}^{\omega}$,
$\lambda>0$, and $\rho\ge0$, there exist $\overline{c}=\overline{c}(\rho,u_{x}^{\omega},\lambda)$
and $\underline{c}=\underline{c}(\rho,u_{x}^{\omega},\lambda)$, $\overline{c}\ge\underline{c}\ge0$,
with strict inequalities for $\rho$ sufficiently small, such that the
unique optimal strategy is characterized as follows:\footnote{The strategy is unique up to tie breaking at the beliefs $\underline{p}^{*},\overline{p}^{*},\underline{p},\overline{p},\check{p}$.
See  \eqref{eq:zetah} and \eqref{eq:zetal} in Appendix \ref{sec:Main_Appendix}
for explicit expressions for $\underline{c}$ and $\overline{c}$. See \eqref{eq:plsG} and \eqref{eq:pHsG} for $\underline{p}^*$ and $\overline{p}^*$, and \eqref{eq:pstarG} for $p^*$.}
\begin{enumerate}
\item (\textbf{No learning}) If $c\ge\overline{c}$, the DM takes action
$x^*(p)$ without any information acquisition. 
\item (\textbf{Own-biased learning}) If $c\in\left[\underline{c},\overline{c}\right)$,
there exist cutoffs $0<\underline{p}^{*}<\check{p}<\overline{p}^{*}<1$
such that for $p\in(\underline{p}^{*},\overline{p}^{*})$, the optimal policy
$\alpha^*(p)$ is given by \eqref{eq:contr_alpha_thm1}. If $p\not\in(\underline{p}^{*},\overline{p}^{*})$,
the DM takes action $x^*(p)$ without any information acquisition. 
\item (\textbf{Own-biased and Opposite-biased learning}) If $c<\underline{c}$,
then there exist cutoffs $0<\underline{p}^{*}<\underline{p}<p^{*}<\overline{p}<\overline{p}^{*}<1$
such that for $p\in(\underline{p}^{*},\overline{p}^{*})$, the optimal policy
is given by 
\begin{equation}
\alpha^*(p)=\begin{cases}
1, & \text{ if }p\in\left(\underline{p}^{*},\underline{p}\right),\\
0, & \text{ if }p\in\left[\underline{p},p^{*}\right),\\
\half & \mbox{ if }p=p^{*},\\
1, & \text{ if }p\in\left(p^{*},\overline{p}\right],\\
0, & \text{ if }p\in\left(\overline{p},\overline{p}^{*}\right).
\end{cases}\label{eq:mixed_alpha_thm1}
\end{equation}
If $p\not\in(\underline{p}^{*},\overline{p}^{*})$ the DM takes action
$x^*(p)$ without any information acquisition. 
\end{enumerate}
\end{thm}
In cases $(b)$ and $(c)$,
there are levels
of confidence, given by $\overline{p}^{*}$ and $\underline{p}^{*}$,
that the DM finds sufficient for making decisions without any evidence.
These beliefs constitute the boundaries of the experimentation region; namely, an immediate action is chosen outside these boundaries.

The optimal policy in case (b) is depicted in Panel (a) of Figure
\ref{fig:structure_of_optimal_solution} on page \pageref{fig:structure_of_optimal_solution} above. In case (c) the pattern
is more complex (see Panel (c) of Figure \ref{fig:structure_of_optimal_solution}).
Both own-biased and opposite-biased learning are optimal for some
beliefs. Theorem \ref{thm:optimal_attention_strategy} shows that
the opposite-biased region $(\underline{p},\overline{p})$ is always
sandwiched between two regions where the own-biased learning strategy is
employed. That is, near the boundaries of the experimentation region,
the own-biased strategy is always optimal.  Figure \ref{fig:examples_optimal_solution_B} shows the value of opposite-biased learning ($V_{opp}(p)$) and own-biased learning ($V_{own}(p)$) for high and low costs. It illustrates that the optimal policy picks the learning heuristic with the higher value. 
\begin{figure}[b]
\begin{tabular}{cc}
\includegraphics[height=0.17\textheight]{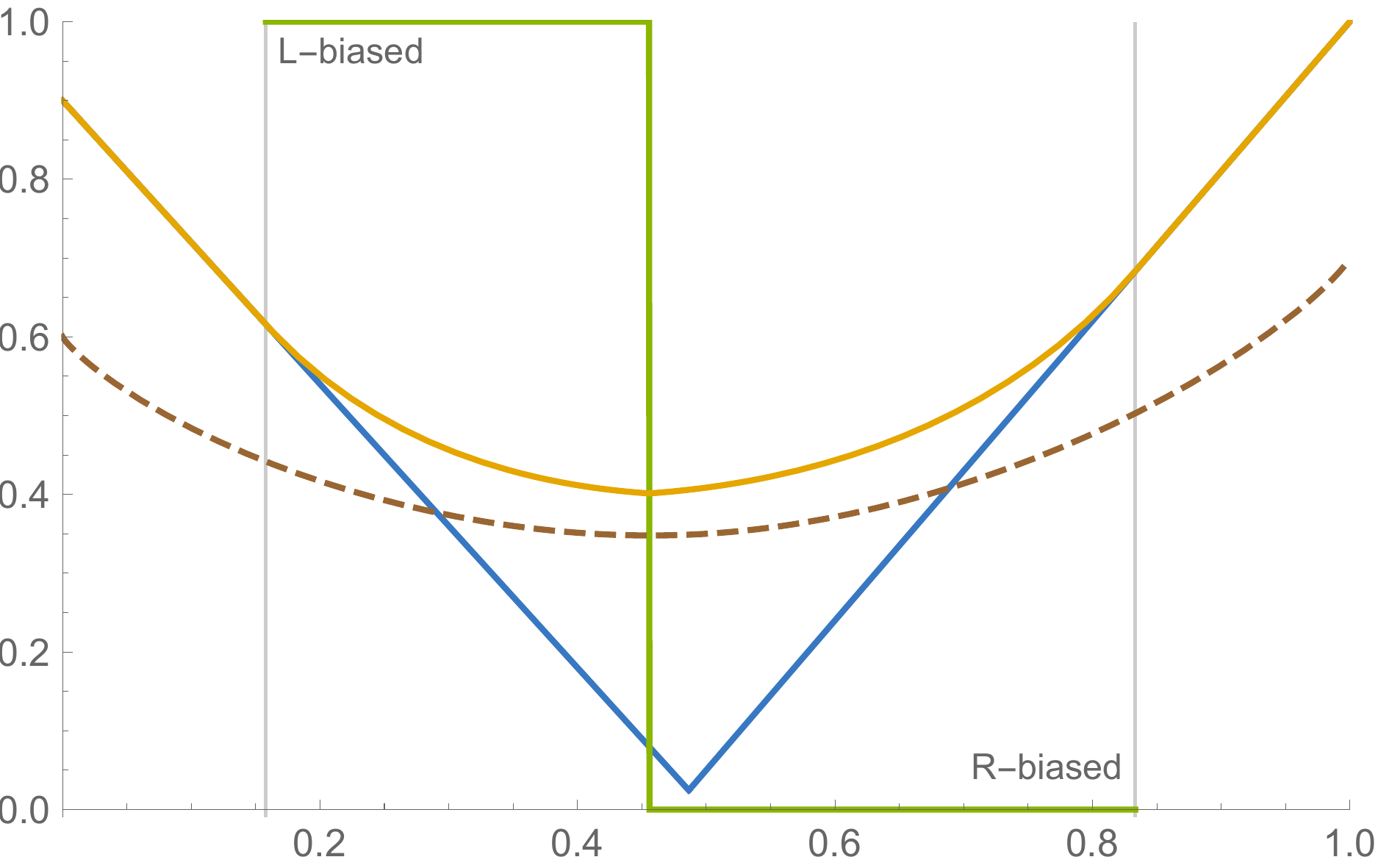}  & \includegraphics[bb=0bp 0bp 550bp 340bp,height=0.17\textheight,clip]{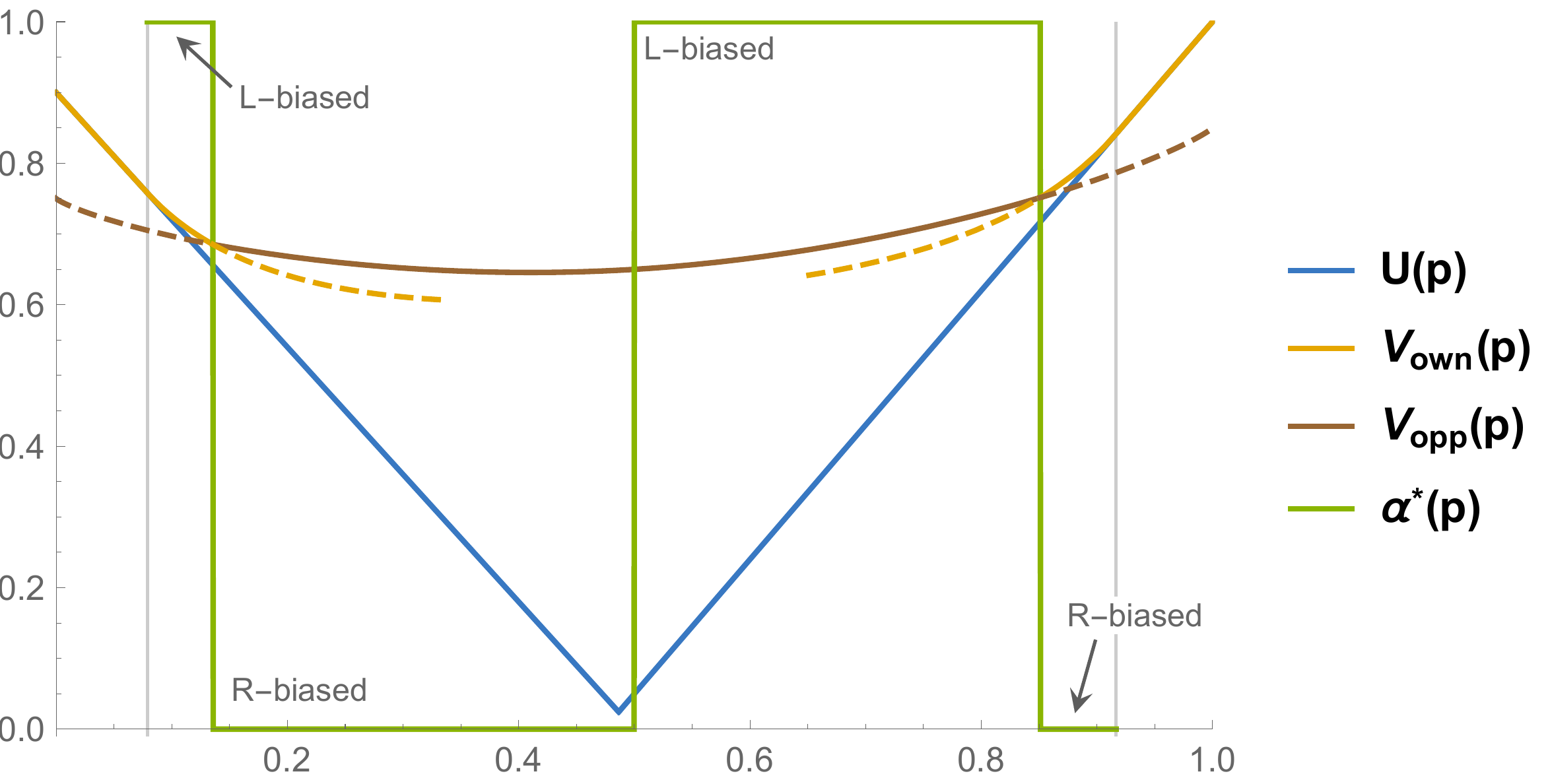}
\includegraphics[bb=560bp 0bp 675bp 340bp,height=0.17\textheight,clip]{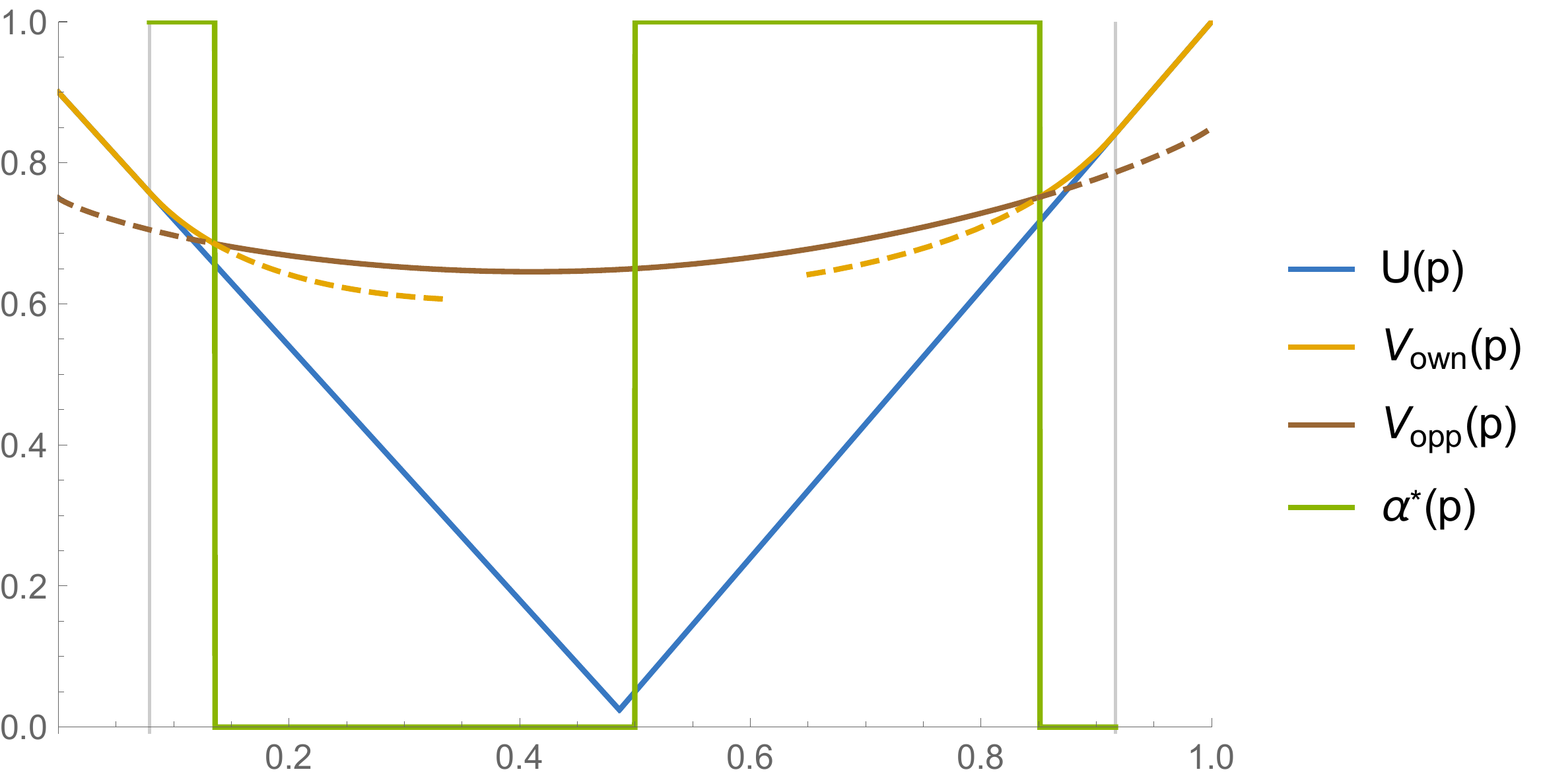}\tabularnewline
\end{tabular}
\caption{Value Function and Optimal Policy.\label{fig:examples_optimal_solution_B}}
\medskip
\begin{footnotesize}
\begin{spacing}{1} 
Note: The value function is the upper envelope of $V_{own}$ and
$V_{opp}$ (solid). ($\lambda=1$, $\rho=0$, $u_{r}^{R}=1$, $u_{\ell}^{L}=.9$,
$u_{\ell}^{R}=u_{r}^{L}=-.9$)
\end{spacing}
\end{footnotesize}
\end{figure}

 The intuition behind the optimal strategy can be explained by a trade-off between speed and accuracy. The opposite-biased strategy leads to complete learning: The DM takes an action only if she receives conclusive evidence  from breakthrough news.
Therefore, she never makes a mistake. By contrast, own-biased learning may  lead to mistakes if the DM's belief drifts to the stopping boundary ($\underline{p}^*$ or $\overline{p}^*$) before observing a breakthrough. In this case the DM stops without learning completely and may choose the wrong action. Therefore, opposite-biased learning has an accuracy advantage. At the same time, full learning under the opposite-biased strategy leads to a potentially long delay, because the DM has to wait for a breakthrough to arrive. By contrast, the delay in the own-biased strategy is limited by the time it takes the belief to reach the stopping boundary. Therefore own-biased learning has a speed advantage.   This explains why the DM never uses the opposite-biased learning except when the cost of learning is sufficiently low, which makes the speed advantage less important. 

 The speed-accuracy tradeoff in the choice between the two strategies also explains why own-biased learning is always optimal near the stopping boundaries. Here, the speed advantage of own-biased learning is particularly large because it takes only a short period of time for the belief to reach the stopping boundary. At the same time, for beliefs near the stopping boundary, there is little uncertainty so that the value of full accuracy is relatively small. The DM therefore prefers speed over accuracy.
For more uncertain beliefs, the value of learning the state is higher, making accuracy more important. At the same time, the speed advantage of own-biased learning is smaller because it takes longer for the belief to reach the stopping boundary. Therefore, the DM prefers the more accurate opposite-biased strategy for more uncertain beliefs.

\paragraph*{Speed and Accuracy of Decisions under the Optimal Strategy.}

We have already noted intuitively that the speed and accuracy of the
DM's decision depends on the mode of learning and her prior belief.
Now we make some formal observations about the optimal strategy in
line with this intuition. 
\begin{prop}[Speed and Accuracy]
\label{prop:speed_and_accuracy}~Suppose that $c<\overline{c}$. 
\begin{enumerate}
\item The average delay in the DM's decision is quasi-concave in the prior
belief. If $c\ge \underline{c}$,
the delay is maximal at $p_{0}=\check{p}$. 
If $c< \underline{c}$, 
 the delay is maximal at some
$p$ inside the opposite-biased learning region. 
\item The probability of a mistake is quasi-convex; it is zero when the initial belief is in the
opposite-biased region.   Within the own-biased region, the probability of a mistake is positive and increasing as $p$ gets closer to a stopping boundary.  
\end{enumerate}
\end{prop}
This proposition shows that initial beliefs matter greatly for
the accuracy of a decision and its timing. A ``skeptic''\textemdash a
DM with uncertain initial belief\textemdash reaches a fully accurate
decision but at the expense of a long delay. By contrast, a ``believer''\textemdash a
DM with a more extreme belief\textemdash sacrifices accuracy in favor
of speed.%
\footnote{Around the boundary beliefs $\overline{p}$
and $\underline{p}$, the outcome in terms of accuracy and speed varies
discontinuously, although the value is continuous.} %
This feature stands
in contrast to rational inattention and drift-diffusion models with
non-shifting stopping boundaries. These latter models predict that
the accuracy of a chosen action is invariant to the DM's initial beliefs,
as long as they are within the experimentation region.\footnote{In a rational inattention model, an optimal  experiment for a DM involves no more signals than the number of actions chosen; otherwise, she can lower her information cost by eliminating the ``wasteful'' signals (see \cite{matejka/mckay:15}). This property means any action chosen after performing an experiment must correspond to a unique posterior, and hence a unique level of accuracy.  In drift-diffusion models, the drift-to-boundary structure of the optimal policy means that the accuracy of each decision is pinned down by a corresponding stopping boundary. In many such models, the stopping boundaries do not change over time, so the accuracy of each action is fixed. See \cite{ratcliff/mckoon:08} for a survey.} 

\subsection{Comparative Statics\label{sec:comparative_statics}}

It is instructive to study how the optimal strategy varies with the
parameters. We start by considering the experimentation region. 

\begin{prop}[Comparative Statics: Boundaries of the Experimentation Region]
\label{prop:Comparative_Statics_Experimentation_Region}~ 
\begin{enumerate}
\item The experimentation region expands as $\rho$ or $c$ falls, and covers
$(0,1)$ in the limit as $(\rho,c)\to(0,0)$.\footnote{We say a region \emph{expands} when a parameter change leads to a
superset of the original region. This includes the case that the region
appears when it was empty before. We say a region \emph{shifts up
(down)} when both boundaries of the region increase (decrease).} 
\item The experimentation region expands as $u_{r}^{L}$ or $u_{\ell}^{R}$
falls (so that ``mistakes'' become more costly), and covers $(0,1)$ in the limit as $(u_{r}^{L},u_{\ell}^{R})\to(-\infty,-\infty)$.%
\footnote{For given $c$, if $u^R_\ell$ and $u^L_r$ are sufficiently small, $\overline{c}>c$ so that the experimentation region is non-empty. Given $\overline{c}>c$, $\underline{p}^{*}$ converges
monotonically to zero as $u_{\ell}^{R}\rightarrow-\infty$, and $\overline{p}^{*}$
converges monotonically to one as $u_{r}^{L}\rightarrow-\infty$.}

\item If $c<\overline{c}$, then the experimentation region shifts down
as $u_{r}^{R}$ increases and up as $u_{\ell}^{L}$ increases. 
\end{enumerate}
\end{prop}
Parts (a) and (b) are quite intuitive. The DM acquires information
for a wider range of beliefs if the cost of learning $(\rho,c)$ falls,
or if mistakes become more costly in the sense that $(u_{r}^{L},u_{\ell}^{R})$
falls. The intuition for (c) is twofold:   an increase in $u_{r}^{R}$ increases the
value of a conclusive $R$-signal  and thus causes $\underline{p}^{*}$ to fall.  Further, if $\rho>0$, the immediate action  $r$ becomes more attractive, which causes $\overline p^*$ to shift down.

Next we consider how the payoffs from mistakes affect the cutoff beliefs
that separate regions in which the DM employs different modes of learning. 
\begin{prop}[Comparative Statics: Mode of Learning]
\label{prop:Comparative_Statics_Modes_of_Learning}~ 
\begin{enumerate}
\item If $\underline{c}<c<\overline{c}$, then the cutoff $\check{p}$ in
 own-biased learning decreases if $u_{\ell}^{R}$ falls, and
increases if $u_{r}^{L}$ falls. 

\item  For given $c$, the opposite-biased region appears ($\underline{c}>c$) for $(u_{r}^{L},u_{\ell}^{R})$ sufficiently small, and expands as $(u_{r}^{L},u_{\ell}^{R})$ falls. Given $\underline{c}>c$, $\underline{p}$ converges monotonically to zero as $u_{\ell}^{R}\rightarrow-\infty$, and $\overline{p}$ converges monotonically to one as $u_{r}^{L}\rightarrow-\infty$. 

\end{enumerate}
\end{prop}
To get an intuition for Part (a), suppose $p<\check{p}$. In this
case, own-biased learning may lead to taking action $\ell$
in the wrong state (state $R$). This mistake becomes more costly as $u_{\ell}^{R}$ becomes smaller and $\check{p}$ shifts down to avoid this mistake. In other words, the DM avoids the $L$-biased source when action $\ell$ becomes more risky.

Part (b) shows that the cost of mistakes also matters for the relative
appeal of the \emph{alternative} learning strategies:   opposite-biased learning  becomes more appealing when mistakes become more costly.
In the limit where mistakes become completely unacceptable, the opposite-biased
strategy becomes optimal for all beliefs. 
One could imagine this limit
behavior as that of a scientist who views conclusive evidence of either
kind\textemdash \textit{proving} or \textit{disproving} a hypothesis\textemdash as
the only acceptable way of advancing science. Such a scientist will
rely solely on opposite-biased learning:  she will initially strive
to prove the hypothesis she conjectures to be true;  
after a series of unsuccessful
attempts to prove the hypothesis, however, she begins to doubt her
initial conjecture, and when the doubt reaches a ``boiling point''
(i.e., $p^{*}$), she begins to put some effort to disproving it as well.\footnote{Own-biased learning could also describe some aspect of scientific
inquiry if a scientist is willing to accept a small margin of error.
For instance, even a careful theorist may not verify thoroughly her
``proof'' if she believes it to be correct. Rather, she may look
for a mistake in her argument,
and without finding one may declare it a correct proof.}

Finally we show how the effect of a higher discount rate compares
 to the effect of a higher flow cost. Intuitively, one would interpret
$\rho$ as a cost of learning and would thus expect that $c$ and
$\rho$ are substitutes in the sense that a higher discount rate requires
a lower flow cost for the same structure to emerge. Formally, one would expect $\partial\overline{c}/\partial\rho<0$
and $\partial\underline{c}/\partial\rho<0$. The following proposition
shows that this is indeed the case if at least one ``mistake payoff'' ($u_{\ell}^{R}$
or $u_{r}^{L}$) is not too small. 
\begin{prop}[Discounting vs. Flow Cost]
\label{prop:role_of_discounting}~ 
\begin{enumerate}
\item Suppose $\overline{c}>0$. Then $\partial\overline{c}/\partial \rho<0$
if and only if $U(p)>0$ for all $p\in[0,1]$. 
\item Suppose $\underline{c}>0$. Then $\partial\underline{c}/\partial \rho<0$
if both $u_{r}^{R}>\left|u_{\ell}^{R}\right|$ and $u_{\ell}^{L}>\left|u_{r}^{L}\right|$;
$\partial\underline{c}/\partial\rho>0$ if $\min\left\{ u_{\ell}^{R},u_{r}^{L}\right\} $
is sufficiently small. 
\end{enumerate}
\end{prop}
If the mistake payoffs $u_{\ell}^{R}$ and $u_{r}^{L}$ are both negative
and sufficiently large in absolute value,  both $\partial\overline{c}/\partial\rho$
and $\partial\underline{c}/\partial\rho$ are positive. Namely, a higher discount rate
calls for more experimentation in this case. This is intuitive since  if losses
are sufficiently large, the DM would prefer to delay their realization, when they are discounted more.
This favors longer experimentation. 

\section{Application: Media Consumption\label{sec:Application}}

Media outlets differ in their partisan biases.\footnote{One method to measure the partisan bias of an outlet is to compare
	the language used by the outlet to the language used by members of
	congress whose partisanship is identified by their voting decisions.
	\citet{gentzkow/shapiro:10} pioneered this method for daily newspapers;
	\citet{martin/yurukoglu:17} use it to identify the bias of cable
	news channels.} There is evidence that the consumption of biased news outlets
affects the political leaning of voters and may change their voting
decisions (see \citet{DellaVigna2007} and \citet{martin/yurukoglu:17} among others).  While a significant fraction of people multi-home and have a news
diet that contains outlets with different biases, consumers tend to
consume news from outlets with a partisan bias similar to their own
position (see for example \citet{Gentzkow2011}). This partisan
(or own-biased) selective exposure can lead to an ``echo-chamber''
effect\textemdash partisan voters become increasingly polarized (see \citet{martin/yurukoglu:17}).  Our model contributes to the literature on media choice by providing theoretical predictions about the optimal news diets for voters with different subjective beliefs, their dynamic evolutions, and the implications for polarization.\footnote{Our rational Bayesian framework to study media choice
shares a theme with \citet{Calvert1985}, \citet{Suen2004}, \citet{burke:08},
\citet{Oliveros2015}, and \citet{yoon:16}\textemdash particularly
the optimality of consuming a biased medium for a Bayesian agent.
These papers are largely static, unlike the current model which is
fully dynamic. \citet{Meyer1991} makes a similar observation in a
dynamic contest environment.  \citet{mullainathan/shleifer:05} assume a ``behavioral''
bias on the part of consumers to predict media slanting.} 

We interpret our DM as a citizen
who votes for one of two candidates, $r$ or $\ell$, possibly after consulting media outlets.\footnote{An alternative interpretation is that the citizen has no decision to make but derives a non-instrumental	payoff from reaching a certain opinion. The citizen subscribes to	media and consults these on an issue of interest. After some time,	she may ``make up her mind'' on the issue and stop acquiring further	information. When making up her mind, she enjoys a payoff that increases in the precision of her belief, for instance $\max\left\{ p_{t},1-p_{t}\right\} $. This interpretation corresponds to our model with	$u_{r}^{R}=u_{\ell}^{L}=1$ and $u_{\ell}^{R}=u_{r}^{L}=0$. 
}
Candidate $r$
has a right-wing platform, and candidate $\ell$ has a left-wing platform.
The state $\omega\in\{R,L\}$ indicates the optimal platform  
for the voter, $u_{x}^{\omega}$ representing the voter's utility
from voting for $x$ in state $\omega$.  The voter incurs flow cost $c>0$ for  paying attention to the media.%
\footnote{Since the time at which the payoff is realized (the election or implementation of a policy) is independent of when the voter makes up her mind, we assume $\rho=0$. Our model does	not capture the effect of a ``deadline,'' which is clearly relevant 	for the election example. However, in line with the two-period example discussed	in Section \ref{sec:Example}, we conjecture that the salient features	of our characterization which are discussed in the context of this application carry over to a model with a sufficiently long deadline.  See Section \ref{sec:Extensions} for further discussion of the effect of a deadline.}
Her belief about the state is captured by $p$. We say the voter
is more right-leaning the higher $p$ is.\footnote{Alternatively one could model a voter's bias in terms of her payoffs. In this case $u_{x}^{\omega}$ could incorporate her ``partisan'' preference.}

Naturally, we interpret the $L$-biased source as a \emph{left-wing outlet} and the $R$-biased source  as a \emph{right-wing outlet}.  For example, a right-wing outlet   publishes information that supports the left-wing candidate only if it passes a high standard of accuracy---that is, if it constitutes conclusive evidence.%
\footnote{The feature that a right-wing outlet could, albeit rarely,  broadcast left-favoring news may appear unusual but is a consequence of our voter being a Bayesian who cannot be systematically misled. This feature is also consistent with empirical evidence. 	\citet{Chiang2011} find that endorsements of presidential candidates by newspapers are only influential if they go against the bias of the newspaper, suggesting that consumers are, to some extent, able to correct for the bias of newspapers, as predicted by the Bayesian model.}
Of course this occurs  rarely.  Most of the time,  the outlet instead reports  $R$-favoring content, which Bayesian consumers perceive as having weak informational content.
In the same vein, we interpret an interior ``attention'' choice  $\alpha\in(0,1)$ as multi-homing by the voter across the two outlets, where $\alpha$ represents the share of time she spends on the left-wing outlet.

 Our interpretation of media bias accords well with the media literature (see \cite{Gentzkow2015} for a survey). A common model in this literature views media bias as arising from the manner in which  an outlet  ``filters'' raw news signals for its viewers.  Section \ref{sec:Micro-foundation} discusses how our information structure can be microfounded by such a model.

Given this interpretation, our theory, more specifically Theorem \ref{thm:optimal_attention_strategy}, provides a rich portrayal of voters' dynamic media choices and their effects.   Voters with extreme beliefs $p\notin\left(\underline{p}^{*},\overline{p}^{*}\right)$
always vote for their favorite candidates without consulting media. Those who consume media
exhibit the following behavior: 
\begin{itemize}
	\item If \emph{news media are moderately informative} so that the cost of information
satisfies $\underline{c}\le c<\overline{c}$, right-leaning voters with $p>\check{p}$ subscribe to  right-wing outlets and
	left-leaning voters subscribe to left-wing outlets. Over time,
	in the absence of breakthrough news  that goes against their initial beliefs, all	voters stick to their initial media choice. 
	\item If \emph{media are highly informative}, so that the cost of information
	is low ($c<\underline{c}$), moderately left-wing ($p\in(\overline{p},\overline{p}^{*})$)
	or moderately right-wing voters ($p\in(\underline{p}^{*},\underline{p})$)
	consume media that are biased against their beliefs, whereas partisan voters
	with more extreme beliefs consume media that are biased in favor
	of their beliefs.  Over time, absent breakthrough news in favor of their initial beliefs, moderate voters	($p\in\left(\underline{p},\overline{p}\right)$) become increasingly
undecided and when their beliefs reach $p^{*}$, they  multi-home 	and divide their attention between both types of outlets; whereas partisan voters become more extreme and continue to subscribe to own-based media.
\end{itemize}

The choice of opposite-biased media by moderate voters may seem counterintuitive. Consider for example a moderate right-wing voter. This voter initially tries to find right-favoring evidence which she expects more likely to arise given their belief.
Interestingly, she expects to find such
evidence in left-wing outlets  as they  scrutinize right-favoring
information more and apply a very high standard for reporting such information. 

In order to understand how consumers' media choices interact with their beliefs, it is useful to perform a simple thought experiment.  Suppose there is a unit mass of  consumers with identical costs of information acquisition $c$; whose payoffs $u^\omega_x$ are identical and symmetric (so that $\hat p=1/2$); and  whose prior beliefs $p_0$ are distributed according to some distribution function $F$ which is symmetric around $1/2$ (i.e., $F(p)=1-F(1-p)$).   Now, fix the state, say $\omega=L$, and study how the distribution of consumers' beliefs change over time due to their media choice. Of particular interest is the extent to which theirs beliefs become more or less polarized over time.  While one can use a number of different measures of polarization,\footnote{See \cite{Esteban2012} for a survey of polarization measures.} we simply focus on the difference between the median belief for consumers with $p\ge 1/2$ and the median belief for consumers with $p\le 1/2$.

\begin{prop} \label{prop:polarization} Fix the true state to be $\omega=L$ and assume symmetric payoffs $u^L_\ell=u^R_r$ and $u^R_\ell=u^L_r$.\footnote{Symmetric results hold if the true state is $\omega = R$.} 
	\begin{enumerate}
		\item  The beliefs in the subpopulation of voters consuming own-biased outlets at time $t=0$ become more polarized over time and converge to a distribution containing three mass points, $\{0, \underline p^*, \overline p^*\}$.
		
		\item If $c\in (\underline c,\overline c)$, the beliefs of all voters become more polarized over time.
		
		\item The  beliefs of the voters consuming opposite-biased outlets at time $t=0$ converge to the true state in the limit as $t \rightarrow \infty$.
	\end{enumerate}
\end{prop}

Figure \ref{fig:media_choice_evolution} shows snapshots of the evolution of beliefs taken at three different times, assuming that initial distribution $F$ is uniform on $[0,1]$.  The colors represent the media choice by voters who are still subscribing  to media. 
\begin{figure}[h]
\begin{center}
\includegraphics[width=1\textwidth,height=4.5cm]{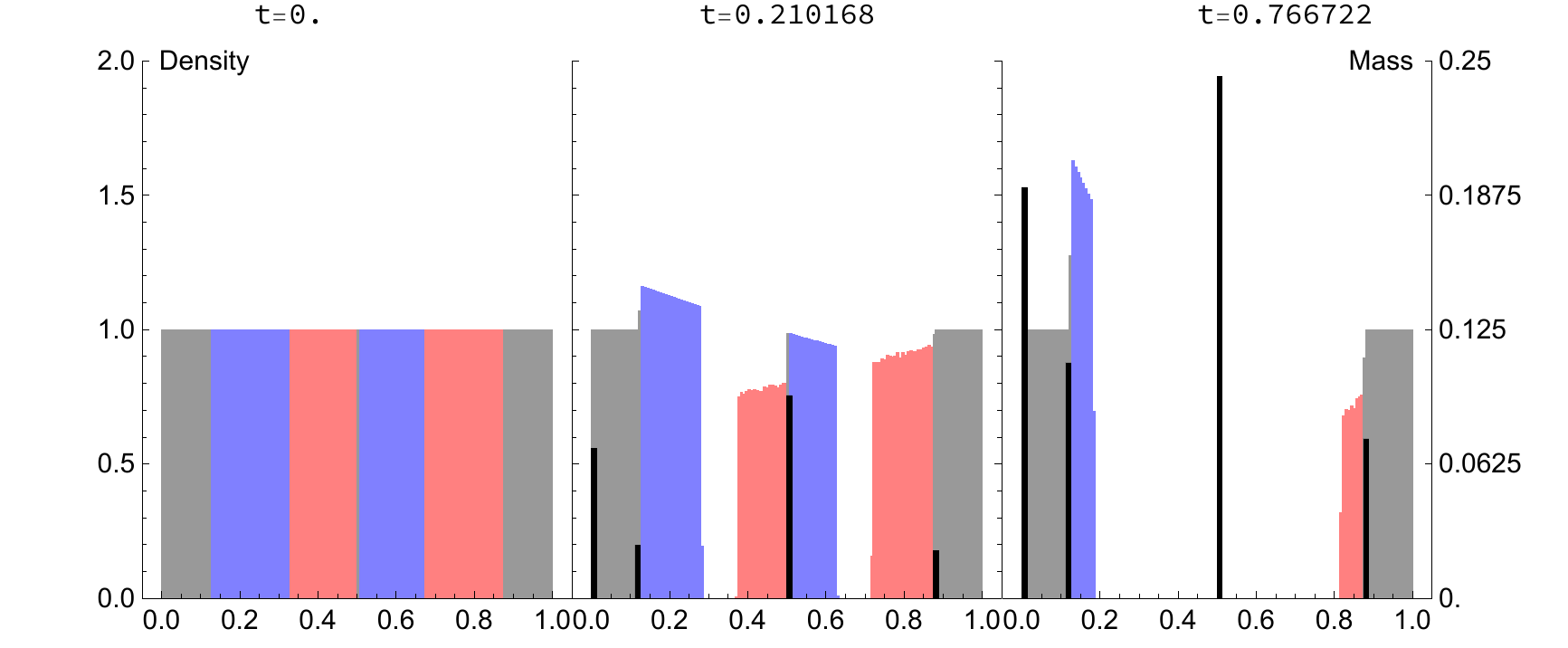}
\end{center}
\vspace{-6pt}
\caption{\label{fig:media_choice_evolution} Evolution of media choice and beliefs
		when the true state is $L$.}
\medskip
\begin{footnotesize}
\begin{spacing}{1} 
Note: Shaded areas represent the density of
			beliefs (left axis). Colors indicate the choice of media with red indicating
			right-wing/conservative and blue indicating left-wing/liberal. Bold
			bars represent mass points of beliefs (right axis).
\end{spacing}
\end{footnotesize}
\end{figure}

The figure shows that those with extreme beliefs consume own-biased outlets.  
Over time, in the absence of contradictory breakthrough news,
these outlets feed such voters with what they believe in, leading them to become more extreme.  Proposition \ref{prop:polarization}.(a) and (b) states that the beliefs of these voters become more polarized over time. Consequently, our model generates self-reinforcing beliefs\textemdash sometimes called an ``echo-chamber'' effect\textemdash that persist until
strong contradictory evidence arrives.  

The evolution of beliefs is quite different for the voters with moderate beliefs. The figure shows that they consume opposite-biased outlets. Absent breakthrough news this leads to an ``anti-echo chamber'' effect. 
Over time, the anti-echo chamber effect makes voters more undecided. Ultimately, they multi-home both left-wing and right-wing outlets, and in the limit, learn the true state, as stated in Proposition \ref{prop:polarization}.(c).\footnote{\cite{Gentzkow2011} present evidence that a significant number of consumers multi-home news channels with different slants.    To the best of our knowledge, ours is the first theoretical model to predict the multi-homing of news outlets with conflicting slants.}

In summary,
our dynamic model of media choice predicts two different dynamics
of belief evolution resulting from optimal media choice: the beliefs
of those who are sufficiently extreme become more polarized, and
the beliefs of those who are sufficiently moderate converge toward
the middle and result in the multi-homing of opposite media outlets. The ``anti-echo chamber'' effect is a novel prediction of our dynamic model and has no analogue in previous literature.

\section{Extensions} \label{sec:Extensions}

In this section, we provide a foundation for our model and link it to a model of filtering bias; and discuss several interesting extensions that correspond to various realistic features of information acquisition. The extensions suggest that our characterization of the optimal policy as well as the proof techniques are surprisingly robust.  While the discussion is kept deliberately informal and intuitive, more detailed arguments can be found in Appendix \ref{app:extensions} in the Supplemental Material.

\subsection{Discrete-Time Foundation for Conclusive Poisson Model}\label{sec:Micro-foundation}

Although we have only considered \emph{conclusive} Poisson experiments, we show here that these experiments are justified as optimal within a more general class of experiments.

Consider a discrete time analogue of our model with an arbitrary period length $dt\in (0, 1/\lambda)$. The DM's problem is the same as before, except that
she incurs a cost of $cdt$ and discounts by the factor of $e^{-\rho dt}$
for each period of information acquisition. In each period, the DM may choose an experiment of the form described in Table \ref{fig:binary}.
\begin{table}[htb]
  \centering

		\begin{tabular}{ccc}
		\tabularnewline
			\hline 
			\hline 
			state/signal  & $L$-signal  & $R$-signal \tabularnewline
			\hline 
			$L$  & $a$  & $1-a$\tabularnewline
			$R$  & $1-b$  & $b$ \tabularnewline
			\hline 
			\hline 
			\multicolumn{3}{l}{{\small{}{constraints: $a,b\in[0,1],1\le a+b\le1+\lambda dt$ }}}
		\end{tabular}
	\caption{General binary-signal experiment}
	\label{fig:binary} 
\end{table}

The total probability $a+b$ of ``informative'' signals is bounded
above by $1+\lambda dt$.
Note that the overall informativeness of
the experiment, measured by $\lambda dt$, is proportional to the length of a period, and vanishes as $dt\to0$. This captures the idea that
real information takes time to arrive. In the limit as $dt\to0$,
$\lambda$ parameterizes the constraint for ``flow'' information.

General binary-signal experiments in discrete time encompass rich
and flexible information structures. Setting $(a,b)=(1,\lambda dt)$ or $(a,b)=(\lambda dt,1)$, we obtain the experiments in Table \ref{fig:two-biased-experiments} that converge to our conclusive Poisson information structure as $dt\to 0$.
More generally, if we set $a=\gamma dt$ and $1-b=(\gamma-\lambda)dt$, for $\gamma> \lambda$, we obtain an inconclusive Poisson experiment in which breakthrough news arrives in both states but at a higher rate in state $L$ than in state $R$.  In this way, \emph{any} posterior belief  $\phi<p$  can be obtained from breakthrough news in the limit as $dt\to 0$, and a similar construction yields jumps to $\phi>p$. If we pick a posterior closer to the prior, breakthrough news arrives with a higher rate. This captures the intuitive idea that a less informative signal can be obtained more easily.\footnote{To see this, fix any posterior $\phi$ below the prior $p$. Consider the experiment with $a=\frac{p(1-\phi)}{p-\phi}\lambda dt$ and $b=1-\frac{(1-p)\phi}{p-\phi}\lambda dt$.  In the limit as $dt\to 0$, the experiment converges to a Poisson process which sends an $L$-breakthrough signal at rate $\frac{p(1-\phi)}{p-\phi}\lambda$ in state $L$, and at rate $\frac{(1-p)\phi}{p-\phi}\lambda$ in state $R$, so that upon receiving that signal the belief becomes exactly $\phi$. The arrival rate increases and converges to $\infty$ as $\phi\nearrow p$.}
Further, a ``mixing'' of Poisson processes can be attained by switching across different experiments within the class. For example, dividing attention with $\alpha=1/2$ in the baseline model, is obtained by switching  between $(a,b)=(1,\lambda dt)$ and $(a,b)=(\lambda dt,1)$ over time.%
\footnote{As usual, in the absence of news, the belief drifts in the direction implied by Bayes rule. One example of this is the stationary policy $\alpha(p)=1/2$, in which no updating occurs in the absence of breakthrough news; this is obtained when two Poisson processes with jumps to zero and one are mixed equally.}

In summary, this class encompasses a range of both conclusive and inconclusive experiments.%
\footnote{In discrete time, the class of experiments also admits a random walk, e.g.~if $a=b=(1+\lambda dt)/2$. However, this process becomes uninformative in the limit as $dt\rightarrow0$. It converges to a diffusion process with identical drift in both states. In other words, an informative DDM cannot be obtained as a limit of the current class.} %
Suppose the DM is free to choose from this rich class of experiments.  Which experiments are optimal?  Will she necessarily choose an accurate signal?   The answer is not obvious, since the DM may find inaccurate signals appealing as they are easier to obtain. Nevertheless, we show that  conclusive experiments are optimal, thus justifying  our focus on them within this class of experiments.

\begin{prop}
\label{prop:microfound} Consider the discrete-time problem with arbitrary
period length $dt\in(0,1/\lambda)$ and finite or infinite number of periods. In each period and for each belief, any binary experiment with $a+b\le1+\lambda dt$ is weakly dominated by either $\sigma^{R}$ or $\sigma^{L}$. 
\end{prop}
Importantly, this proposition does not claim a Blackwell dominance relation. We
show that at each history, one of the experiments $\sigma^{R}$ or
$\sigma^{L}$ is optimal, but which one depends on the current belief
and the continuation payoffs. This result follows from
the convexity
of the DM's continuation payoff in her beliefs, which holds because the payoff of any fixed strategy is linear in the prior belief.\footnote{\label{fn:zhong}\citet{Zhong2017} demonstrates the optimality of a (non-conclusive) Poisson experiment when the DM incurs posterior separable cost that depends on the experiments as well as the current belief. While similar in spirit, our
result is not an implication of his result. We adopt a class of feasible Blackwell experiments that is independent of the DM's belief. Our constraint $a+b<1+\lambda dt$ cannot be derived from a constraint on a posterior
separable information cost function (details are available from the authors on request).  Further, we prove optimality of conclusive experiments, which is not shown in his paper.}

\paragraph{Filtering Interpretation.}  The class of experiments featured in Table \ref{fig:binary} can be motivated via a ``filtering'' model commonly adopted  in the media choice literature. According to this model, media ``bias'' or ``slant''  arises when an outlet filters rich raw information into coarse ``messages'' for the consumers (see \citet{Gentzkow2015} and \citet{Prat2013}).
Suppose that an outlet observes a signal $s\in\mathbb{R}$ where
$s$ is drawn uniformly from $[0,1]$ in state $L$ and uniformly from
$[\lambda dt,1+\lambda dt]$ in state $R$. The outlet
must \textquotedblleft filter\textquotedblright{} the signal into
coarse messages. This could reflect, for instance, limited publishing space or broadcasting time, or a limited capacity of consumers to process rich information.
Suppose, the outlet sends binary messages to the consumer as depicted in Figure \ref{fig:Filtering-Bias}. It sends an $L$-message if
$s<\hat{s}$, and an $R$-message if $s>\hat{s}$, for a threshold $\hat{s}$
chosen by the media outlet. The threshold $\hat{s}$  characterizes
the outlet\textquoteright s \textquotedblleft ideological\textquotedblright{}
orientation; the lower $\hat{s}$, the more $R$-biased it is.  Each filtering-threshold $\hat s$ induces an experiment of the form described in Table \ref{fig:binary} on page \pageref{fig:binary}, with $a=\hat{s}$ and $b=1+\lambda dt-\hat{s}$.
The ``left-wing'' and ``right-wing'' outlets in Section \ref{sec:Application} correspond to cutoffs  $\hat{s}=1$ and  $\hat{s}=\lambda dt$, respectively. 
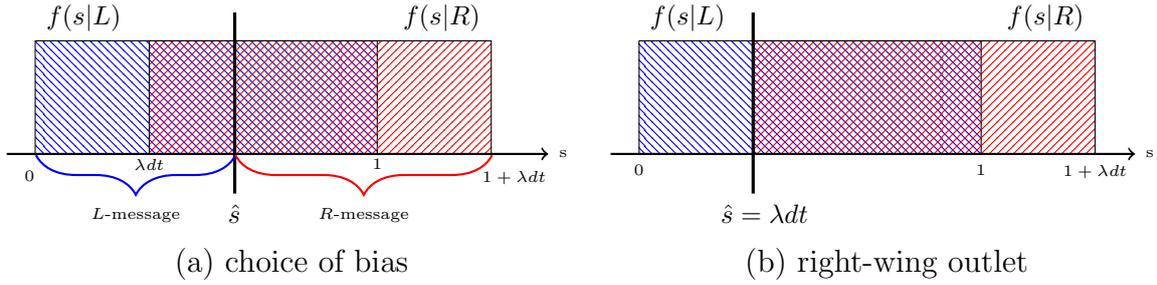
\begin{figure}
	\centering
	\begin{tabular}{cc}
		\begin{tikzpicture}[
		scale=0.75,
		1p line/.style={thick, blue},
		2p line/.style={thick, darkred},
		reference/.style = {dashed, thick},
		axis/.style={very thick},
		move/.style={red, thin ->, dashed, shorten <=2pt, shorten >=2pt}
		] 	
		\draw[pattern=north west lines, pattern color=blue] (0,-6.5) rectangle (6,-4.5); 	
		\draw[pattern=north east lines, pattern color=red] (2,-6.5) rectangle (8,-4.5); 
		
		\node[left] at (1.7, -4.1) {\footnotesize{$f(s|L)$}};
		\node[left] at (8, -4.1) {\footnotesize{$f(s|R)$}};
		
		\draw[->, thick] (-0.5, -6.5) -- (9, -6.5);
		\node[right] at (9, -6.5) {\tiny s};
		\node[below] at (-0.1, -6.6) {\tiny {0}};
		\node[below] at (8.4, -6.6) {\tiny {$1+\lambda dt$}};
		\node[below] at (6, -6.4) {\tiny  1};
		\node[below] at (2, -6.4) {\tiny $\lambda dt$}; 
		
		\draw[-, very thick] (3.5, -7.2) -- (3.5, -4);
		\node[below] at (3.49, -7.2) {\footnotesize{$\hat s$}};
		\draw [thick, blue,decorate,decoration={brace,amplitude=15pt,mirror},xshift=0.4pt,yshift=-0.6pt](0,-6.5) -- (3.5,-6.5) node[black,midway,yshift=-0.8cm] {\tiny $L$-message};
		\draw [thick, red,decorate,decoration={brace,amplitude=15pt,mirror},xshift=0.4pt,yshift=-0.6pt](3.5,-6.5) -- (8,-6.5) node[black,midway,yshift=-0.8cm] {\tiny $R$-message};
		
		\end{tikzpicture} & 
		\begin{tikzpicture}[
		scale=0.75,
		1p line/.style={thick, blue},
		2p line/.style={thick, darkred},
		reference/.style = {dashed, thick},
		axis/.style={very thick},
		move/.style={red, thin ->, dashed, shorten <=2pt, shorten >=2pt}
		]
		\draw[pattern=north west lines, pattern color=blue] (0,-6.5) rectangle (6,-4.5); 	
		\draw[pattern=north east lines, pattern color=red] (2,-6.5) rectangle (8,-4.5); 
		
		\node[left] at (1.7, -4.1) {\footnotesize{$f(s|L)$}};
		\node[left] at (8, -4.1) {\footnotesize{$f(s|R)$}};
		
		\draw[->, thick] (-0.5, -6.5) -- (8.7, -6.5);
		\node[right] at (8.7, -6.5) {\tiny s};
		\node[below] at (0, -6.5) {\tiny  0};
		\node[below] at (8, -6.5) {\tiny $1+\lambda dt$};
		\node[below] at (6, -6.5) {\tiny  1};
		
		\draw[-, very thick] (2, -7.2) -- (2, -4);
		\node[below] at (2.2, -7.2) { \footnotesize{$\hat s=\lambda dt$}};
		
		\end{tikzpicture}
		\tabularnewline
		(a) choice of bias & (b) right-wing outlet\tabularnewline
	\end{tabular}
	\caption{Filtering Bias\label{fig:Filtering-Bias}}
\end{figure}
Compared with these outlets,  the outlets choosing cutoffs $\hat s\in (\lambda dt, 1)$ can be interpreted as more moderate. In Section \ref{sec:Application}, our results hold unchanged if we expand
the set of media outlets to contain all these moderate outlets since Proposition \ref{prop:microfound} shows that, even facing such a rich choice, consumers will still choose from two extreme media outlets.

\subsection{Non-Exclusivity of Attention}\label{subsec:non-exclusive_attention}
Our model does not allow for accidental discovery of evidence; i.e., a DM  never receives evidence that she is not looking for. It is plausible, however, that an individual who looks for $R$-evidence may accidentally find the opposite and become convinced that the state is $L$.  For example, a prosecutor seeking evidence that a suspect is guilty, may stumble upon evidence to the contrary.

This possibility can be easily accommodated within our model by assuming that the DM is limited to an interior attention choice $\alpha\in[\underline{\alpha},\overline{\alpha}]$, where $0<\underline{\alpha}<1/2<\overline{\alpha}<1$. Consequently she will always be exposed to evidence of both types, and may find one type of evidence while looking only for the other.  If we set $\underline{\alpha}=1-\overline{\alpha}$, Theorem \ref{thm:optimal_attention_strategy}
as well as Propositions \ref{prop:Comparative_Statics_Experimentation_Region}, \ref{prop:Comparative_Statics_Modes_of_Learning},
and \ref{prop:role_of_discounting} remain qualitatively unchanged. Of course, the precise cutoffs that characterize the optimal policy change:  The experimentation region
shrinks and $\overline{c}$ becomes smaller as $\overline{\alpha}$ decreases.
The intuition is that a restriction on the feasible information choices
reduces the value of information acquisition.  More interestingly, opposite-biased learning is part of   the optimal policy for a
larger range of cost parameters---the
cutoff $\underline{c}$ increases   as  $\overline{\alpha}$ falls.  Opposite-biased learning is less affected by the restriction since it calls the DM to ultimately divide attention once $p^*$ is reached. Hence, at $p^*$ the constraint on $\alpha$ is not binding, so that the value of opposite-biased learning is less sensitive to the restriction
on the feasible choices for $\alpha$ than the value of own-biased learning.

\subsection{Asymmetric Returns to Attention\label{subsec:Asymmetric-Returns}}
We have also assumed that both states are equally easy to prove. The arrival rate for breakthrough news for each type of evidence is the same. We can easily relax this feature by introducing two different arrival rates, $ \overline{\lambda}^{R}$
for $R$-evidence and $ \overline{\lambda}^{L}$ for $L$-evidence.
For a given attention choice $(\alpha,\beta)$, this means that in
state $R$ evidence arrives at rate $\alpha \overline{\lambda}^{R}$,
and in state $L$ evidence arrives at rate $\beta \overline{\lambda}^{L}$.
To fix ideas, suppose $ \overline{\lambda}^{R}> \overline{\lambda}^{L}$ so that state $R$ is easier to prove.
If  $ \overline{\lambda}^{R}- \overline{\lambda}^{L}$
is small, our characterization of the optimal policy   in Theorem \ref{thm:optimal_attention_strategy} carries over to this case:  For low levels
of the cost $c$, the optimal policy combines own-biased and opposite-biased learning, and for moderate costs only own-biased learning is optimal.  If $ \overline{\lambda}^{R}- \overline{\lambda}^{L}$ is large, the structure changes as illustrated in Figure \ref{fig:asymmetric_lamda}.

\begin{figure}[hbt]
	\begin{tabular}{cc}
		\includegraphics[height=0.18\textheight]{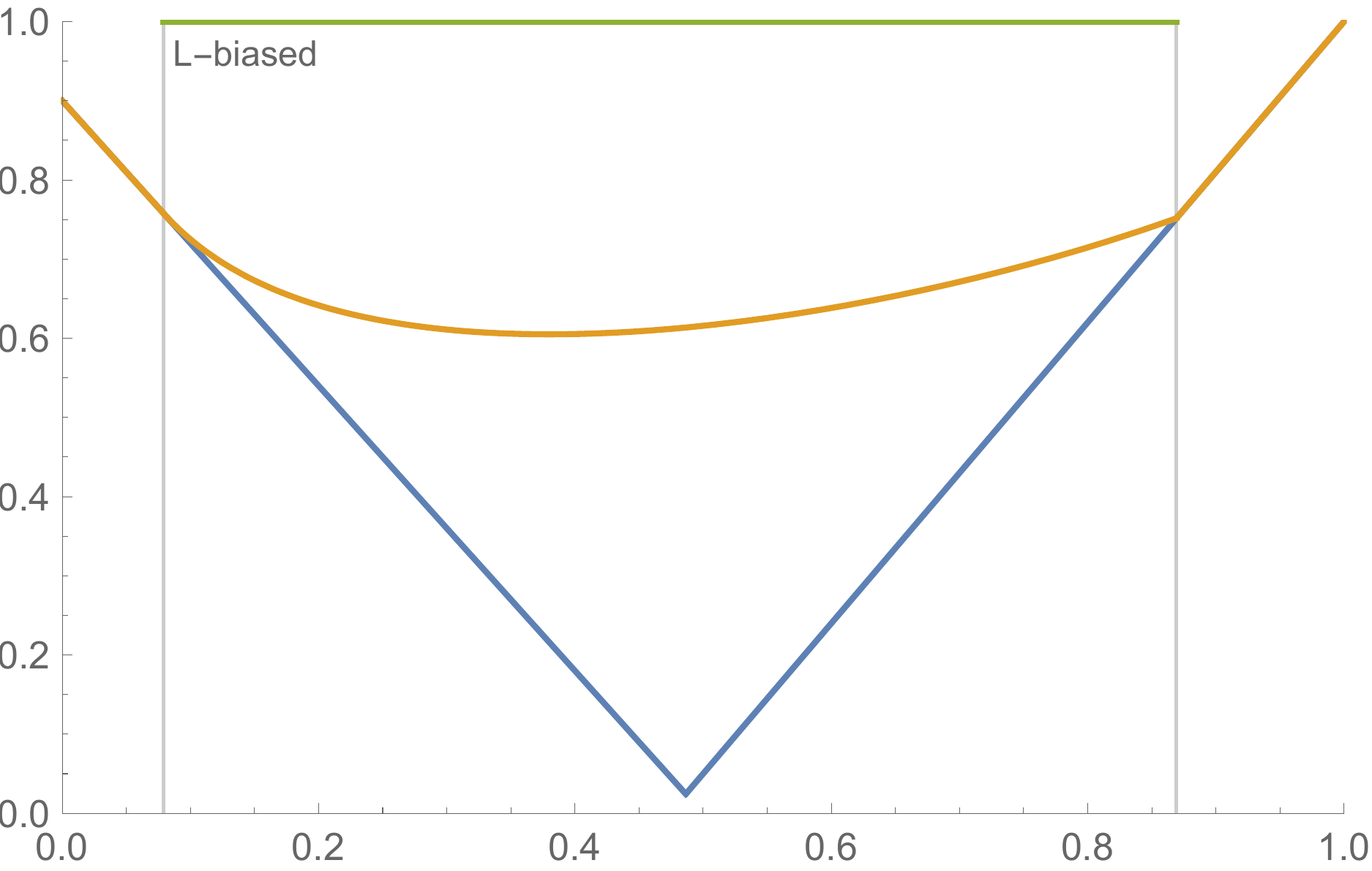}  & \includegraphics[bb=0bp 0bp 545bp 345bp,height=0.18\textheight,clip]{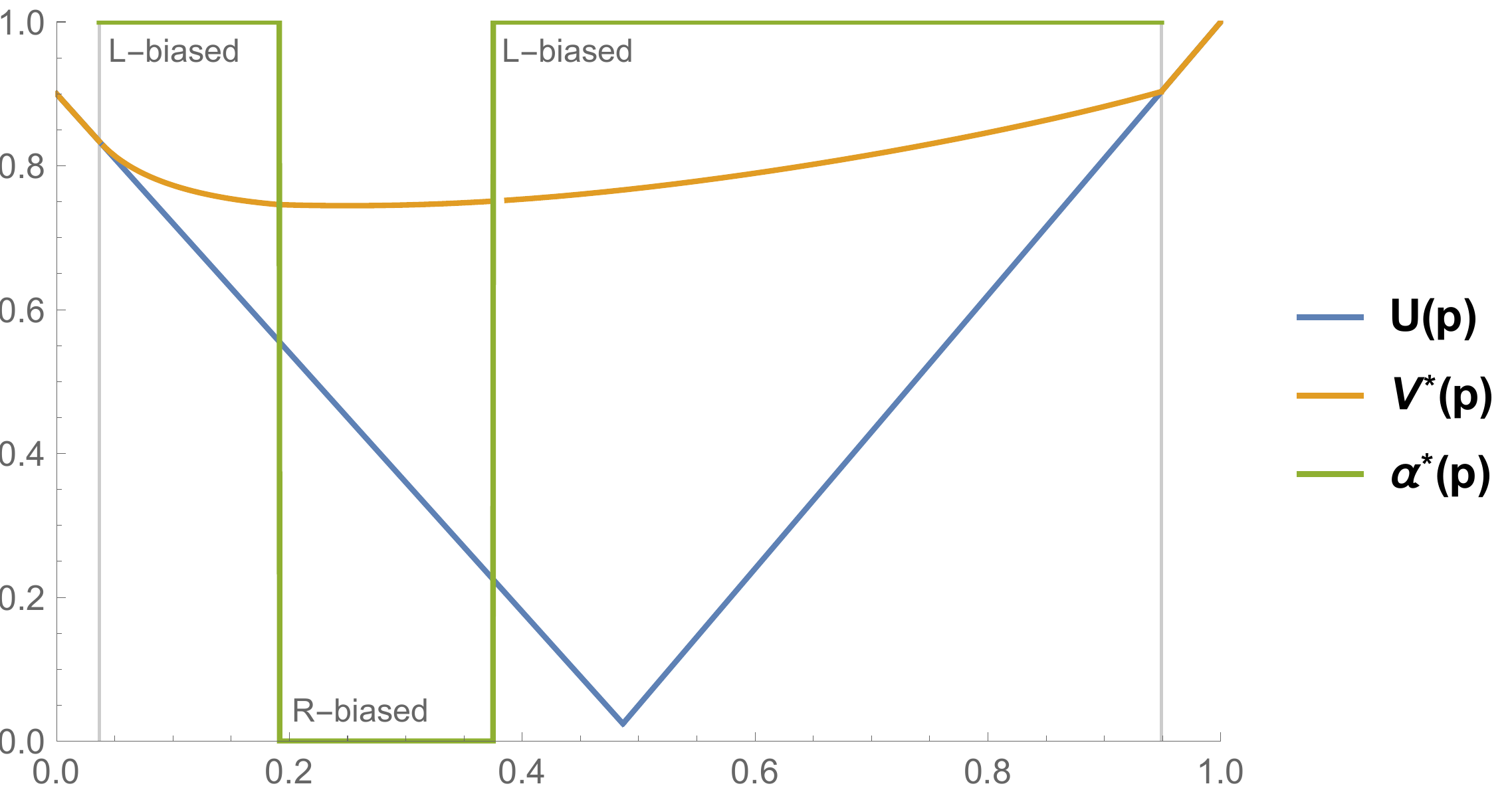}
\includegraphics[bb=560bp 0bp 675bp 345bp,height=0.18\textheight,clip]{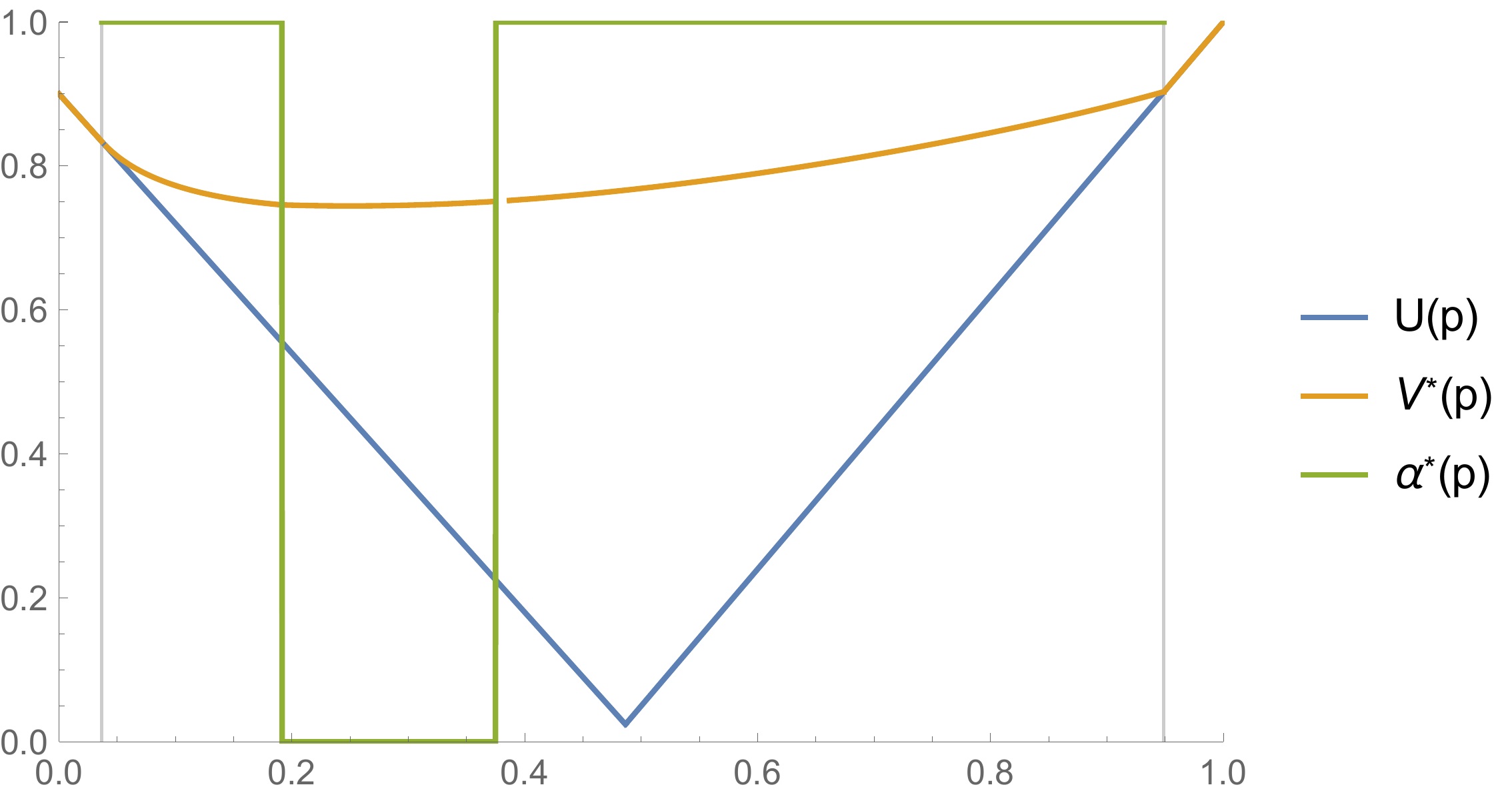}
		 \tabularnewline
		{\footnotesize{}(a) $c=.15$ (moderate cost)}  & {\footnotesize{}(b) $c=.07$ (low cost)}\tabularnewline
	\end{tabular}
\caption{\label{fig:asymmetric_lamda} Value Function and Optimal Policy with
	asymmetric returns to attention.}
\medskip
\begin{footnotesize}
\begin{spacing}{1} 
Note: $ \overline{\lambda}^{R}=1$, $ \overline{\lambda}^{L}=.6$, $\rho=0$,
	$u_{r}^{R}=u_{\ell}^{L}=1$, $u_{r}^{R}=1$, $u_{\ell}^{L}=.9$,
$u_{\ell}^{R}=u_{r}^{L}=-.9$
\end{spacing}
\end{footnotesize}
\end{figure}
In the case of moderate costs  (Panel (a)), the DM  never looks for $L$-evidence, meaning $R$-biased learning is not part of the optimal policy. In the case of low costs (Panel (b)),
the opposite-biased learning strategy appears, but it is skewed toward $L$-biased learning (or $R$-evidence seeking), and  the absorbing state $p^*$  is less than $1/2$ even when the payoffs are symmetric. As in Panel (a), for high beliefs near the stopping region, $R$-biased learning is not optimal, in contrast to the characterization in Theorem  \ref{thm:optimal_attention_strategy}.

\subsection{Diminishing Returns to Attention\label{subsec:Non-linear_model}}
 
 In our model, the DM never splits her attention or multi-homes media outlets, except at the absorbing belief $p^*$.  This feature is a consequence of the ``linear'' attention technology assumed in our model. 
The arrival rate of an $R$-breakthrough is $\lambda \alpha$ and the arrival rate of an $L$-breakthrough is $\lambda (1-\alpha)$. 
This means that the marginal return to  attention to a single news source is constant.
In practice, however, a diminishing marginal return may be realistic in some contexts; namely, one may learn more efficiently from diverse news sources than from just one.  For instance, one may obtain more information by reading the front pages of multiple newspapers, than by reading a single newspaper from front to back. In the Internet era, multi-homing is facilitated by \emph{news aggregators} such as Google News or Yahoo News that curate diverse news sources or perspectives that complement one another.  One can learn more efficiently from such aggregators than by focusing on a single news source.  

Diminishing marginal returns to attention can be incorporated in our model by assuming that the arrival rate of $R$-evidence is given by $\lambda g(\alpha)$, and the arrival rate of $L$-evidence is given by $\lambda g(1-\alpha)$,  where $g$ is an increasing function that satisfies  $g(0)=0$, $g(1)=1$.   Our baseline model corresponds to $g(x)=x$ and diminishing returns obtain if $g(x)$ is strictly concave. Panel (a) of Figure \ref{fig:Gamma} depicts these two cases. 

\begin{figure}[h]
\begin{tabular}{ccc}
	\hspace{0.03\textwidth}	\includegraphics[height=0.17\textheight]{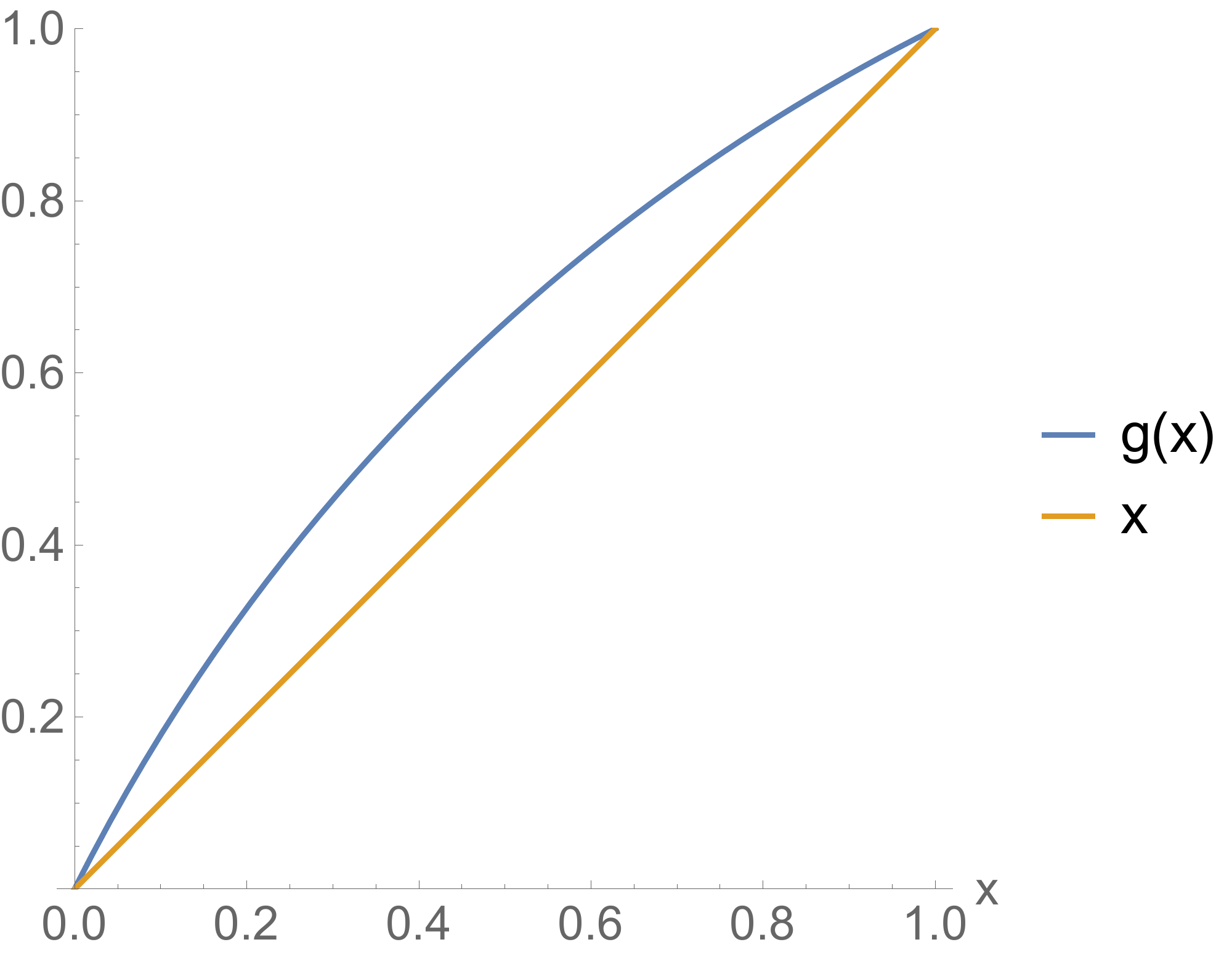}  & \hspace{0.05\textwidth} &
		\includegraphics[height=0.17\textheight]{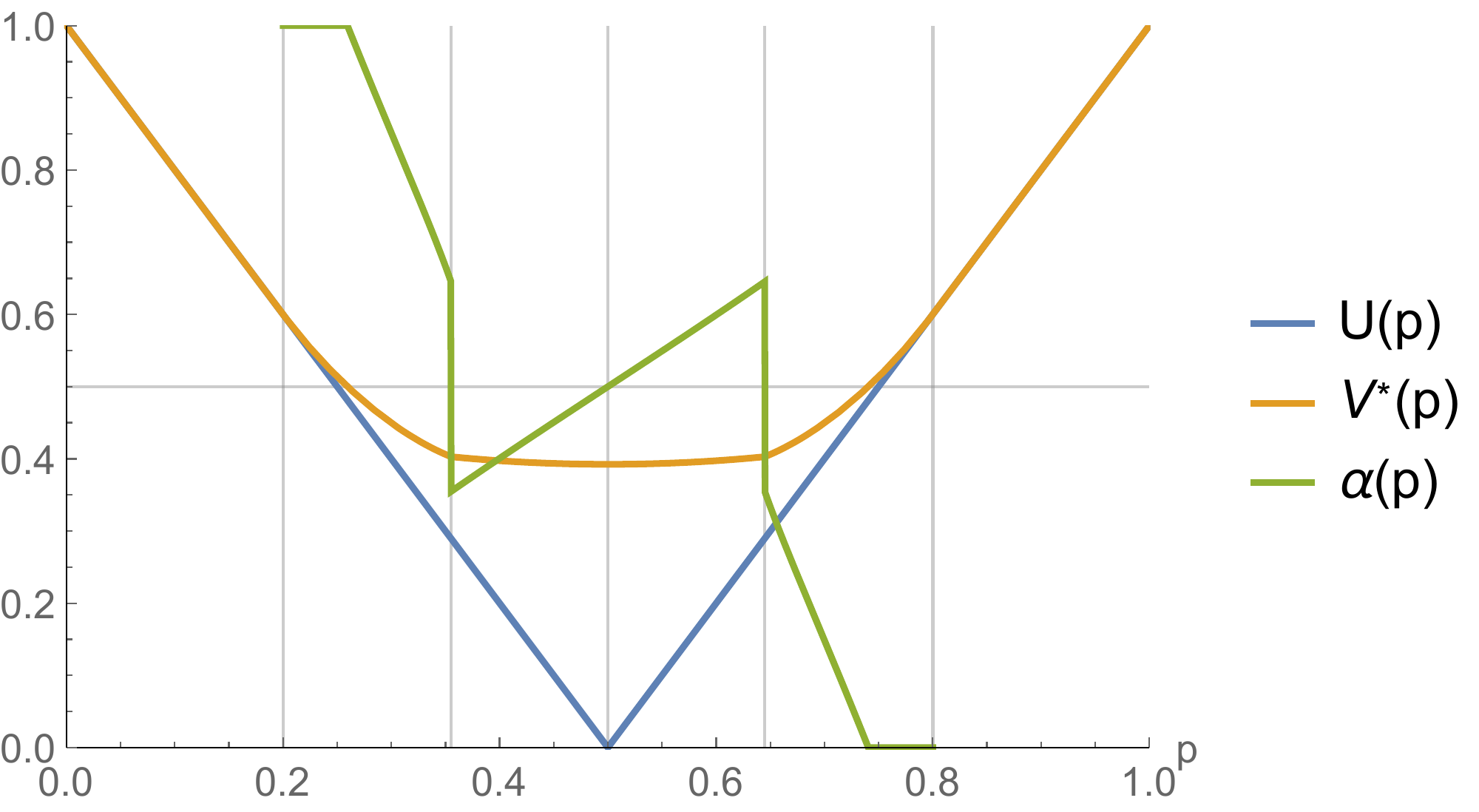}
		\\
		{\footnotesize{}(a) constant vs.~diminishing returns}  && {\footnotesize{}(b) optimal media consumption}
	\end{tabular}
\caption{\label{fig:Gamma}Diminishing returns to attention.}
\medskip
\begin{footnotesize}
\begin{spacing}{1} 
Note: $\lambda=1$, $c=.4$, $\rho=0$, $u^L_\ell=u^R_r=1$, $u^L_r=u^R_\ell=0$, $g(x)=\sqrt{1+4x-x^2}-1$.
\end{spacing}
\end{footnotesize}
\end{figure}

When attention has diminishing marginal returns,  multi-homing, or interior choices of $\alpha$, become optimal for a wide range of beliefs as depicted in panel (b) of Figure \ref{fig:Gamma}. Nevertheless, the basic structure of optimal policy resembles that of Theorem  \ref{thm:optimal_attention_strategy}, if we call the attention choice $L$-biased if $\alpha >1/2$, so that absent breakthrough news the belief  drifts toward $L$, and call the attention choice $R$-biased if $\alpha <1/2$. Specifically, the optimal policy is again characterized by own-biased and opposite-biased learning strategies. The former is optimal for extreme beliefs and the latter may be optimal for moderate beliefs.  

The resulting characterization yields richer implications for the interplay between information choice and beliefs in the context of the media choice. For voters with extreme beliefs, the echo-chamber effect is reinforced. Not only do the beliefs evolve over time due to a biased news diet, the media choice itself evolves. Over time absent breakthrough news, partisan voters' beliefs become more extreme, and this in turn leads to a more biased new-diet. For example, a right-leaning voter consumes more and more right-biased news by decreasing $\alpha$ as her belief moves more to the right over time (see Panel (b) of Figure \ref{fig:Gamma}).
A more formal result is derived for the case of symmetric payoffs in Appendix \ref{sec:General_Gamma} in the Supplemental Material.

\subsection{A Deadline for Decision Making}

What happens if the DM faces a firm deadline for her decision?  Our media application raises this issue since a voter must stop deliberating on the election date.  While a general analysis incorporating a deadline is beyond the scope of the current paper, our two-period example discussed in Section \ref{sec:Example}  sheds some light both on the robustness of our characterization as well as the effect of a deadline.  

First, both modes of learning identified in Theorem \ref{thm:optimal_attention_strategy} can arise in the first period. In the example in Section \ref{sec:Example}, a combination of own-biased learning and opposite-biased learning is optimal. %
%
For other parameter values in the example, own-biased learning is optimal for all beliefs. %
Second, we find a clear deadline effect. In the second period, own-biased learning is always optimal.
We conjecture that this pattern will hold more generally if one were to introduce a firm deadline in our continuous time model---namely, the DM will shift her attention increasingly toward own-biased news sources as the deadline approaches. This is indeed the pattern found by \cite{Stroud2008} from her analysis of the 2004 National Annenberg Election Survey.  She finds that selective exposure and partisan consumption of media outlets intensifies as the election date approaches (see Figures 1 and 2 therein).

\subsection{Non-Binary States and  Actions}

Our model can be easily extended to include more than two actions. Suppose there is a third action $m$ with payoffs $u^R_m\in(u^R_\ell,u^R_r)$ and $u^L_m\in(u^L_r,u^L_\ell)$. Then we can define a new learning heuristic, called ``$m$-strategy,'' that takes the following form:
\[
\underset{p=}{\phantom{|}}\underset{0}{|}\underbrace{\longrightarrow\longrightarrow\vphantom{p^{*}}\negthickspace\negmedspace\longrightarrow\longrightarrow}_{\alpha=0}\underline{p}_m\underbrace{\text{\ensuremath{\vphantom{p^{*}}}---------------------}}_{\text{immediate action }m}\overline{p}_m\underbrace{\longleftarrow\vphantom{p^{*}}\negthickspace\negmedspace\longleftarrow\longleftarrow\longleftarrow}_{\alpha=1}\underset{1}{|}
\]
The two cutoffs $\underline{p}_m$ and $\overline{p}_m$ are chosen optimally. The optimal policy  is a combination of own-biased learning, opposite-biased learning, and the $m$-strategy. It can take various forms. For example if $c\in(\underline{c},\overline{c})$, and the value of the $m$-strategy is higher than the value of opposite-biased learning, the optimal policy takes the following form:
\[
\underset{p=}{\phantom{|}}\underset{0}{|}\underbrace{\text{\ensuremath{\vphantom{p^{*}}}------------}}_{\text{action }\ell}\,\underline{p}^{*}\underbrace{\longleftarrow\vphantom{p^{*}}\negthickspace\negmedspace\longleftarrow}_{\alpha=1} \underline{p}\underbrace{\longrightarrow\longrightarrow\vphantom{p^{*}}\negthickspace}_{\alpha=0}\underline{p}_m\underbrace{\text{\ensuremath{\vphantom{p^{*}}}------------}}_{\text{action }m}\overline{p}_m\underbrace{\longleftarrow\vphantom{p^{*}}\negthickspace\negmedspace\longleftarrow}_{\alpha=1}\overline{p}\underbrace{\vphantom{p^{*}}\negthickspace\longrightarrow\longrightarrow}_{\alpha=0}\bar{p}^{*}\,\underbrace{\text{\ensuremath{\vphantom{p^{*}}}------------}}_{\text{action }r}\underset{1}{|}
\]
Along these lines, a finite number of actions can be added (see Appendix \ref{sec:Third-Action} in the Supplemental Material).

Extending the model to more than two states raises several issues. First, the state-space of the DM's problem becomes multi-dimensional. Second, it is natural to allow for a larger set of news sources if there are more than two states. Within our Poisson framework, with two states, two sources are sufficient to allow for good news and bad news about each state. With $n>2$ states, there could be $n$ good-news sources and $n$ bad-news sources, greatly increasing the complexity of the DM's attention allocation problem. While limited progress has been made in multi-state models with only two news sources,\footnote{See the discussion of \cite{nikandrova/pancs:15} and \cite{Mayskaya2016} in the Introduction.} we conjecture that a characterization in a general model will not be tractable.\footnote{As far as we know, even in the Wald stopping problem, tractable characterizations are not available for more than two states \citep{Peskir2006}.}

\subsection{Non-Conclusive Evidence} \label{sec:non-conclusive} So far, we have assumed that the DM can obtain conclusive evidence---that is, a signal that arrives only in one state. This can be relaxed by introducing ``noise,'' or ``false evidence.''  Suppose the DM looks for $\omega$-evidence. With noise, this is received in state $\omega$ with a Poisson rate of $\overline \lambda$ but \emph{also in state} $\omega'\ne \omega$ with a lower rate $\underline \lambda <\overline \lambda$.
If $\underline \lambda >0$, then an $\omega$-signal is no longer conclusive evidence for state $\omega$. In a previous version of this paper \citep{CM2017}, we analyze this extension and show that, as long as $\underline \lambda$ is sufficiently small, the DM finds it optimal to take an action immediately after receiving a (noisy) breakthrough signal---a property we call \emph{Single Experimentation Property} (SEP).  Given SEP, our characterization
applies without any qualitative changes.\footnote{A general characterization is difficult to obtain if SEP does not hold. In \cite{CM2017} we solve such a case. While our analytical method developed here continues to be useful, the characterization is much more complex, involving multiple jumps across different learning regions. 
} Moreover, the main implications in terms of accuracy and delay reported in Proposition \ref{prop:speed_and_accuracy} continue to hold. A DM with a more uncertain belief, who chooses opposite-biased learning, ends up making a more accurate decision but with a greater delay than a DM with a more extreme belief who employs own-biased learning and as a result is more prone to mistakes.

This dependence of the stochastic choice function on the DM's prior belief was already present in the baseline model. With noise the stochastic choice becomes richer, and new phenomena appear. In particular, we show that, conditional on the prior belief, a decision maker who (by chance) received a breakthrough very quickly, makes a more accurate decision than a DM who had to wait a long time for a breakthrough.  This finding is consistent with so-called ``speed-accuracy complementarity''---a phenomenon that a more delayed decision tends to be less accurate.\footnote{This finding is often documented in perceptual judgment or consumption choice experiments conducted in cognitive psychology.  See \cite{ratcliff/mckoon:08} for a survey, and \cite{fudenberg:15} for a recent economic theory.}
The simple intuition is that noisy evidence is less convincing if the DM was more skeptical when receiving it.  This means that if the DM is unlucky and  waits for a longer time before receiving breakthrough news, the accuracy of her action will suffer.

\section{Conclusion\label{sec:Conclusion}}

We have studied a model in which a decision maker may allocate her
limited attention to collecting different types of evidence that support
alternative actions. Assuming Poisson arrival of evidence, we have
shown that the optimal policy combines \emph{immediate action,} \emph{own-biased
learning}, and \emph{opposite-biased learning} for different prior beliefs.
We have used this characterization to obtain rich predictions
about information acquisition and choices.

We envision several avenues of extending the current work. First,
our model is relatively tractable (e.g., the value function and optimal
policy are available in closed form). Therefore, we hope that this
framework will be useful for integrating dynamic choice of information
in applied theory models. This includes our application to
media choice, which could be extended beyond our analysis in the present
paper. Dynamic information choice might be also embedded in a model of a committee or a jury, or in a principal-agent setup in which a principal tries to induce an agent to acquire information
or a strategic setup such as R\&D competition where different firms
choose from alternative innovation approaches over time.

Second, one may relax the ``lumpiness'' of information to
allow for more gradual information acquisition. Our analysis
about repeated experimentation in \citet{CM2017} points to an extension in this
direction, and suggests that our characterization is robust to such
a generalization. A complete analysis of this case will be useful
for applications in which decision makers learn gradually, such as
a researcher who makes day-to-day decisions about the next steps in
a project, as opposed to a manager who decides only based on reports
that are made once a month. We leave this for future research.
 
\appendix 

\section{Proof of Theorem \ref{thm:optimal_attention_strategy} \label{sec:Main_Appendix}}

In this appendix, we describe the construction of the heuristic strategies
and state the main steps of the proof of Theorem \ref{thm:optimal_attention_strategy}.
Omitted proofs can be found in Section \ref{sec:omitted_proofs}
in the Online Appendix. For mathematical details on optimal control
problems that are used here, see \citet{bardi/capuzzo-dolceta:97}.

\subsection{The DM's problem}

The DM chooses an attention strategy $(\alpha_{t})$ and a time $T\in[0,\infty]$
at which she will stop acquiring information if she does not observe
any signal up to time $T$. Her problem is thus given by 
\begin{align}
V^{*}(p_{0})=\max_{(\alpha_{\tau}),T} & \int_{0}^{T}e^{-\rho t}P_{t}(p_{0},(\alpha_{\tau}))\left[\lambda\alpha_{t}p_{t}\,u_{r}^{R}+\lambda\beta_{t}(1-p_{t})\,u_{\ell}^{L}-c\right]dt+e^{-\rho T}P_{T}(p_{0},(\alpha_{\tau}))U(p_{T}),\label{eq:DMs_problemG}\\
\text{s.t. } & \dot{p}_{t}=-\lambda\left(\alpha_{t}-\beta_{t}\right)p_{t}\left(1-p_{t}\right),\notag
\end{align}
where $P_{t}(p_{0},(\alpha_{\tau})):=p_{0}e^{-\lambda\int_{0}^{t}\alpha_{s}ds}+(1-p_{0})e^{-\lambda\int_{0}^{t}\beta_{s}ds}$,
and $\beta_{t}=1-\alpha_{t}$. Given that the problem is autonomous,
the optimal $\alpha_{t}$ only depends on the belief at time $t$,
and can thus be written as a policy $\alpha(p)$. A policy $\alpha(p)$
is \emph{admissible} if, together with (A.2) it defines a unique path
$p_{t}$ for any prior $p_{0}\in[0,1]$. Similarly, the decision to
stop and take an action only depends on $p_{t}$.

\subsection{Two Benchmarks and a Condition for Experimentation}

Two benchmark value functions prove useful for our analysis. The first
is the value of the \emph{stationary strategy}: 
\begin{equation}
U^{S}(p):=p\frac{\frac{1}{2}\lambda\,u_{r}^{R}-c}{\rho+\frac{1}{2}\lambda}+(1-p)\frac{\frac{1}{2}\lambda\,u_{\ell}^{L}-c}{\rho+\frac{1}{2}\lambda}=\frac{\lambda\left(pu_{r}^{R}+(1-p)u_{\ell}^{L}\right)}{2\rho+\lambda}-\frac{2c}{2\rho+\lambda},\label{eq:UbarG}
\end{equation}
which arises when the DM chooses $\alpha_{t}=\beta_{t}=1/2$ for all
$t$ and takes an optimal action only after receiving the conclusive
breakthrough news.

The second is what we call the \emph{full-attention value}: 
\begin{equation}
U^{FA}(p):=p\frac{\lambda\,u_{r}^{R}-c}{\rho+\lambda}+(1-p)\frac{\lambda\,u_{\ell}^{L}-c}{\rho+\lambda}=\frac{\lambda\left(pu_{r}^{R}+(1-p)u_{\ell}^{L}\right)}{\rho+\lambda}-\frac{c}{\rho+\lambda},\label{eq:UhatG}
\end{equation}
which arises in a ``hypothetical'' scenario in which the DM chooses
(infeasible) attention $\alpha_{t}=\beta_{t}=1$ for all $t$ and
again takes an optimal action only after receiving breakthrough news.
Since limited attention prevents the DM from achieving this value
in our model, $U^{FA}(p)$ serves only as an analytical device for
proofs.

Intuitively, $U^{FA}(p)$ is an upper bound for a payoff the DM can
obtain from experimentation. Note further that $U^{FA}(\cdot)$ is linear and $U(\cdot)$ is piecewise-linear with a kink at $\hat p$. Hence,  the 
condition\footnote{It is easy to check that $U(p)\ge U^{FA}(p)$ for all $p$ if \eqref{eq:EXPG}	is violated.} 
\begin{equation}
\tag{EXP}U^{FA}(\hat{p})>U(\hat{p}) \label{eq:EXPG}
\end{equation}
would be necessary for experimentation to be optimal.   
Indeed, if \eqref{eq:EXPG} does not hold, an immediate action is optimal for all $p\in[0,1]$:
\begin{prop}
\label{prop:no_experimentation}For all $p\in[0,1]$, $U(p)\le V^{*}(p)\le\max\left\{ U(p),U^{FA}(p)\right\} $.
In particular, if \eqref{eq:EXPG} is violated, then $V^{*}(p)=U(p)$
for all $p$.
\end{prop}

\subsection{The Bellman equation}

In light of Proposition \ref{prop:no_experimentation}, in the sequel we only consider the case where \eqref{eq:EXPG} holds, and construct the value function for the range of beliefs
where $U^{FA}(p)>U(p)$. The HJB equation for the DM's problem in
\eqref{eq:DMs_problemG} is the following variational inequality 
\begin{equation}
\max\left\{ -c-\rho V(p)+\max_{\alpha\in[0,1]}F_{\alpha}(p,V(p),V'(p)),U(p)-V(p)\right\} =0,\label{eq:HJB_VIO}
\end{equation}
where 
\begin{equation}
F_{\alpha}(p,V(p),V'(p)):=\begin{Bmatrix}\alpha\lambda p\left(u_{r}^{R}-V(p)\right)+(1-\alpha)\lambda(1-p)\left(u_{\ell}^{L}-V(p)\right)\\
-\lambda\left(2\alpha-1\right)p(1-p)V'(p)
\end{Bmatrix}.\label{eq:HJBG_RHS}
\end{equation}
In the ``experimentation region'' where $V(p)>U(p)$, the HJB equation
reduces to 
\begin{align}
c+\rho V(p) & =F(p,V(p),V'(p))\left(:=\max_{\alpha\in[0,1]}F_{\alpha}(p,V(p),V'(p))\right).\label{eq:HJBG}
\end{align}
If $V(p)=U(p)$, then $T(p)=0$ is optimal and we must have $c+\rho V(p)\ge F(p,V(p),V'(p)).$

In the following, we will construct a candidate value function and show that it
satisfies \eqref{eq:HJB_VIO} for all points of differentiability.
This would be sufficient if the candidate were differentiable everywhere.
Since our candidate function has kinks, we show instead that it is
a \emph{viscosity solution} of \eqref{eq:HJB_VIO}, a necessary and
sufficient   condition for the value function according to the verification
theorem we invoke (see Proposition \ref{prop:envelope_characterization} below).

Note that $F_{\alpha}(\cdot)$ is linear in $\alpha$. Therefore,
the optimal policy is a bang-bang solution and we have $\alpha^{*}(p)\in\{0,1\}$
except for posteriors where the derivative of the objective vanishes.  With $\a$ set respectively to $0$ and 1, we can define functions, $V_0$ and $V_1$, satisfying the ODEs:     
\begin{align}
c+\rho V_{0}(p) & =F_{0}(p,V_{0}(p),V'_{0}(p))=\lambda(1-p)\left(u_{\ell}^{L}-V_{0}(p)\right)+\lambda p(1-p)V_{0}'(p),\label{eq:V_a0}\\
c+\rho V_{1}(p) & =F_{1}(p,V_{1}(p),V'_{1}(p))=\lambda p\left(u_{r}^{R}-V_{1}(p)\right)-\lambda p(1-p)V_{1}'(p).\label{eq:V_a1}
\end{align}
Solutions to these ODEs with boundary condition $V(x)=W$ are well-defined
if $x\in(0,1)$, and denoted by $V_{0}(p;x,W)$ and $V_{1}(p;x,W)$,
respectively.\footnote{$V_{0}(p;x,W)$ and $V_{1}(p;x,W)$ are uniquely defined if $x\in(0,1)$
because \eqref{eq:V_a0} and \eqref{eq:V_a1} satisfy local Lipschitz
continuity for all $p\in(0,1)$.}

\subsection{Own-Biased Strategy\label{subsec:Contradictory-Evidence}}

Recall the structure of the own-biased strategy given by (\ref{eq:contr_alpha_thm1}).
We will define its value, labeled $V_{own}(\cdot)$ to be an upper
envelope of two value functions, $\underline{V}_{own}(\cdot)$ and
$\overline{V}_{own}(\cdot)$, respectively its left- and right-branches.
To this end, we first compute the boundary beliefs, $\underline{p}^{*}$
and $\overline{p}^{*}$, and then construct the two branches by solving
the ODEs (\ref{eq:V_a0}) and (\ref{eq:V_a1}), using boundary conditions
at $\underline{p}^*$ and $\overline{p}^*$. The particular construction will be ultimately
justified later through our verification argument (Proposition \ref{prop:envelope_characterization}).

First, value matching and smooth pasting (relative the immediate action
payoffs) pin down the boundary beliefs:\footnote{We show in Section \ref{subsec:Solution-Candidate} that the boundary
beliefs satisfy $0<\underline{p}^{*}<\hat{p}<\overline{p}^{*}<1$
if \eqref{eq:EXPG} is satisfied.}
\begin{equation}
\underline{p}^{*}:=\frac{u_{\ell}^{L}\rho+c}{\rho\left(u_{\ell}^{L}-u_{\ell}^{R}\right)+\left(u_{r}^{R}-u_{\ell}^{R}\right)\lambda},\label{eq:plsG}
\end{equation}
\begin{equation}
\overline{p}^{*}=\frac{\left(u_{\ell}^{L}-u_{r}^{L}\right)\lambda-u_{r}^{L}\rho-c}{\rho\left(u_{r}^{R}-u_{r}^{L}\right)+\left(u_{\ell}^{L}-u_{r}^{L}\right)\lambda},\label{eq:pHsG}
\end{equation}
Next, we define the value of the left branch as $\underline{V}_{own}(p)=U_{\ell}(p)$
for $p\le\underline{p}^{*}$. For $p>\underline{p}^{*}$, we set $\underline{V}_{own}(p)=V_1(p;\underline{p},U_{\ell}(\underline{p}^{*}))$ which yields:
\begin{equation}
\underline{V}_{own}(p)=-\frac{c}{\rho}(1-p)+\frac{u_{r}^{R}\lambda-c}{\lambda+\rho}p+\frac{\lambda\left(c+u_{\ell}^{L}\rho\right)}{\rho\left(\lambda+\rho\right)}\left(\frac{\underline{p}^{*}}{1-\underline{p}^{*}}\right)^{\frac{\rho}{\lambda}}\left(\frac{1-p}{p}\right)^{\frac{\rho}{\lambda}}(1-p).\label{eq:VctL}
\end{equation}
Similarly, the value $\overline{V}_{own}(p)$ of the right branch
equals $U_r(p)$ for $p\ge\overline{p}^{*}$. For $p<\overline{p}^{*}$,
we set $\overline{V}_{own}(p)=V_0(p;\overline{p}^{*},U_{r}(\overline{p}^{*}))$ which yields:
\begin{equation}
\overline{V}_{own}(p):=-\frac{c}{\rho}p+\frac{u_{\ell}^{L}\lambda-c}{\lambda+\rho}(1-p)+\frac{\lambda\left(c+u_{r}^{R}\rho\right)}{\rho\left(\lambda+\rho\right)}\left(\frac{1-\overline{p}^{*}}{\overline{p}^{*}}\right)^{\frac{\rho}{\lambda}}\left(\frac{p}{1-p}\right)^{\frac{\rho}{\lambda}}p.\label{eq:VctR}
\end{equation}

Combining these functions, we define the value of the own-biased strategy
as $V_{own}(p):=\max\left\{ \underline{V}_{own}(p),\overline{V}_{own}(p)\right\} .$
Without further analysis, it is not clear when $V_{own}(p)$ is the
value of a strategy of the form \eqref{eq:contr_alpha_thm1}. This
will be clarified in Section \ref{subsec:Solution-Candidate}. 

\subsection{Opposite-Biased Strategy}

Recall the structure of the opposite-biased strategy given by (\ref{eq:alpha_conf}). The value of this strategy, denoted by $V_{opp}(p)$,
and the reference belief $p^{*}$ are defined as follows. First, we
observe that the value must equal the stationary value $U^{S}(p^{*})$
at $p^{*}$. Given this, we invoke value matching and smooth pasting
to pin down\footnote{Namely, we insert $V_{0}(p^{*})=U^{S}(p^{*})$ and $V'_{0}(p^{*})=U^{S\prime}(p^{*})$
in \eqref{eq:V_a0}. It turns out that this yields the same value
for $p^{*}$ as the one obtained by inserting $V_{1}(p^{*})=U^{S}(p^{*})$
and $V'_{1}(p^{*})=U^{S\prime}(p^{*})$ in \eqref{eq:V_a1}.} 
\begin{equation}
p^{*}:=\frac{\left(u_{\ell}^{L}\rho+c\right)}{\left(u_{r}^{R}\rho+c\right)+\left(u_{\ell}^{L}\rho+c\right)}.\label{eq:pstarG}
\end{equation}
For $p\le p^{*}$ we have $V_{opp}(p)=V_0(p;p^{*},U^{S}(p^{*}))$  which yields: 
\begin{equation}
\underline{V}_{opp}(p):=-\frac{c}{\rho}p+\frac{u_{\ell}^{L}\lambda-c}{\lambda+\rho}(1-p)+\frac{\lambda}{\rho\left(2\rho+\lambda\right)}\frac{\lambda\left(u_{r}^{R}\rho+c\right)}{\lambda+\rho}\left(\frac{1-p^{*}}{p^{*}}\frac{p}{1-p}\right)^{\frac{\rho}{\lambda}}p.\label{eq:VcfL}
\end{equation}
Likewise, for $p\ge p^{*}$, we have $V_{opp}(p)=V_1(p;p^{*},U^{S}(p^{*}))$  which yields:
\begin{equation}
\overline{V}_{opp}(p):=-\frac{c}{\rho}(1-p)+\frac{u_{r}^{R}\lambda-c}{\lambda+\rho}p+\frac{\lambda}{\rho\left(2\rho+\lambda\right)}\frac{\lambda\left(u_{\ell}^{L}\rho+c\right)}{\lambda+\rho}\left(\frac{p^{*}}{1-p^{*}}\frac{1-p}{p}\right)^{\frac{\rho}{\lambda}}(1-p).\label{eq:VcfR}
\end{equation}

\subsection{Solution Candidate\label{subsec:Solution-Candidate}}

Again we assume \eqref{eq:EXPG} is satisfied. We define our solution candidate as the upper
envelope of $V_{own}(p)$ and $V_{opp}(p)$, denoted by $V_{Env}(p):=\max\left\{ V_{own}(p),V_{opp}(p)\right\} $. This function is characterized as follows.
\begin{prop}[\textbf{Structure of $V_{Env}$}]
\label{prop:Structure_V_G} 
\begin{enumerate}
\item If \eqref{eq:EXPG} holds and $V_{own}(p^{*})\ge V_{opp}(p^{*})$,
then there exists a unique $\check{p}\in\left(\underline{p}^{*},\overline{p}^{*}\right)$
such that $\underline{V}_{own}(\check{p})=\overline{V}_{own}(\check{p})$
and 
\[
V_{Env}(p)=V_{own}(p)=\begin{cases}
\underline{V}_{own}(p), & \text{if }p<\check{p},\\
\overline{V}_{own}(p), & \text{if }p\ge\check{p}.
\end{cases}
\]
\item If \eqref{eq:EXPG} holds and $V_{own}(p^{*})<V_{opp}(p^{*})$, then
$p^{*}\in(\underline{p}^{*},\overline{p}^{*})$, and there exist a
unique $\underline{p}\in(\underline{p}^{*},p^{*})$ such that $V_{own}(\underline{p})=V_{opp}(\underline{p})$,
and a unique $\overline{p}\in(p^{*},\overline{p}^{*})$ such that
$V_{own}(\overline{p})=V_{opp}(\overline{p})$ and 
\[
V_{Env}(p)=\begin{cases}
\underline{V}_{own}(p), & \text{if }p<\overline{p},\\
V_{opp}(p), & \text{if }p\in[\underline{p},\overline{p}],\\
\overline{V}_{own}(p), & \text{if }p>\overline{p}.
\end{cases}
\]
\end{enumerate}
\end{prop}
To understand how we derive the structure of the solution candidate
it is useful to make several geometric observations, which are depicted for illustrative purpose in Figure \ref{fig:branches_for_proof}.  Note first that $\underline{V}_{own}(p)$ is strictly convex on $[\underline{p}^*,1]$, $\overline{V}_{own}(p)$ is strictly convex on $[0,\overline{p}^*]$, and $V_{opp}(p)$ is strictly convex on $[0,1]$. This can be seen directly from   \eqref{eq:VctL}--\eqref{eq:VctR} and \eqref{eq:VcfL}--\eqref{eq:VcfR}.\footnote{For $V_{opp}(p)$, its convexity on the whole interval $[0,1]$ follows from the convexity of the two branches $\underline{V}_{opp}(p)$ and $\overline{V}_{opp}(p)$ and smooth pasting at $p^*$.}
 Figure \ref{fig:branches_for_proof} shows that these three functions coincide with $U^{FA}(p)$ at the endpoints of the respective intervals above. When the endpoints are $p=0$ and $p=1$ this can be seen by comparing the explicit expression and \eqref{eq:UhatG}.
Our first crucial Lemma shows that the value of own-biased learning coincides with 
$U^{FA}$ at the boundary points $\underline{p}^{*}$
and $\overline{p}^{*}$.
\begin{figure}
	\begin{tabular}{cc}
		\includegraphics[bb=0bp 0bp 540bp 400bp,height=0.203\textheight]{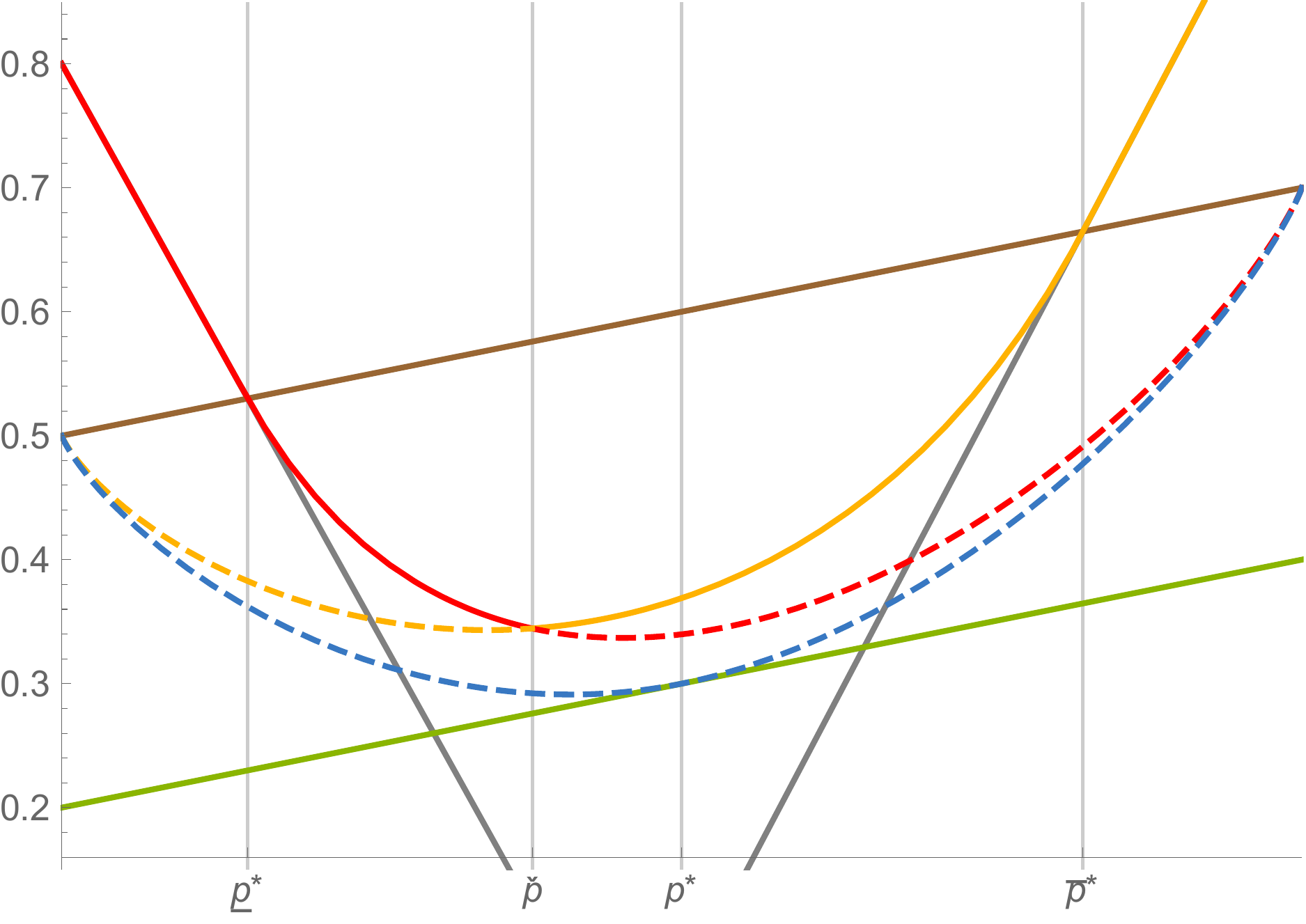} & \includegraphics[bb=0bp 0bp 675bp 400bp,height=0.203\textheight]{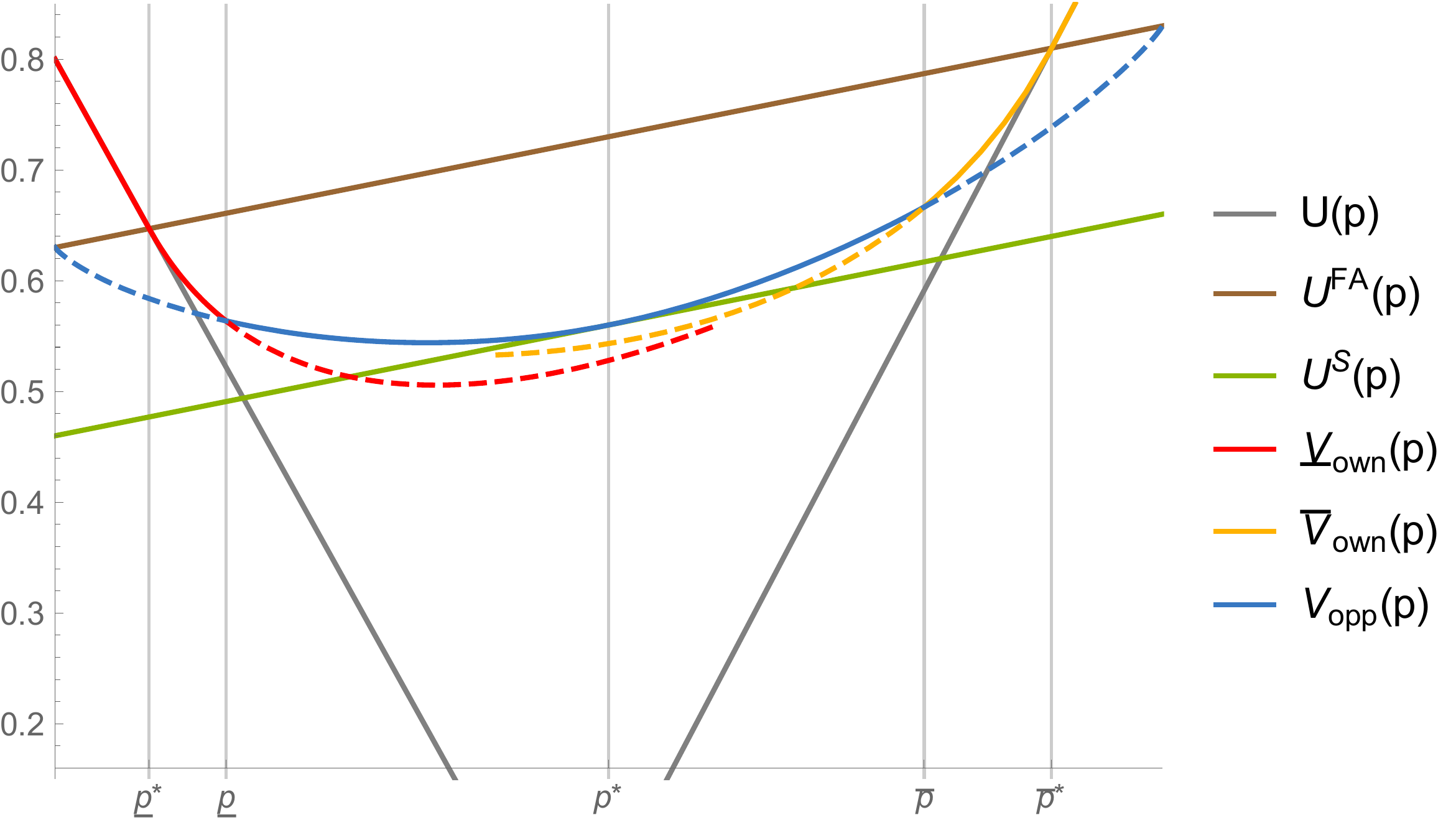}\tabularnewline
		{\footnotesize{}{}{}{}{}{}{}{}(a) only own-biased ($c=.3$)} & {\footnotesize{}{}{}{}{}{}{}{}(b) opposite-biased and own-biased ($c=.13$)}\tabularnewline
	\end{tabular}
	\caption{\label{fig:branches_for_proof} Branches of the value function and
	solution candidate.}
\medskip
\begin{footnotesize}
\begin{spacing}{1} 
Note: Dashed lines indicated segments of the branches that are not part of $V_{Env}$. (Parameters: $\lambda=1$, $\rho=0$, $u^R_r=1$, $u^L_\ell=.8$, $u^R_\ell=-1$, $u^L_r=-.8$)
\end{spacing}
\end{footnotesize}
\end{figure}
\begin{lem}
\label{lem:The-boundary-points}The boundary points $\underline{p}^{*}$
and $\overline{p}^{*}$ satisfy
\begin{equation}
U_{\ell}(\underline{p}^{*})=U^{FA}(\underline{p}^{*})\quad\text{and}\quad U_{r}(\overline{p}^{*})=U^{FA}(\overline{p}^{*}).\label{eq:geom_charac_boundaries}
\end{equation}
If \eqref{eq:EXPG} is satisfied, then $0<\underline{p}^{*}<\hat{p}<\overline{p}^{*}<1$.
\end{lem}

Equipped with these preliminary observations, we can turn to the characterization of the upper envelope in Proposition \ref{prop:Structure_V_G}. The following crucial ``crossing lemma'' characterizes how a function solving \eqref{eq:V_a0} (such as $\overline{V}_{own}$ or $\underline{V}_{opp}$) intersects a function that solves \eqref{eq:V_a1} (such as $\underline{V}_{own}$ or $\overline{V}_{opp}$).
\begin{lem}[\textbf{Crossing Lemma}]
\label{lem:branch_crossing}  Let $V_{0}$ satisfy \eqref{eq:V_a0} and $V_{1}$ satisfy \eqref{eq:V_a1}. 
If $V_{0}(p)=V_{1}(p)>(=)U^{S}(p)$
at some at $p\in(0,1)$, then $V_{0}'(p)>(=)V_{1}'(p)$. 
\end{lem}

The Lemma implies that $\underline{V}_{own}$ must cross $\overline{V}_{own}$  (and likewise $\underline{V}_{own}$ and $\overline{V}_{opp}$  must cross  $\underline{V}_{opp}$ and $\overline{V}_{own}$ respectively from above), when an intersection occurs above the stationary value function.

Suppose first that $\max\{\underline{V}_{own}(p),\overline{V}_{own}(p)\}\ge U^S(p)$ for all $p\in[0,1]$. Our preliminary observations imply that $\underline{V}_{own}$ and $\overline{V}_{own}$ must cross each other at some $p\in(\underline{p}^*,\overline{p}^*)$.%
\footnote{Since $\underline{V}_{own}$ is strictly convex and equal to the linear function $U^{FA}(p)$ for $p\in\{\underline{p}^*,1\}$, we have $\underline{V}_{own}(\overline{p}^*)<U^{FA}(\overline{p}^*)=\overline{V}_{own}(\overline{p}^*)$. Similarly, we obtain $\overline{V}_{own}(\underline{p}^*)<\underline{V}_{own}(\underline{p}^*)$. Hence there must be an intersection.} %
Since their upper envelope exceeds $U^S(p)$, the Crossing Lemma \ref{lem:branch_crossing} implies that $\underline{V}_{own}(p)$ intersects $\overline{V}_{own}(p)$ from above. Consequently the intersection point $\check{p}$ must be unique, as depicted in Panel (a) of Figure \ref{fig:branches_for_proof}. This summarizes the crucial step for part (a) of Proposition \ref{prop:Structure_V_G}.%
\footnote{See the complete proof in Appendix \ref{sec:omitted_proofs} for the remaining steps: (i) $\max\{\underline{V}_{own}(p^*),\overline{V}_{own}(p^*)\}\ge U^S(p^*)$ implies the stronger condition $\max\{\underline{V}_{own}(p),\overline{V}_{own}(p)\}\ge U^S(p)$ for all $p\in[0,1]$ that we used here, and (ii) $V_{opp}(p)\le V_{own}(p)$ for all $p\in(0,1)$ if $\max\{\underline{V}_{own}(p^*),\overline{V}_{own}(p^*)\}\ge U^S(p^*)$.} %

Suppose next that $\max\{\underline{V}_{own}(p^*),\overline{V}_{own}(p^*)\}< U^S(p^*)$, as in Part (b) of Proposition \ref{prop:Structure_V_G}, a case depicted in Panel (b) of Figure \ref{fig:branches_for_proof}. Consider the interval $(\underline{p}^*,p^*)$.
From Lemma \ref{lem:The-boundary-points} and our preliminary observations, we have $V_{opp}(\underline{p}^*)<U^{FA}(\underline{p}^*)=\underline{V}_{own}(\underline{p}^*)$.
 Since $V_{opp}(p^*)>\underline{V}_{own}(p^*)$, there must be an intersection between $V_{opp}(p)$ and $\underline{V}_{own}(p)$ at some $p \in (\underline{p}^*, p^*)$.
 Since $V_{opp}(p)=\underline{V}_{opp}(p)>U^S(p)$ for all $p<p^*$, the Crossing Lemma \ref{lem:branch_crossing} implies that $\underline{V}_{own}(p)$ intersects $V_{opp}(p)$ from above and hence there is a unique intersection $\underline{p}\in(\underline{p}^*,p^*)$. 
 The characterization of the upper envelope for $p<p^*$ is completed by noting that $\overline{V}_{own}(p)<{V}_{opp}(p)$ for all $p<p^*$.\footnote{This follows from the fact that $\overline{V}_{own}(p^*) < \overline{V}_{opp} (p^*)=\underline{V}_{opp}(p^*)$.  Since both $\overline{V}_{own}$  and $\underline{V}_{opp}$ satisfy \eqref{eq:V_a0}, the former stays below the latter for all $p<p^*$.} A symmetric argument is used to characterize $V_{Env}(p)$ for $p>p^*$. 

\subsection{Verification of the Candidate}

We now show that $V_{Env}$ is the value function of the DM's problem
in \eqref{eq:DMs_problemG}. 
\begin{prop}
\label{prop:envelope_characterization}$V^{*}(p)=V_{Env}(p)$ for
all $p\in[0,1]$. Up to tie-breaking at $\check{p}$, $\underline{p}$,
$\overline{p}$, $\underline{p}^{*}$, and $\overline{p}^{*}$, the
optimal policy is unique. 
\end{prop}
For $p\notin(\underline{p}^{*},\overline{p}^*)$, $V_{Env}(p)=U(p)$
which is equal to $V^{*}(p)$ by Proposition \ref{prop:no_experimentation}.
To show optimality for beliefs inside the experimentation region,
the following Lemma is crucial.
\begin{lem}[\textbf{Unimprovability}]
\begin{enumerate}
\item If $V_{0}$ satisfies \eqref{eq:V_a0} and $V_{0}(p)\ge U^{S}(p)$
for some $p\in\left[0,1\right]$, then $V_{0}$ satisfies \eqref{eq:HJBG}
at $p$, and $\alpha=0$ is a maximizer. 
\item If $V_{1}$ satisfies \eqref{eq:V_a1} and $V_{1}(p)\ge U^{S}(p)$
for some $p\in\left[0,1\right]$, then $V_{1}$ satisfies \eqref{eq:HJBG}
at $p$, and $\alpha=1$ is a maximizer. 
\end{enumerate}
The maximizers are unique if $V_0(p),V_1(p)>U^{S}(p)$. \label{lem:Branches_satisfy_HJB} 
\end{lem}
As we have argued in the previous section, $V_{opp}(p)\ge U^{S}(p)$
for all $p\in[0,1]$, which implies $V_{Env}(p)\ge U^{S}(p)$.
Remember that $V_{Env}(p)$ is constructed of functions that satisfy
\eqref{eq:V_a0} or \eqref{eq:V_a1}, respectively. Therefore, Lemma
\ref{lem:Branches_satisfy_HJB} shows that $V_{Env}(p)$ satisfies
the HJB equation for all points where it is differentiable. We have thus verified optimality for all $p$ where $V_{Env}(p)$ is differentiable.   For 
verification at points where $V_{Env}(p)$ is not differentiable,
we show that it is a viscosity solution of the HJB equation (see the
proof of Proposition \ref{prop:envelope_characterization} in Appendix
\ref{appendix:proofs_mainsection} in the Supplemental Material).

\subsection{Proof of Theorem \ref{thm:optimal_attention_strategy}\label{subsec:Proof_of_Theorem_1}}

We show that Theorem \ref{thm:optimal_attention_strategy} holds with
the cutoffs: 
\begin{align}
\overline{c} & :=0\vee\frac{\lambda\left(u_{r}^{R}-u_{\ell}^{R}\right)\left(u_{\ell}^{L}-u_{r}^{L}\right)-\rho\left(u_{r}^{R}u_{\ell}^{L}-u_{\ell}^{R}u_{r}^{L}\right)}{\left(u_{r}^{R}-u_{\ell}^{R}\right)+\left(u_{\ell}^{L}-u_{r}^{L}\right)},\label{eq:zetah}\\
\underline{c} & :=0\vee\begin{cases}
\overline{c}\wedge\min\left\{ \frac{\left(\rho+\lambda\right)\left(u_{r}^{R}-u_{\ell}^{R}\right)}{1+\left(\frac{2\rho+\lambda}{\lambda}\right)^{\lambda/\rho}}-\rho u_{r}^{R},\frac{\left(\rho+\lambda\right)\left(u_{\ell}^{L}-u_{r}^{L}\right)}{1+\left(\frac{2\rho+\lambda}{\lambda}\right)^{\lambda/\rho}}-\rho u_{\ell}^{L}\right\}  & \text{if }\rho>0,\\
\overline{c}\wedge\frac{\lambda}{1+e^{2}}\min\left\{ \left(u_{r}^{R}-u_{\ell}^{R}\right),\left(u_{\ell}^{L}-u_{r}^{L}\right)\right\}  & \text{if }\rho=0,
\end{cases}\label{eq:zetal}
\end{align}
where $x\vee y=\max\left\{ x,y\right\} $ and $x\wedge y=\min\left\{ x,y\right\} $. 
\begin{proof}[Proof of Theorem \ref{thm:optimal_attention_strategy}]
Straightforward algebra shows that \eqref{eq:EXPG} is equivalent
to 
\[
c\left(\left(u_{r}^{R}-u_{\ell}^{R}\right)+\left(u_{\ell}^{L}-u_{r}^{L}\right)\right)+\rho\left(u_{r}^{R}u_{\ell}^{L}-u_{\ell}^{R}u_{r}^{L}\right)<\lambda\left(u_{r}^{R}-u_{\ell}^{R}\right)\left(u_{\ell}^{L}-u_{r}^{L}\right).
\]
By Propositions \ref{prop:no_experimentation}, immediate action is
optimal for all $p\in[0,1]$, if \eqref{eq:EXPG} is violated, which
holds if and only if $c\ge\overline{c}$, where $\overline{c}$ is
given by \eqref{eq:zetah}. This proves part (a).

Conversely, if $c\le\overline{c}$, then \eqref{eq:EXPG} is satisfied,
and by Propositions \ref{prop:Structure_V_G} and \ref{prop:envelope_characterization}
experimentation is optimal for some beliefs. We show that 
\begin{equation}
c\ge\underline{c}\quad\iff\quad\max\left\{ \underline{V}_{own}(p^{*}),\overline{V}_{own}(p^{*})\right\} \ge U^{S}(p^{*}).\label{eq:characerization_zetal}
\end{equation}
By Propositions \ref{prop:Structure_V_G} and \ref{prop:envelope_characterization},
and the Unimprovability Lemma \ref{lem:Branches_satisfy_HJB}, this
implies that the policies stated in Parts (b) and (c) of Theorem \ref{thm:optimal_attention_strategy}
are optimal.

We first assume $\rho>0$. The closed-form solutions for $\underline{V}_{own}(p^{*})$
and $\overline{V}_{own}(p^{*})$ in \eqref{eq:VctL} and \eqref{eq:VctR}
can be used to show that 
\begin{align*}
 & \max\left\{ \underline{V}_{own}(p^{*}),\overline{V}_{own}(p^{*})\right\} \ge U^{S}(p^{*})\\
\iff & \left(\max\left\{ \frac{c+\rho u_{r}^{R}}{\lambda u_{r}^{R}-(\rho+\lambda)u_{\ell}^{R}-c},\frac{c+\rho u_{\ell}^{L}}{\lambda u_{\ell}^{L}-(\rho+\lambda)u_{r}^{L}-c}\right\} \right)^{\frac{\rho}{\lambda}}\ge\frac{\lambda}{2\rho+\lambda}\iff c\ge\underline{c}.
\end{align*}
This proves \eqref{eq:characerization_zetal} for $\rho>0$. Taking
the limit $r\rightarrow0$ yields the result for $r=0$.\footnote{This can also be obtained directly by solving the ODEs for $r=0$
to obtain $\underline{V}_{own}(p^{*})$ and $\overline{V}_{own}(p^{*})$.}

By definition we have $\overline{c}\ge\underline{c}\ge0$. It remains
to show that the inequalities are strict for $\rho$ sufficiently
small. For $\rho=0$, $\underline{c}>0$ and 
\begin{align*}
\overline{c} & =\lambda\min\left\{ \left(u_{r}^{R}-u_{\ell}^{R}\right),\left(u_{\ell}^{L}-u_{r}^{L}\right)\right\} \frac{\max\left\{ \left(u_{r}^{R}-u_{\ell}^{R}\right),\left(u_{\ell}^{L}-u_{r}^{L}\right)\right\} }{\left(u_{r}^{R}-u_{\ell}^{R}\right)+\left(u_{\ell}^{L}-u_{r}^{L}\right)}\\
 & >\lambda\min\left\{ \left(u_{r}^{R}-u_{\ell}^{R}\right),\left(u_{\ell}^{L}-u_{r}^{L}\right)\right\} \frac{1}{2}>\underline{c}.
\end{align*}
Since both cutoffs are continuous in $\rho$, $\overline{c}>\underline{c}>0$
for $\rho$ in a neighborhood of zero. 
\end{proof}

\begin{singlespace}
	\bibliographystyle{economet}
	\bibliography{bibcm}
\end{singlespace}

\newpage{}

\section*{Supplemental Material (for online publication)}

\section{Remaining Proofs from Section \ref{sec:Analysis_Baseline}}

\label{appendix:proofs_mainsection}

\subsection{Omitted Proofs from Appendix \ref{sec:Main_Appendix}.}\label{sec:omitted_proofs}

\subsubsection{Proof of Proposition \ref{prop:no_experimentation}}
\begin{proof}
Since the DM can stop immediately, we have $V^{*}(p)\ge U(p)$. For
the second inequality, consider the problem of a decision maker who
can choose $\alpha_{t}\in[0,1]$ and $\beta_{t}\in[0,1]$ without
the constraint that $\alpha_{t}+\beta_{t}=1$. Clearly the value of
this problem exceeds $V^{*}(p)$ for all $p$. The value function
of the unconstrained problem is $\max\left\{ U(p),U^{FA}(p)\right\} $.
To see this, it is optimal to choose $\alpha_{t}=\beta_{t}=1$. Given
this policy, the belief does not change over time if no breakthrough
occurs. The optimal policy is therefore either to stop immediately
or to wait without deadline until a breakthrough occurs. Hence the
value of the unconstrained problem is $\max\left\{ U(p),U^{FA}(p)\right\} $.
Therefore $V^{*}(p)=U(p)=\max\left\{ U(p),U^{FA}(p)\right\} $ if
\eqref{eq:EXPG} is violated.
\end{proof}

\subsubsection{Proof of Lemma \ref{lem:branch_crossing}}
\begin{proof}
Suppose $V_{0}(p)=V_{1}(p)=V(p)$ for some $p\in(0,1)$. Solving \eqref{eq:V_a0}
and \eqref{eq:V_a1} for $V_{0}'(p)$ and $V_{1}'(p)$ and some algebra
yields 
\[
V_{0}'(p)-V_{1}'(p)=\frac{\lambda+2\rho}{\lambda p(1-p)}\left(V(p)-U^{S}(p)\right).
\]
Therefore $\mbox{sgn}(V_{0}'(p)-V_{1}'(p))=\mbox{sgn}\left(V(p)-U^{S}(p)\right)$. 
\end{proof}

\subsubsection{Proof of Lemma \ref{lem:Branches_satisfy_HJB}}
\begin{proof}
Consider first the case that $V_{0}(p)$ satisfies \eqref{eq:V_a0}.
With $V=V_{0}(p)$, and substituting $V'=V'_{0}(p)$ from \eqref{eq:V_a0},
we have 
\[
\frac{\partial F_{\alpha}(p,V_{0}(p),V'_{0}(p))}{\partial\alpha}=\frac{2\rho+\lambda}{\lambda}\left(U^{S}(p)-V_{0}(p)\right).
\]
This implies that $\alpha=0$ is a maximizer if $V_{0}(p)\ge U^{S}(p)$,
and the unique maximizer if the inequality is strict. This proves
Part (a). The proof of Part (b) follows from a similar argument. 
\end{proof}

\subsubsection{Proof of Proposition \ref{prop:Structure_V_G}}

The following three lemmas establish properties of the function $U^{S},V^{FA},V_{own}$
and $V_{opp}$ that are used in the proof of Proposition \ref{prop:Structure_V_G}.
Some of these properties were already established in Appendix \ref{sec:Main_Appendix}
and are repeated here for convenience.
\begin{lem}[Properties of $U^{S}(p)$ and $U^{FA}(p)$]
\label{lem:Properties_of_Benchmarks} 
\begin{enumerate}
\item $U^{S}(p)<U^{FA}(p)$ for all $p\in[0,1]$. 
\item $U^{S}(p)$ and $U^{FA}(p)$ are linear in $p$.\\
 If $U^{S}(p)\ge U(p)$ for some $p\in[0,1]$, then $U'_{\ell}(p)<U^{S\prime}(p)<U'_{r}(p)$
for all $p\in[0,1]$.\\
 If $U^{FA}(p)\ge U(p)$ for some $p\in[0,1]$, then $U'_{\ell}(p)<U^{FA\prime}(p)<U'_{r}(p)$
for all $p\in[0,1]$. 
\item $U^{S}(p),U^{FA}(p)<U(p)$ at $p\in\{0,1\}$; and for all $p\in[0,1]$,
$U^{S}(p)$ and $U^{FA}(p)$ are strictly decreasing without bound
in $c$. 
\end{enumerate}
\end{lem}
\begin{proof}
(a) $U^{S}(p)<U^{FA}(p)$ is immediate from the expressions in \eqref{eq:UbarG}
and \eqref{eq:UhatG}.

(b) Linearity is obvious. Suppose $U^{S}(p)\ge U(p)$ for some $p\in[0,1]$.
To show $U'_{\ell}(p)<U^{S\prime}(p)$ for all $p$, suppose
by contradiction that $U^{S\prime}(p)\le U'_{\ell}(p)$ for
some $p$. Note that $U^{S}(0)=\frac{u_{\ell}^{L}\lambda-2c}{\lambda+2\rho}<u_{\ell}^{L}=U_{\ell}(0)$.
Hence, $U^{S\prime}(p)\le U'_{\ell}(p)$ and the linearity of
these functions imply $U^{S}(p)<U_{\ell}(p)\le U(p)$ for all $p$,
which is a contradiction. The other inequalities are proven similarly.

Part (c) is obtained from straightforward algebra. 
\end{proof}
The following lemma summarizes the properties of the own-biased strategy: 
\begin{lem}
\label{lem:properties_contradictory} 
\begin{enumerate}
\item $\underline{V}_{own}(p)$ and $\overline{V}_{own}(p)$ are continuously
differentiable and convex on $\left(0,1\right)$; 
\item $\underline{V}_{own}(p)$ is strictly convex and $\underline{V}_{own}(p)>U_{\ell}(p)$
on $(\underline{p}^{*},1]$, and $\overline{V}_{own}(p)$ is strictly
convex and $\overline{V}_{own}(p)>U_{r}(p)$ on $[0,\overline{p}^{*})$.
$V_{own}(p)>U(p)$ for $p\in(\underline{p}^{*},\overline{p}^{*})$. 
\item If $\underline{p}^{*},\overline{p}^{*}\in(0,1)$, they satisfy 
\begin{equation}
U_{\ell}(\underline{p}^{*})=U^{FA}(\underline{p}^{*}),\quad\text{and}\quad U_{r}(\overline{p}^{*})=U^{FA}(\overline{p}^{*}).\label{eq:intersection_U_and_Uhat}
\end{equation}
\item Suppose \eqref{eq:EXPG} holds. Then, $0<\underline{p}^{*}<\overline{p}^{*}<1$,
$\underline{V}_{own}(p)<U^{FA}(p)$ for $p\in(\underline{p}^{*},1)$,
$\overline{V}_{own}(p)<U^{FA}(p)$ for $p\in(0,\overline{p}^{*})$,
and $V_{own}(p)=U(p)>U^{FA}(p)$ for $p\not\in[\underline{p}^{*},\overline{p}^{*}]$. 
\item If \eqref{eq:EXPG} is violated, then $V_{own}(p)=U(p)$ for all $p\in[0,1]$. 
\end{enumerate}
\end{lem}
\begin{proof}
Parts (a)-(c) follow from straightforward algebra. For part (d), note
that \eqref{eq:EXPG} together with part (c) and Lemma \ref{lem:Properties_of_Benchmarks}.(b)
imply $0<\underline{p}^{*}<\hat{p}<\overline{p}^{*}<1$ and $U^{FA}(p)<U(p)$
for $p\notin\left[\underline{p}^{*},\overline{p}^{*}\right]$. This
implies $V_{own}(p)=U(p)>U^{FA}(p)$ for $p\not\in[\underline{p}^{*},\overline{p}^{*}]$.
To show that $\underline{V}_{own}(p)<U^{FA}(p)$ for $p\in(\underline{p}^{*},1)$,
note that $\underline{V}_{own}(\underline{p}^{*})=U_{\ell}(\underline{p}^{*})=U^{FA}(\underline{p}^{*})$
from part (c), and $\underline{V}_{own}(1)=U^{FA}(1)$ from \eqref{eq:VctL}.
Since $U^{FA}(p)$ is linear by Lemma \ref{lem:Properties_of_Benchmarks}
and $\underline{V}_{own}(p)$ is strictly convex $(\underline{p}^{*},1]$
by part (b), this implies implies that $\underline{V}_{own}(p)<U^{FA}(p)$
for $p\in(\underline{p}^{*},1)$. $\overline{V}_{own}(p)<U^{FA}(p)$
for $p\in(0,\overline{p}^{*})$ is proven similarly.

Part (e) holds because by part (c), $\underline{p}^{*}>\overline{p}^{*}$
if \eqref{eq:EXPG} is violated. 
\end{proof}
We next observe several properties of $V_{opp}(p)$. 
\begin{lem}
\label{lem:properties_confirmatory} 
\begin{enumerate}
\item $V_{opp}(p)$ is continuously differentiable and strictly convex on
$\left(0,1\right)$, and $V_{opp}(p)\ge U^{S}(p)$ for all $p\in[0,1]$
with strict inequality for $p\neq p^{*}$. 
\item Then, $V_{opp}(p)\le U^{FA}(p)$ for all $p\in[0,1]$, with equality
if and only if $p\in\left\{ 0,1\right\} $. 
\end{enumerate}
\end{lem}
\begin{proof}
Part (a) follows from straightforward algebra. For part (b), again
by straightforward algebra we get $U^{FA}(0)=\underline{V}_{opp}(0)=V_{opp}(0)$
and $U^{FA}(1)=\overline{V}_{opp}(0)=V_{opp}(1)$. Since $U^{FA}(p)$
is linear and $V_{opp}$ is strictly convex, this implies $V_{opp}(p)<U^{FA}(p)$
for all $p\in[0,1]$. 
\end{proof}
We are now ready to prove Proposition \ref{prop:Structure_V_G}. For
the reader's convenience, we restate the proposition.
\begin{prop*}[\textbf{Structure of $V_{Env}$}]
\begin{enumerate}
\item If \eqref{eq:EXPG} holds and $V_{own}(p^{*})\ge V_{opp}(p^{*})$,
then there exists a unique $\check{p}\in\left(\underline{p}^{*},\overline{p}^{*}\right)$
such that $\underline{V}_{own}(\check{p})=\overline{V}_{own}(\check{p})$
and 
\[
V_{Env}(p)=V_{own}(p)=\begin{cases}
\underline{V}_{own}(p), & \text{if }p<\check{p},\\
\overline{V}_{own}(p), & \text{if }p\ge\check{p}.
\end{cases}
\]
\item If \eqref{eq:EXPG} holds and $V_{own}(p^{*})<V_{opp}(p^{*})$, then
$p^{*}\in(\underline{p}^{*},\overline{p}^{*})$, and there exist a
unique $\underline{p}\in(\underline{p}^{*},p^{*})$ such that $V_{own}(\underline{p})=V_{opp}(\underline{p})$,
and a unique $\overline{p}\in(p^{*},\overline{p}^{*})$ such that
$V_{own}(\overline{p})=V_{opp}(\overline{p})$ and 
\[
V_{Env}(p)=\begin{cases}
\underline{V}_{own}(p), & \text{if }p<\overline{p},\\
V_{opp}(p), & \text{if }p\in[\underline{p},\overline{p}],\\
\overline{V}_{own}(p), & \text{if }p>\overline{p}.
\end{cases}
\]
\end{enumerate}
\end{prop*}
\begin{proof}
\textbf{Part (a):} We first prove that $V_{own}(p)\ge V_{opp}(p)$
for all $p\in[0,1]$. Since ${V}_{own}(p)\ge U^{FA}(p)>{V}_{opp}(p)$
for $p\not\in[\underline{p}^{*},\overline{p}^{*}]$, it suffices to
show ${V}_{own}(p)\ge{V}_{opp}(p)$ for $p\in[\underline{p}^{*},\overline{p}^{*}]$.
To this end, suppose first $p^{*}>\underline{p}^{*}$ and consider
$p\in[\underline{p}^{*},p^{*}]$ so that $V_{opp}(p)=\underline{V}_{opp}(p)$.
Recall from Lemmas \ref{lem:properties_contradictory} and \ref{lem:properties_confirmatory}
that $\underline{V}_{own}(\underline{p}^{*})=U^{FA}(\underline{p}^{*})>V_{opp}(\underline{p}^{*})$.
Since ${V}_{opp}(\cdot)\ge U^{S}(\cdot)$, by the Crossing Lemma \ref{lem:branch_crossing},
$\underline{V}_{own}$ can cross $V_{opp}=\underline{V}_{opp}(p)$
only from above on $[\underline{p}^{*},{p}^{*})$. If $\underline{V}_{own}(p^{*})\ge\underline{V}_{opp}(p)(p^{*})$,
by the Crossing Lemma \ref{lem:branch_crossing}, $\underline{V}_{opp}(p)<\underline{V}_{own}(p)\le V_{own}(p)$
for all $p\in[\underline{p}^{*},{p}^{*}]$. If $\underline{V}_{own}(p^{*})<\underline{V}_{opp}(p^{*})$,
then $\overline{V}_{own}(p^{*})=V_{own}(p^{*})\ge\underline{V}_{opp}(p^{*})$.
Since both $\overline{V}_{own}(p)$ and $\underline{V}_{opp}(p^{*})$
satisfy \eqref{eq:V_a0}, we must have $\underline{V}_{opp}(p)\le\overline{V}_{own}(p)\le V_{own}(p)$
for all $p\in[\underline{p}^{*},{p}^{*}]$. Either way, we have proven
that $V_{opp}(p)=\underline{V}_{opp}(p)\le V_{own}(p)$ for all $p\in[\underline{p}^{*},{p}^{*}]$.
A symmetric argument proves that $V_{opp}(p)\le V_{own}(p)$ for all
$p\in[{p}^{*},\overline{p}^{*}]$ in case ${p}^{*}<\overline{p}^{*}$.

We have now proven that $V_{own}(p)\ge V_{opp}(p)$ for all $p\in[0,1]$.
Recall from Lemma \ref{lem:properties_contradictory} that $\underline{V}_{own}(\underline{p}^{*})=U^{FA}(\underline{p}^{*})>\overline{V}_{own}(\underline{p}^{*})$
and $\overline{V}_{own}(\overline{p}^{*})=U^{FA}(\overline{p}^{*})>\underline{V}_{own}(\overline{p}^{*})$.
By the intermediate value theorem, there exists $\check{p}\in\left(\underline{p}^{*},\overline{p}^{*}\right)$
where $\underline{V}_{own}(\check{p})=\overline{V}_{own}(\check{p})$.
For any $p$ we have $V_{own}(p)\ge V_{opp}(p)$ and $V_{opp}(p)\ge U^{S}(p)$
and hence $V_{own}(\check{p})\ge U^{S}(\check{p})$. The Crossing
Lemma \ref{lem:branch_crossing} then implies that $\underline{V}_{own}$
cannot cross $\overline{V}_{own}$ from below at $\check{p}$.\footnote{$\underline{V}_{own}$ and $\overline{V}_{own}$ could be equal to
$U^{S}$ at $\check{p}$ which means that two branches are tangent.
However, the convexity of both branches and the fact that $V_{own}(p)\ge U^{S}(p)$
for all $p$, means that $\underline{V}_{own}$ cannot cross $\overline{V}_{own}$
from below at any point of intersection. Therefore $\check{p}$ is
unique.} This means that the intersection point $\check{p}$ is unique and
the structure stated in part (a) obtains.

\textbf{Part (b):} We first prove that $p^{*}\in\left(\underline{p}^{*},\overline{p}^{*}\right)$.
By Lemma \ref{lem:properties_contradictory}, $V_{own}(p)\ge U(p)$
for all $p\in[0,1]$. This implies $V_{opp}(p^{*})>U(p^{*})$, and
since $V_{opp}(p^{*})=U^{S}(p^{*})<U^{FA}(p^{*})$, and since by Lemma
\ref{lem:properties_contradictory}.(d) $U^{FA}(p)\le U(p)$ for $p\notin\left(\underline{p}^{*},\overline{p}^{*}\right)$,
we must have $p^{*}\in\left(\underline{p}^{*},\overline{p}^{*}\right)$.
Next, by Lemma \ref{lem:properties_confirmatory}.(b), $V_{opp}(\underline{p}^{*})<U^{FA}(\underline{p}^{*})=\underline{V}_{own}(\underline{p}^{*})$.
Therefore, $V_{opp}(p)$ and $\underline{V}_{own}(p)$ intersect at
some $\underline{p}\in\left(\underline{p}^{*},p^{*}\right)$ and by
the Crossing Lemma \ref{lem:branch_crossing}, the intersection is
unique since $V_{opp}(\underline{p})>U^{S}(\underline{p})$ for $\underline{p}\in\left(\underline{p}^{*},p^{*}\right)$
by Lemma \ref{lem:properties_confirmatory}.(a). Moreover, for $p<p^{*}$,
we have $V_{opp}(p)>\overline{V}_{own}(p)$ since both satisfy \eqref{eq:V_a0},
and hence $\overline{V}_{own}(p)<\underline{V}_{own}(p)$ for all
$p\in\left(\underline{p}^{*},\underline{p}\right)$. This proves the
result for $p\le p^{*}$. For $p>p^{*}$ the arguments are symmetric.
\end{proof}

\subsubsection{Proof of Proposition \ref{prop:envelope_characterization}}
\begin{proof}
If \eqref{eq:EXPG} is violated, $V_{Env}(p)=U(p)$ since $\underline{p}^{*}>\overline{p}^{*}$
by Proposition \ref{prop:no_experimentation}. Moreover Proposition
\ref{prop:no_experimentation} shows that $V^{*}(p)=U(p)=V_{Env}(p)$
in this case. Similarly, if \eqref{eq:EXPG} is satisfied, by Lemma
\ref{lem:The-boundary-points} and Proposition \ref{prop:Structure_V_G},
$V_{Env}(p)=U(p)$ for all $p\notin(\underline{p}^{*},\overline{p}^{*})$
and Proposition \ref{prop:no_experimentation} shows that $V^{*}(p)=U(p)=V_{Env}(p)$
for $p\notin(\underline{p}^{*},\overline{p}^{*})$.

It remains to verify $V^{*}(p)=V_{Env}(p)$ for $p\in(\underline{p}^{*},\overline{p}^{*})$
when \ref{eq:EXPG} is satisfied. In the remainder of this proof we
write $V(p)=V_{Env}(p)$. Theorem III.4.11 in \citet{bardi/capuzzo-dolceta:97}
characterizes the value function of a dynamic programming problem
with an optimal stopping decision as in \eqref{eq:DMs_problemG} as
the (unique) viscosity solution of the HJB equation.\footnote{To formally apply their theorem, we have to use $P_{t}$ as a second
state-variable and define a value function $v(p,P)=PV(p)$. Since
$v$ is continuously differentiable in $P$, it is straightforward
to apply the result directly to $V(p)$.} For all $p\in(0,1)$ where $V(p)$ is differentiable, this requires
that $V(p)$ satisfy \eqref{eq:HJB_VIO}. 

Consider points of differentiability $p\in\left(\underline{p}^{*},\overline{p}^{*}\right)$.
From \eqref{eq:VctL} and \eqref{eq:VctR}, we obtain that $\underline{V}_{own}$
and $\overline{V}_{own}$ are strictly convex on $(\underline{p}^{*},\overline{p}^{*})$.
Smooth pasting at $\underline{p}^{*}$ and $\overline{p}^{*}$, respectively,
implies that $\underline{V}_{own}(p)>U_{\ell}(p)$ and $\overline{V}_{own}(p)>U_{r}(p)$,
and therefore $V_{own}(p)>U(p)$ for $p\in(\underline{p}^{*},\overline{p}^{*})$.
This implies that \eqref{eq:HJB_VIO} is equivalent to \eqref{eq:HJBG}
for all $p\in\left(\underline{p}^{*},\overline{p}^{*}\right)$. Since
$V(p)$ satisfies \eqref{eq:V_a0} or \eqref{eq:V_a1} at points of
differentiability, and $V(p)\ge V_{opp}(p)\ge U^{S}(p)$, the Unimprovability
Lemma \ref{lem:Branches_satisfy_HJB} implies that $V(p)$ satisfies
\eqref{eq:HJBG}. Since $V_{opp}$ is strictly convex (see the discussion
after Lemma \ref{lem:branch_crossing}), $V_{opp}(p)>U^{S}(p)$, and
hence Lemma \ref{lem:Branches_satisfy_HJB} implies that the optimal
policy is unique at all points where $V(p)$ is differentiable except
$p^{*}$. At $p^{*}$, the HJB equation is satisfied for any $\alpha\in[0,1]$
but $\alpha=1/2$ is the only maximizer that defines an admissible
policy.

We have shown that $V(p)$ satisfies \eqref{eq:HJB_VIO} for all points
of differentiability. For $V(p)$ to be a viscosity solution it remains
to show that for all points of non-differentiability,
\begin{align}
\max\left\{ -c-\rho V(p)+F(p,V(p),\rho),U(p)-V(p)\right\}  & \le0,\label{eq:VIO_non_diff}
\end{align}
for all $\rho\in\left[V'_{-}(p),V'_{+}(p)\right]$; and the opposite
inequality holds for all $\rho\in\left[V'_{+}(p),V'_{-}(p)\right]$,
where $V'_{-}(p)$ denotes the left derivative at $p$, and $V'_{+}(p)$
denotes the right derivative at $p$. By Proposition \ref{prop:Structure_V_G},
non-differentiability at $\check{p}$ if \eqref{eq:EXPG} holds and
$V_{own}(p^{*})\ge V_{opp}(p^{*})$; and at $\underline{p}$ and $\overline{p}$
if \eqref{eq:EXPG} holds and $V_{own}(p^{*})<V_{opp}(p^{*})$. Since
$V(p)\ge U^{S}(p)$, the Crossing Lemma \ref{lem:branch_crossing}
implies that $V(p)$ has convex kinks at all these points so that
$V'_{-}(p)\le V'_{+}(p)$. Therefore it suffices to check \eqref{eq:VIO_non_diff}
for all $\rho\in\left[V'_{-}(p),V'_{+}(p)\right]$. $F_{\alpha}$
is linear in $\alpha$ (see \eqref{eq:HJBG_RHS}), so it suffices
to consider $\alpha\in\left\{ 0,1\right\} $. For $\alpha=1$ we have
$F_{1}(p,V(p),\rho)\le F_{1}(p,V(p),V'_{-}(p))$ and for $\alpha=0$
we have $F_{0}(p,V(p),\rho)\le F_{0}(p,V(p),V'_{+}(p))$. Therefore
if $U(p)\le V(p)$, which holds for our candidate solution by construction,
\begin{equation}
c+\rho V(p)\ge\max\left\{ F_{1}(p,V(p),V'_{-}(p)),F_{0}(p,V(p),V'_{+}(p))\right\} \label{eq:VIO_non_diff_suff}
\end{equation}
implies that \eqref{eq:VIO_non_diff} holds for all for $\rho\in\left[V'_{-}(p),V'_{+}(p)\right]$.
We distinguish two cases.

\textbf{Case A:}\textit{ \eqref{eq:EXPG} is satisfied and $V_{own}(p^{*})\ge V_{opp}(p^{*})$.}
Consider $p=\check{p}$. \eqref{eq:VIO_non_diff_suff} becomes 
\[
c+\rho\underline{V}_{own}(\check{p})=c+\rho\overline{V}_{own}(\check{p})\ge\max\left\{ F_{1}(\check{p},\underline{V}_{own}(\check{p}),\underline{V}_{own}'(\check{p})),F_{0}(\check{p},\overline{V}_{own}(\check{p}),\overline{V}_{own}'(\check{p}))\right\} .
\]
By the Unimprovability Lemma \ref{lem:Branches_satisfy_HJB}, this
holds with equality since $\underline{V}_{own}(p)$ satisfies \eqref{eq:V_a1}
and $\overline{V}_{own}(p)$ satisfies \eqref{eq:V_a0} at $\check{p}$.
As we have argued earlier, $V_{own}(p)>U(p)$ for all $p\in\left(\underline{p}^{*},\overline{p}^{*}\right)$
and hence $V(\check{p})>U(\check{p})$. \eqref{eq:VIO_non_diff} is
thus satisfied at $\check{p}$.

\textbf{Case B:} \textit{\eqref{eq:EXPG} is satisfied and $V_{own}(p^{*})<V_{opp}(p^{*})$.}
The proof is similar to Case B.

We have thus shown that $V(p)$ is a viscosity solution of \eqref{eq:HJB_VIO}
which is sufficient for $V(p)$ to be the value function of problem
\eqref{eq:DMs_problemG}. 
\end{proof}

\subsection{Proof of Proposition \ref{prop:speed_and_accuracy}}
\begin{proof}
(a) Denote the expected delay until the DM takes an action by $\tau(p)$.
At $p^{*}$ the DM uses $\alpha=1/2$. Hence the arrival rate of a
signal is $\lambda/2$ and the expected delay is given by the expectation
of the exponential distribution: 
\[
\tau(p^{*})=\frac{2}{\lambda}.
\]
For $p_{0}\in(p^{*},\overline{p})$, the expected delay must
satisfy a recursive equation with respect to any $t$:
\[
\tau(p_{0})=\int_{0}^{t}s(p_{0}\lambda e^{-\lambda s})ds+(p_{0}e^{-\lambda t})(\tau(p_{t})+t).
\]
Differentiating both sides by $t$ yields 
\begin{align*}
0= & (p_{0}e^{-\lambda t})(\tau'(p_{t})\dot{p}_{t}+1)-(\lambda pe^{-\lambda t})\tau(p_{t}),
\end{align*}
which, upon setting $t=0$, reduces to: 
\begin{align*}
\tau'(p)= & \frac{1-(\lambda p)\tau(p)}{p(1-p)\lambda}.
\end{align*}
Solving this differential equation with boundary condition $\tau(p^{*})=\frac{2}{\lambda}$
and some algebra yields $\tau''(p)<0$. Moreover the right derivative
of $\tau$ at $p^{*}$ is given by 
\[
\tau'(p_{+}^{*})=\frac{1-2p^{*}}{p^{*}(1-p^{*})\lambda}.
\]

Using similar steps for $p_{0}\in(\underline{p},p^{*})$ we have $\tau'(p_{t})=\frac{\lambda(1-p)\tau(p_{t})-1}{p(1-p)\lambda}$
and $\tau''(p)<0$. The left derivative of $\tau$ at $p^{*}$ is
given by $\tau'(p_{-}^{*})=\tau'(p_{+}^{*})$. Since $\tau$ is concave
on $\left(\underline{p},p^{*}\right)$ and on $\left(p^{*},\overline{p}\right)$
and $\tau'(p_{-}^{*})=\tau'(p_{+}^{*})$, we conclude that $\tau$
is concave on $\left(\underline{p},\overline{p}\right)$.

To show that $\tau$ is quasi-concave, it remains to show that $\tau$
is decreasing on $\left[\overline{p},\overline{p}^{*}\right]$ and
increasing on $\left[\underline{p}^{*},\underline{p}\right]$. Since
the argument is essentially the same for both cases, we consider $\left[\overline{p},\overline{p}^{*}\right]$.
The expected delay implied by the own-biased strategy is 
\[
\tau(p)=\int_{0}^{\overline{T}^{*}(p)}s((1-p)\lambda e^{-\lambda s})ds+(1-p)e^{-\lambda\overline{T}^{*}(p)}\overline{T}^{*}(p).
\]
where $\overline{T}^{*}(p)$ is the time it takes for the belief to
reach $\overline{p}^{*}$ in the absence of a signal if the DM follows
the own-biased strategy (i.e., seeks $L$-signals). Since $\overline{T}^{*}(p)$
is decreasing in $p$ we have 
\[
\tau'(p)=(1-p)e^{-\lambda\overline{T}^{*}}\overline{T}^{*\prime}(p)<0.
\]
Therefore it remains to show that $\tau(\overline{p}_{-})\ge\tau(\overline{p}_{+})$.

Suppose $r=0$. If at $\overline{p}$, the DM follows the own-biased
strategy, she enjoys the payoff of 
\begin{equation}
\left[\overline{p}u_{r}^{R}+(1-\overline{p})u_{\ell}^{L}\right]-\overline{p}\left[u_{r}^{R}-u_{\ell}^{R}\right]-c\int_{0}^{\overline{T}^{*}(\overline{p})}(1-H(t)))dt.\label{eq:contr_payoff_ph_r0}
\end{equation}
where 
 $H$ is the distribution of the time at which the DM makes a decision.

Suppose instead that the DM follows the opposite-biased strategy.
In this case her expected payoff (for $r=0$) is given by 
\begin{equation}
\left[\overline{p}u_{r}^{R}+(1-\overline{p})u_{\ell}^{L}\right]-c\int_{0}^{\infty}(1-G(t)))dt.\label{eq:conf_payoff_ph_r0}
\end{equation}
where $G$ is the distribution of time at which the DM makes a decision.

Since at $\overline{p}$, the DM is indifferent between both strategies
we must have 
\[
\int_{0}^{\overline{T}^{*}(\overline{p})}(1-H(t)))dt<\int_{0}^{\infty}(1-G(t)))dt,
\]
i.e., the DM will take a longer time for decision if she chooses a
opposite-biased strategy instead. This proves part (a) of the Proposition
for $r=0$ if $c<\underline{c}$. By continuity the result extends
to $r$ in a neighborhood of zero. For the case that $c\in[\underline{c},\overline{c})$,
it suffices to invoke the argument used for $\left(\overline{p},\overline{p}^{*}\right)$
for the whole interval $\left(\check{p},\overline{p}^{*}\right)$.

(b) Consider $p>\check{p}$ so that the DM uses $\alpha=0$ according
to the opposite-biased strategy. Inserting this in \eqref{eq:pdot}
and integrating we get 
\[
p_{t}=\frac{e^{t\lambda}p_{0}}{1+\left(e^{t\lambda}-1\right)p_{0}}.
\]
Setting $p_{\overline{T}^{*}}=\overline{p}^{*}$ and solving for $\overline{T}^{*}$
we get 
\[
\overline{T}^{*}=\frac{1}{\lambda}\log\left(\frac{\overline{p}^{*}}{\overline{p}^{*}-\overline{p}^{*}}\frac{1-p_{0}}{p_{0}}\right).
\]
The probability of a mistake is therefore 
\[
(1-p)\left(1-e^{-\lambda\overline{T}^{*}}\right)=\frac{\left(1-\overline{p}^{*}\right)\left(\overline{p}^{*}-p_{0}\right)}{\overline{p}^{*}(1-p_{0})}.
\]
Differentiating this with respect to $p_{0}$, we get 
\[
-\frac{\overline{p}^{*}-p_{0}}{\overline{p}^{*}(1-p_{0})}<0.
\]
This proves that the probability of a mistake decreases in the distance
to $\overline{p}^{*}$ for high $p$. For low $p$ the argument is
symmetric. 
\end{proof}

\subsection{Proof of Proposition \ref{prop:Comparative_Statics_Experimentation_Region}\label{subsec:Proof-Comparative-Statics-Experimentation-Region}}
\begin{proof}
(a) By \eqref{eq:intersection_U_and_Uhat}, $\underline{p}^{*}$ and
$\overline{p}^{*}$ are given by the intersections of $U(p)$ and
$U^{FA}(p)$. Since $U(p)$ is independent of $r$ and $c$, and $U^{FA}(p)$
is strictly decreasing in both parameters, the experimentation region
expands as $r$ or $c$ fall. As $(r,c)\rightarrow(0,0)$, we have
$U^{FA}(p)\rightarrow U(p)$ for $p\in\{0,1\}$, hence the experimentation
region converges to $(0,1)$.

(b) The dependence of $\underline{p}^{*}$ and $\overline{p}^{*}$
on $u_{\ell}^{R}$ and $u_{r}^{L}$ is straightforward from the expressions
for the cutoffs in \eqref{eq:plsG} and \eqref{eq:pHsG}. 
 (c) By \eqref{eq:intersection_U_and_Uhat}, $\underline{p}^{*}$
is the intersection between $U_{\ell}(p)$ and $U^{FA}(p)$. The former
is independent of $u_{r}^{R}$ and the latter is increasing in $u_{r}^{R}$.
Hence $\partial\underline{p}^{*}/\partial u_{r}^{R}<0$. Also by \eqref{eq:intersection_U_and_Uhat},
$\overline{p}^{*}$ is the intersection between $U_{r}(p)$ and $U^{FA}(p)$.
We have 
\[
\frac{\partial U_{r}(p)}{\partial u_{r}^{R}}=p>\frac{\lambda}{r+\lambda}p=\frac{\partial U^{FA}(p)}{\partial u_{r}^{R}}.
\]
This implies that $\partial\overline{p}^{*}/\partial u_{r}^{R}<0$.
The comparative statics with respect to $u_{\ell}^{L}$ is derived
similarly. 
\end{proof}

\subsection{Proof of Proposition \ref{prop:Comparative_Statics_Modes_of_Learning}\label{subsec:Proof-Comparative-Statics-Mode-of-Learning}}
\begin{proof}
(a) We prove $\partial\check{p}/\partial u_{\ell}^{R}>0$; the other
case follows from a symmetric argument. Consider $\overline{V}_{own}(p)$.
Since the right branch of the own-biased value function is obtained
from a strategy that takes action $\ell$ only if a signal has been
received, its value is independent of $u_{\ell}^{R}$, as can be seen
from \eqref{eq:VctR}. On the other hand we have $\partial\underline{V}_{own}(p)/\partial u_{\ell}^{R}>0$
from \eqref{eq:VctL}. Therefore the point of intersection of $\underline{V}_{own}$
and $\overline{V}_{own}$ is increasing in $u_{\ell}^{R}$.

(b) It is clear from \eqref{eq:zetal} that $\underline{c}$ is decreasing
in $u_{\ell}^{R}$ and $u_{r}^{L}$. Therefore, it suffices to consider
the case that $c<\underline{c}$. We prove that $\underline{p}\rightarrow0$
monotonically as $u_{\ell}^{R}\rightarrow-\infty$. If a opposite-biased
region exists, $\underline{p}\in(\underline{p}^{*},p^{*})$ is defined
as the unique intersection between $\underline{V}_{opp}(p)$ and $\underline{V}_{own}(p)$.
Note that $\underline{V}_{opp}(p)$ is independent of $u_{\ell}^{R}$
since the opposite-biased strategy never leads to a mistake. As in
(d) we have $\partial\underline{V}_{own}(p)/\partial u_{\ell}^{R}>0$.
Moreover, Lemma \ref{lem:branch_crossing} shows that $\underline{V}_{own}(p)$
crosses $\underline{V}_{opp}(p)$ from above at $\underline{p}$.
Since $V_{opp}$ is independent of $u_{\ell}^{R}$ this implies that
of $\underline{p}$ is monotonically increasing in $u_{\ell}^{R}$.

Since $\underline{p}$ is bounded from below, there exists $q=\lim_{u_{\ell}^{R}\rightarrow-\infty}\underline{p}<p^{*}$.
Suppose by contradiction that $q>0$. Notice that, for each $p\in[q,p^{*}]$,
as $u_{\ell}^{R}\rightarrow-\infty$. 
\[
\underline{V}_{own}(p)\rightarrow\frac{\lambda u_{r}^{R}pr-\lambda c(1-p)-cr}{\left(r+\lambda\right)r}=:\underline{V}_{own}^{\circ}(p),
\]
where we used the fact that $\underline{p}^{*}/(1-\underline{p}^{*})\rightarrow0$
as $u_{\ell}^{R}\rightarrow-\infty$.

Note that the convergence is uniform on $[q,p^{*}]$ since $q>0$.\footnote{Recall from \eqref{eq:VctL} that $\underline{V}_{own}(p)\to\infty$
as $p\to0$, hence the condition $q>0$ is necessary here.} Simple algebra yields 
\begin{align*}
\underline{V}_{own}^{\circ}(p^{*}) & \le U^{S}(p^{*}),\\
\text{and}\quad\underline{V}_{own}^{\circ\prime}(p) & >U^{S\prime}(p).
\end{align*}
Since $\underline{V}_{own}^{\circ}(p)$ is linear in $p$, this implies
that $\underline{V}_{opp}(q)\ge U^{S}(q)>\underline{V}_{own}^{\circ}(q)$
which is a contradiction and we must have $q=0$. The proof for $\overline{p}$
is essentially the same. 
\end{proof}

\subsection{Proof of Proposition \ref{prop:role_of_discounting}\label{subsec:Proof-role-of-discounting}}
\begin{proof}
(a) We have 
\begin{align*}
\frac{\partial\overline{c}}{\partial r} & =-\frac{u_{r}^{R}u_{\ell}^{L}-u_{\ell}^{R}u_{r}^{L}}{(u_{r}^{R}+u_{\ell}^{L})-(u_{\ell}^{R}+u_{r}^{L})},
\end{align*}
and hence $\text{sgn}\left(\partial\overline{c}/\partial r\right)=\text{sgn}\left(u_{\ell}^{R}u_{r}^{L}-u_{r}^{R}u_{\ell}^{L}\right)$.
It is straightforward to verify that $U(\hat{p})>0$ if and only if
$u_{r}^{R}u_{\ell}^{L}-u_{\ell}^{R}u_{r}^{L}>0$.

(b) Denoting $Z(\rho):=(\rho+1)/\left(1+\left(2\rho+1\right)^{\frac{1}{\rho}}\right)$,
we have 
\begin{align*}
\frac{\partial\underline{c}}{\partial r} & =\begin{cases}
Z'(r/\lambda)\left(u_{r}^{R}-u_{\ell}^{R}\right)-ru_{r}^{R} & \text{if }\left(\lambda Z(r/\lambda)-r\right)\left(u_{r}^{R}-u_{\ell}^{L}\right)-\lambda Z(r/\lambda)\left(u_{\ell}^{R}-u_{r}^{L}\right)<0,\\
Z'(r/\lambda)\left(u_{\ell}^{L}-u_{r}^{L}\right)-ru_{\ell}^{L} & \text{if }\left(\lambda Z(r/\lambda)-r\right)\left(u_{r}^{R}-u_{\ell}^{L}\right)-\lambda Z(r/\lambda)\left(u_{\ell}^{R}-u_{r}^{L}\right)>0.
\end{cases}
\end{align*}
Consider the first case. Since $Z'(\rho)\in\left[\left(1+3e^{2}\right)/\left(1+e^{2}\right)^{2},1/2\right]$,
\[
Z'(r/\lambda)\left(u_{r}^{R}-u_{\ell}^{R}\right)-u_{r}^{R}<\frac{1}{2}\left(u_{r}^{R}-u_{\ell}^{R}\right)-u_{r}^{R}=-\frac{1}{2}\left(u_{r}^{R}+u_{\ell}^{R}\right),
\]
which is negative if $u_{r}^{R}>\left|u_{\ell}^{R}\right|$. Conversely,
if $u_{\ell}^{R}$ is sufficiently negative $Z'(r/\lambda)\left(u_{r}^{R}-u_{\ell}^{R}\right)-u_{r}^{R}>0$.
The argument for the second case is similar. 
\end{proof}

\subsection{Proof of Proposition \ref{prop:polarization}}
\begin{proof}
Let $F_{t}(p)$ be the distribution function of beliefs in the
whole population at time $t$. Denote the density, whenever it exists by $f_{t}(p)$. Denote by $\delta_t(p)=F_t(p)-F_t^-(p)$ the mass at $p$ if there is a mass point. For $t=0$ we have the uniform distribution $F_{0}(p)=p$. 

(a) For part (a) we consider the subpopulation of voters with prior beliefs in $\mathcal{P}_{own}=[\underline{p}^*,\underline{p}]\cup [\overline{p},\underline{p}^*]$. Initially, these voters consume own-biased news. If we consider the same subpopulation at later points $t>0$, then their beliefs either remain inside $\mathcal{P}_{own}$, voters who have received an $L$-breakthrough, however, have a belief $p_t=0$. Therefore, for $t>0$ we consider the subpopulation of voters with beliefs in $\mathcal{P}_{own}^0=\mathcal{P}_{own}\cup\{0\}$. Within $\mathcal{P}^0_{own}$ we consider the median belief for voters with $p_t>1/2$, denoted $m_t^r$ and the median belief for voters with $p_t<1/2$, denoted by $m_t^\ell$. 

We first consider $m_t^r$ which is given by
\[
m_{t}^{r}=F_{t}^{-1}\left(F_{t}(\overline{p})+\frac{F_{t}(\overline{p}^{*})-F_{t}(\overline{p})}{2}\right)
\]
We show that this is increasing in $t$ whenever $m_{t}^{r}<\overline{p}^{*}$.
All individuals in $\mathcal{P}_{own}^0 \cup (1/2,1]$  consume $R$-biased news. This leads to two possible changes in their beliefs that effects the median. First, for voters who receive breakthrough news the belief becomes
$0$ so that they leave the set $\mathcal{P}_{own}^0 \cup (1/2,1]$. Note that conditional on the state being $L$ all individuals who acquire information recieve $L$-breakthroughs at rate $\lambda$. If $m^r_t<\overline{p}^*$, all voters in $\mathcal{P}_{own}^0 \cup (1/2,1]$ below the median still acquire information but some voters above the median have already stoped. Therefore, more voters below the median recieve breakthrough than above the median. This increases the median.

Second absent a breakthrough the belief of a voter in $\mathcal{P}_{own}^0 \cup (1/2,1]$ drifts upwards. The upward drift also increses the median.  Hence, if $m^r_t<\overline{p}^*$, $m_{t}^{r}$ is increasing over time. If $m_t^r=\overline{p}^*$, it remains constant for all $t'>t$.

Next consider the subpopulation of individuals with beliefs in $\mathcal{P}_{own}^0 \cup [0,1/2)$.
This subpopulation is composed of (i) the voters who initially consume
own-biased news and have a prior $p_0<1/2$, and (ii) voters who initially
consume own-biased news and have a prior of $p>1/2$, but received
breakthrough news at some time $t'\le t$. The median belief at time
$t$ of individuals with beliefs below $1/2$ in this subset is given
by 
\[
m_{r}^{\ell}=\begin{cases}
\underline{p}^{*}, & \text{if }F_{t}(0)\ge F_{t}(\underline{p})-F_{t}(\underline{p}^{*})\\
F_{t}^{-1}\left(F_{t}(\underline{p})-\frac{\delta_{t}(0)+F_{t}(\underline{p})-F_{t}^{-}(\underline{p}*)}{2}\right), & \text{otherwise.}
\end{cases}
\]
$m_{t}^{\ell}$ is moved by two forces. First, individuals with $p>1/2$
who receive breakthroughs enter the population with $p<1/2$, and since
they have a belief $p=0$ after the breakthrough this reduces the
median. Second, indivduals with beliefs $p<1/2$ who consume own-biased
news never receive breakthroughs if the true state is $L$. Therefore
their beliefs drift downwards which further decreases $m_{t}^{\ell}$.

In summary we have shown that $m_{t}^{r}-m_{t}^\ell$ is increasing
which concludes the proof of part (a).

Part (b) follows from similar arguments since all voters who consume
any news choose own-biased news by assumption. Therefore their beleif
dynamics as as in case (a). The reminaing voters do not consume any
news so that their beliefs remain constant and leave the median in
the subpopulations above and below $1/2$ unaffected.

The proof of part (c) is immediate from the definition of the opposite-biased strategy. 
\end{proof}

\section{Extensions}\label{app:extensions}

\subsection{Discrete Time Foundation}\label{appendix:discrete_time_foundation}

\begin{proof} [Proof of Proposition \ref{prop:microfound}]
	If the DM chooses an experiment with parameters $a$ and $b=1+\lambda dt-a$,
	then the posteriors are $q^{R}:=p\left(\lambda dt+1-a\right)/\left(p\lambda dt+(1-a)\right)$ when the $R$-signal is received, and 
	$ q^{L}:=p\left(a-\lambda dt\right)/\left(a-p\lambda dt\right)$ when an $L$-signal is received.
	The unconditional probabilities of the signals are $\text{Prob}\left[R\text{-signal}\right]=p\lambda dt+(1-a)$, and $\text{Prob}\left[L\text{-signal}\right]=a-p\lambda dt$.
	Hence the DM maximizes 
	\begin{equation}\label{eq:objective_general_experiment}
	\max_{a\in[\lambda dt,1]}\left(p\lambda dt+(1-a)\right)\widetilde{V}(q^{R})+\left(a-p\lambda dt\right)\widetilde{V}(q^{L})
	\end{equation}
	where $\widetilde{V}(q)=\max\left\{ U(q),e^{-\rho dt}V(q) - c\right\} $ and $V(p)$ is the optimal value function. 
	We note that $V(p)$ is
	weakly convex.\footnote{To see this, note that the expected value of a fixed strategy (i.e.~a mapping that specifies the attention choice and action for each
		history) is linear in the prior belief. The value function is therefore
		the upper envelope of a family of linear functions, which implies convexity.}
	Therefore the \emph{continuation value} $\widetilde{V}(p)$ is also weakly convex.
	
In the following, we fix an arbitrary weakly convex continuation value $\widetilde{V}$ and belief $p\in(0,1)$. We show that \eqref{eq:objective_general_experiment} is maximized by $\alpha=\lambda dt$ or $\alpha=1$. To do this, we rewrite the objective in \eqref{eq:objective_general_experiment} for an arbitrary choice $\hat{a}\in[\lambda dt,1]$ in a way that can be bounded by the value for $\alpha=\lambda dt$ or $\alpha=1$.

So we fix any $\hat{a}\in[\lambda dt,1]$ and denote the implied posteriors by $\hat{q}^{R}$ and $\hat{q}^{L}$. To rewrite the objective in \eqref{eq:objective_general_experiment}, we construct alternative payoff parameters $\hat{u}^\omega_x$ so that the resulting stopping payoffs satisfy $\widehat{U}_\ell(\hat{q}^L)=\widetilde{V}(\hat{q}^L)$ and $\widehat{U}'_\ell(\hat{q}^L)=\widetilde{V}'(\hat{q}^L)$, as well as $\widehat{U}_\r(\hat{q}^R)=\widetilde{V}(\hat{q}^R)$ and $\widehat{U}'_r(\hat{q}^R)=\widetilde{V}'(\hat{q}^R)$.\footnote{We use the notation $\widehat{U}_\ell(p):= p\hat{u}_{\ell}^{R}+(1-p)\hat{u}_{\ell}^{R}$, $\widehat{U}_r(p):=p\hat{u}_{r}^{R}+(1-p)\hat{u}_{r}^{R}$, and $\widehat{U}(p):= \max\{\widehat{U}_\ell(p), \widehat{U}_r(p)  \}.$}
Theses conditions yields:
	\begin{align*}
	\hat{u}_{r}^{R} & :=\widetilde{V}(\hat{q}^{R})+(1-\hat{q}^{R})\widetilde{V}'(\hat{q}^{R}),&
	\hat{u}_{\ell}^{R} & :=\widetilde{V}(\hat{q}^{R})-\hat{q}^{R}\widetilde{V}'(\hat{q}^{R}),\\
	\hat{u}_{r}^{L} & :=\widetilde{V}(\hat{q}^{L})+(1-\hat{q}^{L})\widetilde{V}'(\hat{q}^{L}),&
	\hat{u}_{\ell}^{R} & :=\widetilde{V}(\hat{q}^{L})-\hat{q}^{L}\widetilde{V}'(\hat{q}^{L}).
	\end{align*}
	
By definition,  $\widehat{U}(p)$ is tangent to $\tilde V(p)$ at $p=\hat{q}^{L}$ and at $p=\hat{q}^{R}$, and is everywhere weakly below $\widetilde V(p)$, given the convexity of $\widetilde V(p)$.
	
	The objective in \eqref{eq:objective_general_experiment} for $\hat{a}$ can be rearranged and
	bounded as follows:
	\begin{align*}
	& \left(p\lambda dt+(1-\hat{a})\right)\widetilde{V}(\hat{q}^{R})+\left(\hat{a}-p\lambda dt\right)\widetilde{V}(\hat{q}^{L})\\
	= & \left(p\lambda dt+(1-\hat{a})\right)\widehat{U}(\hat{q}^{R})+\left(\hat{a}-p\lambda dt\right)\widehat{U}(\hat{q}^{L})\\
	= & p(1+\lambda dt-\hat{a})\hat{u}_{r}^{R}+(1-p)(1-\hat{a})\hat{u}_{r}^{L}+p(a-\lambda dt)\hat{u}_{\ell}^{R}+(1-p)\hat{a}\hat{u}_{\ell}^{L}\\
	\le & \max_{a\in\{\lambda dt,1\}}p(1+\lambda dt-a)\hat{u}_{r}^{R}+(1-p)(1-a)\hat{u}_{r}^{L}+p(a-\lambda dt)\hat{u}_{\ell}^{R}+(1-p)a\hat{u}_{\ell}^{L}\\
	= & \left(p\lambda dt+(1-\hat{a}^{*})\right)\widehat{U}(\hat{q}^{R*})+\left(\hat{a}^{*}-p\lambda dt\right)\widehat{U}(\hat{q}^{L^{*}})\\
	\le & \left(p\lambda dt+(1-\hat{a}^{*})\right)\widetilde{V}(\hat{q}^{R*})+\left(\hat{a}^{*}-p\lambda dt\right)\widetilde{V}(\hat{q}^{L^{*}}).
	\end{align*}
	In the second line, we have replaced $\widetilde{V}$ by $\widehat{U}$. Writing this out in the third
	line, we see that the expression is linear in $a$. Therefore, maximizing
	over $a\in\{\lambda dt,1\}$, we get a weakly higher value. In the fifth line
	$\hat{a}^{*}$ denotes a maximizer from the forth line and $\hat{q}^{\omega*}$
	denotes the corresponding posterior beliefs. The last inequality follows
	from the fact that  $\widetilde{V}$ is weakly above $\widehat{U}$. This shows that the optimal $a$	can be found in $\{\lambda dt,1\}$.
\end{proof}

\subsection{Non-Exclusivity of Attention\label{subsec:Appendix-Imperfect-Specialization}}

The proofs of our main results only require minor modifications. One
important change is that the full attention strategy has to be defined
using $\alpha=\beta=\overline{\alpha}$. Without this modification,
Lemmas \ref{lem:The-boundary-points} and  \ref{lem:Properties_of_Benchmarks}--\ref{lem:properties_confirmatory}
are no longer valid. We also have to replace $V_{0}$ and $V_{1}$
by solutions to $c+\rho V(p)=F_{\alpha}(p,V(p),V'(p))$ for $\alpha=\underline{\alpha}$
and $\alpha=\overline{\alpha}$, respectively. The value of the stationary
strategy $U^{S}(p)$ is unchanged as discussed in the main text. The
crucial Lemmas \ref{lem:branch_crossing} and \ref{lem:Branches_satisfy_HJB}
continue to hold without modification.

Explicit expressions for the boundaries of the experimentation region
and the absorbing point $p^{*}$ are now given by
\begin{align*}
\underline{p}^{*} & =\frac{u_{\ell}^{L}\rho+c}{\rho\left(u_{\ell}^{L}-u_{\ell}^{R}\right)+\left(u_{r}^{R}-u_{\ell}^{R}\right)\lambda\overline{\alpha}},\\
\overline{p}^{*} & =\frac{\left(u_{\ell}^{L}-u_{r}^{L}\right)\lambda\overline{\alpha}-u_{r}^{L}\rho-c}{\rho\left(u_{r}^{R}-u_{r}^{L}\right)+\left(u_{\ell}^{L}-u_{r}^{L}\right)\lambda\overline{\alpha}},\\
p^{*} & =\frac{\left(u_{\ell}^{L}\rho+c\right)}{\left(u_{r}^{R}\rho+c\right)+\left(u_{\ell}^{L}\rho+c\right)}.
\end{align*}
One can see from the first two expressions that $\underline{p}^{*}$
increases and $\overline{p}^{*}$ decreases if we decrease the upper
bound $\overline{\alpha}$. This confirms the claim that the experimentation
region shrinks if the constraint on $\alpha$ is tightened. 

The cutoffs $\overline{c}$, $\underline{c}$ are given by:
\begin{align}
\overline{c} & :=0\vee\frac{\lambda\overline{\alpha}\left(u_{r}^{R}-u_{\ell}^{R}\right)\left(u_{\ell}^{L}-u_{r}^{L}\right)-\rho\left(u_{r}^{R}u_{\ell}^{L}-u_{\ell}^{R}u_{r}^{L}\right)}{\left(u_{r}^{R}-u_{\ell}^{R}\right)+\left(u_{\ell}^{L}-u_{r}^{L}\right)},\label{eq:zetah-1}\\
\underline{c} & :=0\vee\begin{cases}
\overline{c}\wedge\frac{\lambda}{1+e^{2}}\min\left\{ \left(u_{r}^{R}-u_{\ell}^{R}\right),\left(u_{\ell}^{L}-u_{r}^{L}\right)\right\}  & \text{if }\rho=1-\overline{\alpha}=0,\\
\overline{c}\wedge\min\left\{ \frac{\left(\rho+\lambda\overline{\alpha}\right)\left(u_{r}^{R}-u_{\ell}^{R}\right)}{1+\left(\frac{2\rho+\lambda}{\left(2\overline{\alpha}-1\right)\lambda}\right)^{\frac{2\overline{\alpha}-1}{(1-\overline{\alpha})+\frac{\rho}{\lambda}}}}-\rho u_{r}^{R},\frac{\left(\rho+\lambda\overline{\alpha}\right)\left(u_{\ell}^{L}-u_{r}^{L}\right)}{1+\left(\frac{2\rho+\lambda}{\left(2\overline{\alpha}-1\right)\lambda}\right)^{\frac{2\overline{\alpha}-1}{(1-\overline{\alpha})+\frac{\rho}{\lambda}}}}-\rho u_{\ell}^{L}\right\}  & \text{otherwise.}
\end{cases}\label{eq:zetal-1}
\end{align}
From the first expression it is immediately clear that $\overline{c}$
decreases if we reduce the upper bound $\overline{\alpha}$. It is
less obvious that $\underline{c}$ increases at the same time. To
see this, remember from the proof of Theorem \ref{thm:optimal_attention_strategy}
that $c>\underline{c}$ is equivalent to 
\[
\max\left\{ \underline{V}_{own}(p^{*}),\overline{V}_{own}(p^{*})\right\} >U^{S}(p^{*}).
\]
The right-hand side of this inequality is independent of $\overline{\alpha}$.
The left-hand however, is decreasing in $\overline{\alpha}$.

\subsection{Asymmetric Returns to Attention\label{subsec:Appendix-Asymmetric-Returns}}

In this section we revisit three crucial results that are used to
prove Theorem \ref{thm:optimal_attention_strategy}, and outline how
they are changed if $\overline{\lambda}^{R}\neq\overline{\lambda}^{L}$.
Throughout we assume that $\overline{\lambda}^{R}\ge\overline{\lambda}^{L}$.
Up to relabeling this is without loss of generality.
The three crucial results are: 
\begin{enumerate}
	\item The Crossing Lemma \ref{lem:branch_crossing} and the Unimprovability
	Lemma \ref{lem:Branches_satisfy_HJB}. In Appendix \ref{sec:Main_Appendix},
	we considered solutions $V_{0}$ and $V_{1}$ to the HJB equation
	where we set $\alpha=0$ or $\alpha=1$, respectively. If we generalize
	the HJB equation to allow for $\overline{\lambda}^{R}>\overline{\lambda}^{L}$,
	we can obtain similar solutions $V_{0}$ and $V_{1}$. Lemma \ref{lem:branch_crossing}
	also uses the value of the stationary strategy as a benchmark. The
	definition of the stationary strategy has to be modified if $\overline{\lambda}^{R}>\overline{\lambda}^{L}$.
	The Bayesian updating formula in the absence of a signal is now given
	by:
	\begin{equation}
	\dot{p}_{t}=-\left(\overline{\lambda}^{R}\alpha_{t}-\overline{\lambda}^{L}\beta_{t}\right)p_{t}\left(1-p_{t}\right),\label{eq:pdotAR}
	\end{equation}
	Hence the stationary attention strategy is given by 
	\[
	\alpha^{S}=\frac{\overline{\lambda}^{L}}{\overline{\lambda}^{R}+\overline{\lambda}^{L}}.
	\]
	Note that this coincides with the definition of our main model
	where $\alpha^{S}=1/2$ if $\overline{\lambda}^{R}=\overline{\lambda}^{L}$.
	The value of the stationary strategy is now 
	\[
	U^{S}(p):=p\frac{\alpha^{S}\overline{\lambda}^{R}\,u_{r}^{R}-c}{\rho+\alpha^{S}\overline{\lambda}^{R}}+(1-p)\frac{\beta^{S}\overline{\lambda}^{L}\,u_{\ell}^{L}-c}{\rho+\beta^{S}\overline{\lambda}^{L}}.
	\]
	With this definition, Lemmas \ref{lem:branch_crossing} and 
	\ref{lem:Branches_satisfy_HJB} continue to hold.
	\item Properties of the own-biased strategy in Lemma \ref{lem:properties_contradictory}:\footnote{Lemma 5 repeats the statements of Lemma 
	\ref{lem:The-boundary-points} so we do not discuss Lemma 
	\ref{lem:The-boundary-points} separately.} In Appendix \ref{sec:Main_Appendix} we have constructed the own-biased
	strategy by first obtaining the boundary points $\underline{p}^{*}$
	and $\overline{p}^{*}$ from value matching and smooth pasting. Following
	the same steps while allowing for $\overline{\lambda}^{R}>\overline{\lambda}^{L}$
	we get 
\begin{align}
\underline{p}^{*}&=\frac{u_{\ell}^{L}\rho+c}{\rho\left(u_{\ell}^{L}-u_{\ell}^{R}\right)+\left(u_{r}^{R}-u_{\ell}^{R}\right)\overline{\lambda}^{R}},\label{eq:plsAR}\\
\overline{p}^{*}&=\frac{\left(u_{\ell}^{L}-u_{r}^{L}\right)\overline{\lambda}^{L}-u_{r}^{L}\rho-c}{\rho\left(u_{r}^{R}-u_{r}^{L}\right)+\left(u_{\ell}^{L}-u_{r}^{L}\right)\overline{\lambda}^{L}}.\label{eq:pHsAR}
\end{align}
	The branches of the own-biased solution are then given by particular
	solutions $V_{0}$ and $V_{1}$ that satisfy the boundary conditions
	$V_{0}(\overline{p}^{*})=U_{r}(\overline{p}^{*})$ and $V_{1}(\underline{p}^{*})=U_{\ell}(\underline{p}^{*})$.%

	Lemma \ref{lem:properties_contradictory}.(a)-(b) hold unchanged if $\overline{\lambda}^{R}>\overline{\lambda}^{L}$.
	For the other results in Lemma \ref{lem:properties_contradictory}, we need to modify the definition of the full-attention strategy. We
	compute separately the value of full attention if the DM can obtain
	both types of evidence at rate $\overline{\lambda}^{R}$:
	\begin{align*}
	U_{R}^{FA}(p): & =p\frac{\overline{\lambda}^{R}\,u_{r}^{R}-c}{\rho+\overline{\lambda}^{R}}+(1-p)\frac{\overline{\lambda}^{R}\,u_{\ell}^{L}-c}{\rho+\overline{\lambda}^{R}}=\frac{\overline{\lambda}^{R}\,\left(pu_{r}^{R}+(1-p)u_{\ell}^{L}\right)-c}{\rho+\overline{\lambda}^{R}},
	\end{align*}
	and at rate $\overline{\lambda}^{L}$:
	\begin{align*}
	U_{L}^{FA}(p): & =p\frac{\overline{\lambda}^{L}\,u_{r}^{R}-c}{\rho+\overline{\lambda}^{L}}+(1-p)\frac{\overline{\lambda}^{L}\,u_{\ell}^{L}-c}{\rho+\overline{\lambda}^{L}}=\frac{\overline{\lambda}^{L}\,\left(pu_{r}^{R}+(1-p)u_{\ell}^{L}\right)-c}{\rho+\overline{\lambda}^{L}}.
	\end{align*}
	Generalizing Lemma \ref{lem:properties_contradictory}.(c) we obtain
	now obtain: 
	\[
	U_{\ell}(\underline{p}^{*})=U_{R}^{FA}(\underline{p}^{*}),\qquad\text{and}\qquad U_{r}(\overline{p}^{*})=U_{r}^{FA}(\overline{p}^{*}).
	\]
	Lemma \ref{lem:properties_contradictory}.(d) refers to the condition
	\eqref{eq:EXPG}. If $\overline{\lambda}^{R}>\overline{\lambda}^{L}$,
	we need to define two separate conditions 
	\begin{align}
	\tag{\ensuremath{\text{EXP}_{R}}}U_{R}^{FA}(\hat{p}) & >U(\hat{p}),\label{eq:EXP_R}\\
	\tag{\ensuremath{\text{EXP}_{L}}}U_{L}^{FA}(\hat{p}) & >U(\hat{p}).\label{eq:EXP_L}
	\end{align}
	With this Lemma \ref{lem:properties_contradictory}.(d) generalizes
	as follows: If \eqref{eq:EXP_R} holds, $0<\underline{p}^{*}<\hat{p}$
	and $\underline{V}_{own}(p)<U_{R}^{FA}(p)$ for all $p\in(\underline{p}^{*},1)$,
	$\underline{V}_{own}(p)>U_{R}^{FA}(p)$ for $p<\underline{p}^{*}$
	and $\underline{V}_{own}(p)=U_{R}^{FA}(p)$ if $p\in\left\{ \underline{p}^{*},1\right\} $.
	If \eqref{eq:EXP_L} holds, $0<\underline{p}^{*}<\overline{p}^{*}<\hat{p}$
	and $\overline{V}_{own}(p)<U_{L}^{FA}(p)$ for all $p\in(0,\overline{p}^{*})$,
	$\overline{V}_{own}(p)>U_{L}^{FA}(p)$ for $p>\overline{p}^{*}$ and
	$\overline{V}_{own}(p)=U_{L}^{FA}(p)$ if $p\in\left\{ 0,\overline{p}^{*}\right\} $.\\
	Lemma \eqref{lem:properties_contradictory}.(e) generalizes as follows:
	If \eqref{eq:EXP_L} is violated, then $\overline{V}_{own}=U(p)$ for
	all $p\in[\hat{p},1]$. If \eqref{eq:EXP_R} is violated, then $V_{own}=U(p)$
	for all $p\in[0,\hat{p}]$. 
	\item Properties of the opposite-biased solution in Lemma \ref{lem:properties_confirmatory}: As in the main model,
	we can use smooth pasting and value matching with $U^{S}$ to obtain
	$p^{*}$ as follows: 
	\begin{equation}
	p^{*}=\frac{\left(u_{\ell}^{L}\rho+c\right)\overline{\lambda}^{L}}{\left(u_{r}^{R}\rho+c\right)\overline{\lambda}^{R}+\left(u_{\ell}^{L}\rho+c\right)\overline{\lambda}^{L}}.\label{eq:pstarG_AR}
	\end{equation}
	As before we obtain the branches of the opposite-biased strategy as particular
	solutions $V_{0}$ and $V_{1}$ with the boundary condition $V_{0}(p^{*}),V_{1}(p^{*})=U^{S}(p^{*})$,
	and set

	\[
	V_{own}(p):=\begin{cases}\underline{V}_{own}(p), &\text{if }p<p^*,\\
	\overline{V}_{own}(p), &\text{if }p\ge p^* .
	\end{cases}
	\]
	With this definition Lemma \ref{lem:properties_confirmatory}.(a)
	holds unchanged. Part (b) of the Lemma has to be modified: $V_{opp}(p)=\underline{V}_{opp}(p)\le U_{L}^{FA}(p)$
	for all $p\in[0,p^{*}]$ with strict inequality for $p\neq p^{*}$,
	and $V_{opp}(p)=\overline{V}_{own}(p)\le U_{R}^{FA}(p)$ for all $p\in[p^{*},1]$
	with strict inequality for $p\neq p^{*}$.
\end{enumerate}
The fact that the important Lemmas \ref{lem:branch_crossing} and \ref{lem:Branches_satisfy_HJB} continue to hold and
we still have $V_{opp}(p)>U^{S}(p)$ for all $p\neq p^{*}$ from Lemma
\ref{lem:properties_confirmatory}.(a), together imply that many of the structural properties of $V_{Env}(p)=\max\left\{ V_{own}(p),V_{opp}(p)\right\} $
are preserved and the structure
of the optimal policy is similar to the main model with $\overline{\lambda}^{R}=\overline{\lambda}^{L}$.

However, there is one crucial difference. It is now possible that $\overline{V}_{own}(p)$ is dominated
by $\underline{V}_{own}(p)$ or by $V_{opp}(p)$ for all $p<\overline{p}^{*}$. We only consider the case
that $V_{own}(p^{*})>U^{S}(p^{*})$. In this case we can use similar
steps as in the proof of Proposition \ref{prop:Structure_V_G}.(a)
to show that $V_{Env}(p)=V_{own}(p)$ for all $p\in[0,1]$, i.e., opposite-biased learning is never optimal. However,
it is no longer guaranteed that there exists a point of intersection
$\check{p}\in(\underline{p}^{*},\overline{p}^{*})$ between $\underline{V}_{own}(p)$
and $\overline{V}_{own}(p)$. This is most easily seen by considering
the generalization of Lemma \ref{lem:properties_contradictory}.(c) outlined above.
It is easy to see that $U_{R}^{FA}(p)>U_{L}^{FA}(p)$ for all $p\in[0,1]$
since $\overline{\lambda}^{R}>\overline{\lambda}^{L}$. Since both
functions are strictly decreasing in $c$, we can find levels of $c$
for which $U_{L}^{FA}(p)<U(p)$ for all $p\in[0,1]$ but $U_{R}^{FA}(p)>U(p)$
for some $p$. In this case 
\begin{equation}
V_{own}(p)=\max\left\{ \underline{V}_{own}(p),\overline{V}_{own}(p)\right\} =\max\left\{ \underline{V}_{own}(p),U_{r}(p)\right\} ,\quad\forall p\in[0,1].\label{eq:Vct_AR}
\end{equation}

More generally, \eqref{eq:Vct_AR} may also hold if $U_{L}^{FA}(p)>U(p)$
for some $p\in[0,1]$. Before we could rule out this case since for
$\overline{\lambda}^{R}=\overline{\lambda}^{L}$, $\underline{V}_{own}(p)<U^{FA}(p)$
for all $p\in(\underline{p}^{*},1)$ and $\overline{V}_{own}(\overline{p}^{*})=U^{FA}(\overline{p}^{*})$.
This is not longer true if $\overline{\lambda}^{R}>\overline{\lambda}^{L}$.
The example in Panel (a) of Figure \ref{fig:asymmetric_lamda} depicts
such a case. Moreover we can argue that, as claimed in Section \ref{subsec:Asymmetric-Returns},
this case only arises if $\overline{\lambda}^{R}-\overline{\lambda}^{L}$
is sufficiently large. To see this fix $\overline{\lambda}^{R}$ such
that $\underline{V}_{own}(p)>U(p)$ for some $p$. Note that $\underline{V}_{own}(p)$ is independent
of $\overline{\lambda}^{L}$, since $\underline{p}^*$ does not depend on $\overline{\lambda}^L$ and $\underline{V}_{own}(p)$ is the value of seeking $R$-evidence. Moreover, we have $\underline{V}_{own}(\underline{p}^{*})>U_{r}(\underline{p}^{*})$
and one can easily verify that $\underline{V}_{own}(1)=U_{R}^{FA}(1)<U_{r}(p)$.
Since $U_{r}(p)$ is linear in $p$ and $\underline{V}_{own}(p)$ can
be verified to be strictly convex, there exists a unique $\overline{q}$
such that $\overline{V}_{own}(\underline{q})=U_{r}(\overline{q})$.
If $\overline{q}\ge\overline{p}^{*}$, the Crossing Lemma \ref{lem:branch_crossing}
implies that there exists no intersection of $\underline{V}_{own}(p)$
and $\overline{V}_{own}(p)$ between $\underline{p}^{*}$ and $\overline{p}^{*}$.
In this case \eqref{eq:Vct_AR} holds. Conversely, if $\overline{q}<\overline{p}^{*}$ an
intersection point $\check{p}\in(\underline{p}^{*},\overline{p}^{*})$
exists and $V_{own}(p)$ has the same structure as in our main
model. It remains to argue that $\overline{q}\ge\overline{p}^*$ only if $\overline{\lambda}^{R}-\overline{\lambda}^{L}$
is sufficiently large. For $\overline{\lambda}^{R}-\overline{\lambda}^{L}= 0$, $\underline{V}_{own}(\overline{p}^*)<U^{FA}_R(\overline{p}^*)=U^{FA}_L(\overline{p}^*)=\overline{V}_{own}(\overline{p}^*)=U_r(\overline{p}^*)$. Therefore, $\overline{q}<\overline{p}^*$. Decreasing $\lambda^L$ while holding $\lambda^R$ fixed does not change $\overline{q}$ but decreases $\overline{p}^*$. Hence, there exists a cutoff for $\lambda^L$ below which (for given $\lambda^R$), $\overline{q}\ge\overline{p}^*$.

\subsection{Diminishing Returns to Attention\label{sec:General_Gamma}}

As an extension to the main model in Section \ref{sec:Model},
we show that the general structure of the solution is preserved if the arrival rate of breakthroughs from a given news source does not increase linearly in the amount of attention allocated to the source.
 For the proofs we adopt a different notation than in Section  \ref{subsec:Non-linear_model}. 
 Note that each choice of attention gives rise to a pair of arrival rates $(\lambda^R,\lambda^L)$. For given $g(x)$ a pair $(\lambda^R,\lambda^L)$ is feasible if there exists $\alpha \in [0,1]$ such that $\lambda^R \le \lambda g(\alpha)$ and $\lambda^L \le \lambda g(1- \alpha)$.
 Instead of working with the function $g(x)$ we introduce a function $\Gamma(\lambda^R)$ that characterizes the upper bound of the set of feasible pairs  $(\lambda^R,\lambda^L)$ as follows:%
\footnote{The feasible set of arrival rates can also be derived from a model with many news sources but constant returns to attention. In this model, a news source is now characterized by two parameters $(\lambda^{R},\lambda^{L})$.
If an amount of attention $\alpha_{i}$ is directed to a news-source
given by $(\lambda_{i}^{R},\lambda_{i}^{L})$, the DM will receive
a signal from that source that confirms state $R$ with Poisson arrival
rate $\lambda_{i}^{R}\alpha_{i}$ if the state is indeed $R$ and
she will receive a signal that confirms state $L$ with Poisson arrival
rate $\lambda_{i}^{L}\alpha_{i}$ if the state is $L$. Hence, when
allocating her attention over two news sources with parameters $(\lambda_{i}^{R},\lambda_{i}^{L})$
and $(\lambda_{j}^{R},\lambda_{j}^{L})$ with attention levels $\alpha_{i}$
and $\alpha_{j}=1-\alpha_{i}$, the DM will receive a signal that
confirms $R$ with Poisson rate $\lambda^{R}=\alpha_{i}\lambda_{i}^{R}+(1-\alpha_{j})\lambda_{j}^{R}$,
and a signal that confirms $L$ with Poisson rate $\lambda^{L}=\alpha_{i}\lambda_{i}^{L}+(1-\alpha_{i})\lambda_{j}^{L}$.
The set of feasible arrival rates $(\lambda^{R},\lambda^{L})$ is
thus a weakly convex subset of $\mathbb{R}_{+}$. We denote the upper
bound of this set $\Gamma(\lambda^{R})$ and note that weak convexity
of the set implies weak concavity $\Gamma(\lambda^{R})$. 
In the main model studied before we had $\Gamma(\lambda^{R})=1-\lambda^{R}$,
which is the linear boundary that is spanned by the two primitive
news sources given by $(1,0)$ and $(0,1)$.
}
\[
\left\{ (\lambda^{R},\lambda^{L})\in\mathbb{R}_{+}\,\middle|\,\lambda^{L}\le\Gamma(\lambda^{R})\right\} .
\]
Remember from Section  \ref{subsec:Non-linear_model} that $\lambda^R = \lambda g(\alpha)$ and $\lambda^L = \lambda g(1-\alpha)$. If we normalize $\lambda=1$, we can derive $\Gamma(\lambda)$ from the function  $g(x)$:\footnote{The normalization of the upper bound is without loss of generality
	since only the ratios $\rho/\lambda$ and $c/\lambda$ matter.}
\[
\Gamma(\lambda^R)=g(1-g^{-1}(\lambda^R)).
\]

Clearly, the DM will only chose pairs of arrival rates on the upper
bound, i.e., $(\lambda^{R},\Gamma(\lambda^{R}))$, so we can describe
her choice by $\lambda^{R}$. To simplify the notation we omit the
superscript and write $\lambda$ instead of $\lambda^{R}$. Moreover
we assume that $\lambda\in[0,1]$. We maintain the following assumptions about the function $\Gamma$. 
\begin{assumption}
	\label{assu:Gamma}$\Gamma:[0,1]\rightarrow[0,1]$ is twice continuously
	differentiable, strictly decreasing, strictly convex, and satisfies
	$\Gamma(0)=1$, $\Gamma(1)=0$ and $\Gamma'(\gamma)=-1$, where $\gamma$
	is the unique fixed point of $\Gamma$. 
\end{assumption}

Note that $\Gamma'(\gamma)=-1$ is always fulfilled if $\Gamma$ is derived from a differentiable function $g$ since $\Gamma(\Gamma(x))=x$ in this case which implies that the graph of $\Gamma$ is symmetric with respect to the 45-degree line.

\begin{example}
	\label{exa:Gamma}A parametric example is obtained by setting $g(x)=\sqrt{1+4x+x2}-2$. The inverse is $g^{-1}(x)=2\sqrt{4-2x-x^2}$ and we obtain
\[
\Gamma(\lambda^R)=g(1-g^{-1}(\lambda^R))=\sqrt{6\sqrt{4-2\lambda^R-(\lambda^R)^2}+\lambda^R(2+\lambda^R)-8}-1.
\]	
This is the example used in Figure \ref{fig:Gamma} in Section  \ref{subsec:Non-linear_model}.
\end{example}

\subsubsection{The Decision Maker's Problem}

The DM's posterior evolves according to 
\begin{equation}
\dot{p}_{t}=-p_{t}(1-p_{t})\left(\lambda_{t}-\Gamma(\lambda_{t})\right),\label{eq:pdot_Gamma}
\end{equation}
The objective is given by 
\begin{align*}
J\left((\lambda_{t})_{t\ge0},T;p_{0}\right):= & \begin{Bmatrix}\int_{0}^{T}e^{-\rho t}P_{t}(p_{0},(\lambda_{\tau}))\left(p_{t}\lambda_{t}u_{r}^{R}+(1-p_{t})\Gamma(\lambda_{t})u_{\ell}^{L}\right)dt\\
+e^{-\rho T}P_{T}(p_{0},(\lambda_{\tau}))U(p_{T})
\end{Bmatrix},\\
\text{where }\;P_{t}(p_{0},(\lambda_{\tau})):= & p_{0}e^{-\int_{0}^{t}\lambda_{s}ds}+(1-p_{0})e^{-\int_{0}^{t}\Gamma(\lambda_{s})ds}.
\end{align*}
The DM solves the problem \eqref{eq:P-Gamma} given by: 
\begin{equation}
\tag{\ensuremath{\mathcal{P}^{\Gamma}}}V(p_{0}):=\sup_{\left((\lambda_{t})_{t\ge0},T\right)}J\left((\lambda_{t})_{t\ge0},T;p_{0}\right)\quad\text{s.t. \eqref{eq:pdot_Gamma}, and }\lambda_{t}\in[0,1].\label{eq:P-Gamma}
\end{equation}
We define 
\[
H(p,V(p),V'(p),\lambda):=\begin{Bmatrix}\lambda p\left(u_{r}^{R}-V(p)\right)+\Gamma(\lambda)(1-p)\left(u_{\ell}^{L}-V(p)\right)\\
-p(1-p)(\lambda-\Gamma(\lambda))V'(p)
\end{Bmatrix}.
\]
The HJB equation for \eqref{eq:P-Gamma} is 
\begin{equation}
\max\left\{ -c-\rho V(p)+\max_{\lambda\in[0,1]}H(p,V(p),V'(p),\lambda),U(p)-V(p)\right\} =0.\label{eq:HJB_VIO_Gamma}
\end{equation}
If $V(p)>U(p)$ this simplifies to 
\begin{equation}
c+\rho V(p)=\max_{\lambda\in[0,1]}H(p,V(p),V'(p),\lambda).\label{eq:HJB_Gamma}
\end{equation}
The first-order condition is given by 
\begin{equation}
\frac{\partial H(p,V(p),V'(p),\lambda)}{\partial\lambda}=\begin{Bmatrix}p\left(u_{r}^{R}-V(p)\right)+\Gamma'(\lambda)(1-p)\left(u_{\ell}^{L}-V(p)\right)\\
-p(1-p)(1-\Gamma'(\lambda))V'(p)
\end{Bmatrix}=0.\label{eq:FOC_Gamma}
\end{equation}
For a given policy $\lambda(p)$, we obtain the differential equation
\begin{align}
c+\rho V(p) & =H(p,V(p),V'(p),\lambda(p))\label{eq:ODE_V_Gamma}\\
\iff c+\rho V(p) & =\begin{Bmatrix}\lambda(p)p\left(u_{r}^{R}-V(p)\right)+\Gamma(\lambda(p))(1-p)\left(u_{\ell}^{L}-V(p)\right)\\
-p(1-p)(\lambda(p)-\Gamma(\lambda(p)))V'(p)
\end{Bmatrix}.\nonumber 
\end{align}

As in our original model, we will define two candidate value functions.
For this purpose, we state the HJB equation for problems in which
the DM is either restricted to choose $\lambda\ge\gamma$, 
\begin{eqnarray}
c+\rho V_{+}(p) & =\max_{\lambda\in[\gamma,1]}H(p,V_{+}(p),V_{+}'(p),\lambda),\label{eq:HJB+_Gamma}
\end{eqnarray}
or $\lambda\le\gamma$: 
\begin{eqnarray}
c+\rho V_{-}(p) & =\max_{\lambda\in[0,\gamma]}H(p,V_{-}(p),V_{-}'(p),\lambda).\label{eq:HJB-_Gamma}
\end{eqnarray}
we denote policies corresponding to solution to \eqref{eq:HJB+_Gamma}
and \eqref{eq:HJB-_Gamma} by $\lambda_{+}(p)$ and $\lambda_{-}(p)$,
respectively.

\subsubsection{Preliminary results}

We first revisit some definitions made for the original model. The
stationary strategy is now given by choosing $\lambda=\gamma$ until
a signal arrives and then taking an optimal action according to the
signal. The value of this strategy is now given by 
\[
U^{S}(p)=\frac{\gamma}{\rho+\gamma}U^{*}(p)-\frac{c}{\rho+\gamma},
\]
where 
\[
U^{*}(p)=pu_{r}^{R}+(1-p)u_{\ell}^{L}
\]
is the first best value that is achieved if the DM can learn the state
without any delay.

As in the original model, we obtain a crossing condition for functions
that satisfy \eqref{eq:HJB+_Gamma} and \eqref{eq:HJB-_Gamma} and
a condition under which solutions to \eqref{eq:HJB+_Gamma} and \eqref{eq:HJB-_Gamma}
satisfy \eqref{eq:HJB_Gamma}. 
\begin{lem}[Crossing Lemma]
	\label{lem:single_crossing_general_gamma}Suppose $V_{+}(p)$ is
	$\mathcal{C}^{1}$ at $p$ and satisfies \eqref{eq:HJB+_Gamma} and
	$V_{-}(p)$ is $\mathcal{C}^{1}$ at $p$ and satisfies \eqref{eq:HJB-_Gamma}.
	If $V_{+}(p)=V_{-}(p)\ge U^{S}(p)$, then $V_{+}'(p)\le V_{-}'(p)$.
	If $V_{+}(p)=V_{-}(p)>U^{S}(p)$, then $V_{+}'(p)<V_{-}'(p)$.
\end{lem}
\begin{proof}[Proof of Lemma \ref{lem:single_crossing_general_gamma}]
	Suppose $V(p):=V_{+}(p)=V_{-}(p)\ge U^{S}(p)$ at $p$ and denote
	the maximizers in \eqref{eq:HJB+_Gamma} and \eqref{eq:HJB-_Gamma}
	by $\lambda_{+}(p)$ and $\lambda_{-}(p)$ respectively.
	
	From (\ref{eq:HJB+_Gamma}) and (\ref{eq:HJB-_Gamma}), we obtain
	\begin{align*}
	& p(1-p)(\Gamma(\lambda_{-}(p))-\lambda_{-}(p))(\lambda_{+}(p)-\Gamma(\lambda_{+}(p)))(V_{-}'(p)-V_{+}'(p))\\
	= & \left(\delta(p)\rho+\Delta(p)\right)\left[V(p)-\frac{\frac{\Delta(p)}{\delta(p)}}{\frac{\Delta(p)}{\delta(p)}+\rho}U^{*}(p)+\frac{1}{\frac{\Delta(p)}{\delta(p)}+\rho}c\right]\\
	\ge & \left(\delta(p)\rho+\Delta(p)\right)\left[V(p)-\frac{\gamma}{\gamma+\rho}U^{*}(p)+\frac{1}{\gamma+\rho}c\right]\\
	= & \left(\delta(p)\rho+\Delta(p)\right)\left[V(p)-U^{S}(p)\right],
	\end{align*}
	where
	
	\begin{align*}
	\delta(p):= & \Gamma(\lambda_{-}(p))-\lambda_{-}(p)+\lambda_{+}(p)-\Gamma(\lambda_{+}(p))>0,\\
	\Delta(p):= & \lambda_{+}(p)\Gamma(\lambda_{-}(p))-\lambda_{-}(p)\Gamma(\lambda_{+}(p))>0,
	\end{align*}
	since $\lambda_{+}(p)>\gamma>\lambda_{-}(p)$. The inequality can
	be seen as follows. First, one can verify that $(\Delta(p)/\delta(p),\Delta(p)/\delta(p))$
	is the point of intersection between the forty-five degrees line and
	the line segment between two points, $(\lambda_{-}(p),\Gamma(\lambda_{-}(p)))$
	and $(\lambda_{+}(p),\Gamma(\lambda_{+}(p)))$. Since $\Gamma$ is
	concave, we must have $\Delta(p)/\delta(p)<\gamma$. Since $\delta(p),\Delta(p)\ge0$,
	if $V(p)\ge U^{S}(p)$, the last expression is non-negative, and if
	$V(p)>U^{S}(p)$, it is strictly positive.
\end{proof}
\begin{lem}[Unimprovability]
	\begin{enumerate}
		\item Suppose $V_{+}(p)$ is $\mathcal{C}^{1}$ at $p$ and satisfies \eqref{eq:HJB+_Gamma}.
		If $V_{+}(p)\ge\max\{U^{S}(p),U(p)\}$, then $V_{+}(p)$ satisfies
		\eqref{eq:HJB_Gamma} at $p$. 
		\item Suppose $V_{-}(p)$ is $\mathcal{C}^{1}$ at $p$ and satisfies \eqref{eq:HJB-_Gamma}.
		If $V_{-}(p)\ge\max\{U^{S}(p),U(p)\}$, then $V_{-}(p)$ satisfies
		\eqref{eq:HJB_Gamma} at $p$. 
	\end{enumerate}
	\label{lem:Unimprovability_General_Gamma}
\end{lem}
\begin{proof}[Proof of Lemma \ref{lem:Unimprovability_General_Gamma}]
	We prove the first statement; the second follows symmetrically. Suppose
	the optimal policy satisfies $\lambda_{+}(p)>\gamma$. By the condition,
	it is not improvable by an immediate action or by any $\lambda\ge\gamma$.
	Hence, it suffices to show that it is not improvable by any $\lambda_{-}<\gamma$.
	
	Substituting $V_{+}'(p)$ from \eqref{eq:HJB+_Gamma} and rearranging
	we get 
	\begin{align*}
	& H(p,V_{+}(p),V_{+}'(p),\lambda_{+}(p))-H(p,V_{+}(p),V_{+}'(p),\lambda_{-})\\
	= & \frac{\hat{\delta}(p)\rho+\hat{\Delta}(p)}{\lambda_{+}(p))-\Gamma(\lambda_{+}(p)))}\left[V_{+}(p)-\frac{\frac{\hat{\Delta}(p)}{\hat{\delta}(p)}}{\frac{\hat{\Delta}(p)}{\hat{\delta}(p)}+\rho}U^{S}+\frac{1}{\frac{\hat{\Delta}(p)}{\hat{\delta}(p)}+\rho}c\right]\\
	\ge & \frac{\hat{\delta}(p)\rho+\hat{\Delta}(p)}{\lambda_{+}(p))-\Gamma(\lambda_{+}(p)))}\left[V_{+}(p)-U^{S}(p)\right],
	\end{align*}
	where 
	\[
	\hat{\delta}(p):=\Gamma(\lambda_{-})-\lambda_{-}+\lambda_{+}(p)-\Gamma(\lambda_{+}(p))\mbox{ and }\hat{\Delta}(p):=\lambda_{+}(p)\Gamma(\lambda_{-})-\lambda_{-}\Gamma(\lambda_{+}(p)).
	\]
	The inequality follows from the same observation as in the proof of
	Lemma \eqref{lem:single_crossing_general_gamma}. 
\end{proof}
Before constructing the value function for \eqref{eq:P-Gamma}, we
make one general observation about the boundaries of the experimentation
region and the value opposite-biased signals at the boundaries.

For this purpose we consider a model in which the DM has full attention.
In this case we have $\lambda^{R}=1=\lambda^{L}$ and the DM only
chooses when to stop. Note that Assumption \ref{assu:Gamma} precludes
the DM from choosing $\lambda^{R}=1=\lambda^{L}$ so the full attention
model only serves as a hypothetical benchmark.

The value of this stopping problem is given by 
\[
\widehat{V}(p):=\max\left\{ U(p),U^{FA}(p)\right\} ,
\]
where 
\[
U^{FA}(p)=\frac{1}{\rho+1}U^{*}(p)-\frac{c}{\rho+1}.
\]
Moreover, we note that by Assumption \ref{assu:Gamma}, $(\lambda,\Gamma(\lambda))\le(1,1)$
for all $\lambda\in(0,1)$. Therefore, $\widehat{V}(p)$ is an upper
bound for the value function of the problem \eqref{eq:P-Gamma}.

Remember that in our original  model, the boundaries of the experimentation
region are given by the points of intersection between $U^{FA}(p)$
and $U(p)$: 
\begin{align}
U^{FA}(\overline{p}^{*}) & =U_{r}(\overline{p}^{*}).\label{eq:pLs_Gamma}\\
U^{FA}(\underline{p}^{*}) & =U_{\ell}(\underline{p}^{*}).\label{eq:pHs_Gamma}
\end{align}
If \eqref{eq:EXPG} is satisfied, we have $\underline{p}^{*}<\overline{p}^{*}$.
We now show that the value of \eqref{eq:P-Gamma} is equal to $\widehat{V}$
at these boundaries. This immediately shows that $\underline{p}^{*}$
and $\overline{p}^{*}$ are the boundaries of the experimentation
region in \eqref{eq:P-Gamma}. Moreover, we show that under Assumption
\ref{assu:Gamma}, at these boundaries, the DM does not benefit from
interior choices $\lambda\in(0,1)$. 
\begin{prop}
	Suppose \eqref{eq:EXPG} is satisfied. Then $\underline{p}^{*}$ and
	$\overline{p}^{*}$ given by \eqref{eq:pLs_Gamma} and \eqref{eq:pHs_Gamma}
	are the boundaries of the the experimentation region for the optimal
	solution to \eqref{eq:P-Gamma}. At $\underline{p}^{*}$ and $\overline{p}^{*}$,
	the value of \eqref{eq:P-Gamma} coincides with the value of our original
	model and equals $U^{FA}(p)$ The loss of restricting the DM to chose
	$\lambda\in\left\{ 0,1\right\} $ vanishes as $p\downarrow\underline{p}^{*}$
	and $p\uparrow\overline{p}^{*}$. 
\end{prop}
\begin{proof}
	If the DM is restricted to chose $\lambda\in\left\{ 0,1\right\} $,
	her optimal strategy coincides with the optimal strategy in our original
	model. The value in our original model is a lower bound for the value
	of \eqref{eq:P-Gamma}. Since at $\underline{p}^{*}$ and $\overline{p}^{*}$
	the value in our original model coincides with the upper bound $U^{FA}(p)$,
	it must also coincide with the value of \eqref{eq:P-Gamma}. 
\end{proof}
Note that while Assumption \ref{assu:Gamma} requires $\Gamma(\lambda)<1$
for $\lambda>0$, it does not rule out an Inada condition like $\lim_{\lambda\rightarrow0}\Gamma'(\lambda)=0$.
This shows that at the boundaries of the experimentation region, the
value of a opposite-biased signal is zero even if it is cost-less to
obtain. We will see below when we characterize the value function
that without an Inada condition, there exist neighborhoods of $\underline{p}^{*}$
and $\overline{p}^{*}$ such that the DM does not suffer any loss
if in these neighborhoods she uses $\lambda=1$ and $\lambda=0$,
respectively.

\subsubsection{Construction of Solutions to the HJB equation}

For the remainder of this section, we will focus on the cases that
the payoffs are symmetric. This simplifies the derivations and is
sufficient to understand the main features of the optimal solution
in the extension. Formally we impose: 
\begin{assumption}
	\label{assu:symmetric_payoffs}$u_{r}^{R}=u_{\ell}^{L}=U^{S}$ and
	$u_{\ell}^{R}=u_{r}^{L}=\underline{u}$ for some $\overline{u}>\underline{u}>0$. 
\end{assumption}
In contrast to our original model, it may now be optimal to choose
$\lambda\in(0,1)$ for beliefs $p\in\left(\underline{p}^{*},\overline{p}^{*}\right)$,
i.e., in the interior of the experimentation region. For an interval
where this is the case, we will obtain a differential equation for
$\lambda(p)$ and furthermore an equation that expresses $V(p)$ as
a function of $\lambda(p)$. We begin with the latter. To state the
result in concise form we define 
\[
A(\lambda):=\frac{\Gamma(\lambda)-\Gamma'(\lambda)\,\lambda}{\Gamma(\lambda)-\Gamma'(\lambda)\,\lambda+\rho\left(1-\Gamma'(\lambda)\right)},\quad\text{and}\quad B(\lambda):=\frac{1-\Gamma'(\lambda)}{\Gamma(\lambda)-\Gamma'(\lambda)\,\lambda+\rho\left(1-\Gamma'(\lambda)\right)}.
\]
A basic observation that we will use at several points is that these
two functions are (inverse) U-shaped with (maximum) minimum at $\lambda=\gamma$. 
\begin{lem}
	\label{lem:U-shape_AB}If Assumption \ref{assu:Gamma} is satisfied,
	\[
	A'(\lambda)>(<)0\;\iff\;B'(\lambda)<(>)0\;\iff\;\lambda>(<)\gamma.
	\]
\end{lem}
\begin{proof}
	The Lemma follows from straightforward algebra which we omit here. 
\end{proof}
\begin{lem}
	\label{lem:Solutions_exceed_lower_bound}Suppose Assumptions \ref{assu:Gamma}
	and \ref{assu:symmetric_payoffs} are satisfied. If $p\in(0,1),$
	$V(p)$ is continuously differentiable at $p$ and satisfies \eqref{eq:HJB_Gamma}
	with maximizer $\lambda(p)\neq\gamma$, then 
	\begin{equation}
	V(p)\ge A(\lambda(p))\overline{u}-B(\lambda(p))\,c\ge U^{S}(p)\label{eq:V(p)_as_function_of_lambda(p)}
	\end{equation}
	If $\lambda$ satisfies \eqref{eq:FOC_Gamma} at $p$, then the first
	inequality binds. The statement continues to hold if we replace $V$,
	$\lambda$, and \eqref{eq:HJB_Gamma}, by $V_{+}$, $\lambda_{+}$
	and \eqref{eq:HJB+_Gamma}, or $V_{-}$, $\lambda_{-}$ and \eqref{eq:HJB-_Gamma}. 
\end{lem}
\begin{proof}[Proof of Lemma \ref{lem:Solutions_exceed_lower_bound}]
	We define the LHS of \eqref{eq:FOC_Gamma} as 
	\begin{equation}
	X:=\left(p+(1-p)\Gamma'(\lambda)\right)\left(\overline{u}-V(p)\right)-p(1-p)(1-\Gamma'(\lambda))V'(p).\label{eq:X_Gamma}
	\end{equation}
	Eliminating $V'(p)$ from \eqref{eq:ODE_V_Gamma} and \eqref{eq:X_Gamma}
	we obtain an expression for $V(p)$ in terms of $\lambda(p)$ and
	$X$: 
	\[
	V(p)=A(\lambda(p))\overline{u}-B(\lambda(p))\,c+\frac{X\left(\lambda-\Gamma(\lambda(p))\right)}{\Gamma(\lambda)-\Gamma'(\lambda)\,\lambda+\rho\left(1-\Gamma'(\lambda)\right)}.
	\]
	
	If $\lambda(p)$ is a maximizer in \eqref{eq:HJB_Gamma}, we must
	have 
	\[
	X\begin{cases}
	\ge0 & \text{if }\lambda=1,\\
	=0 & \text{if }\lambda\in(0,1),\\
	\le0 & \text{if }\lambda=0.
	\end{cases}
	\]
	Since $\lambda=1$ implies $\lambda-\Gamma(\lambda(p))>0$ and $\lambda=0$
	implies $\lambda-\Gamma(\lambda(p))<0$ we have 
	\[
	V(p)\ge A(\lambda(p))\overline{u}-B(\lambda(p))\,c,
	\]
	and the inequality holds with equality if $X=0$ which is equivalent
	to $\lambda$ satisfying \eqref{eq:FOC_Gamma}. This proves the first
	inequality and the first statement.
	
	The second inequality follows from Lemma \ref{lem:U-shape_AB} and
	$A(\gamma)\overline{u}-B(\gamma)\,c=U^{S}(p)$, which is obtained
	from straightforward algebra. It is straightforward to adapt the proofs
	to $V_{+}$ and $V_{-}$. 
\end{proof}
Using Lemma \ref{lem:Solutions_exceed_lower_bound} we can obtain
an ODE for $\lambda$ that holds whenever the optimal policy is interior,
i.e., it satisfies \eqref{eq:FOC_Gamma}. 
\begin{lem}
	\label{lem:ODE_lambda}Suppose Assumptions \ref{assu:Gamma} and \ref{assu:symmetric_payoffs}
	are satisfied. If $p\in(0,1),$ $V$ is continuously differentiable
	at $p$ and satisfies \eqref{eq:ODE_V_Gamma} and the maximizer is
	$\lambda(p)\neq\gamma$ and satisfies \eqref{eq:FOC_Gamma} at $p$,
	then 
	\begin{equation}
	\lambda'(p)=\frac{\left[p+(1-p)\Gamma'(\lambda(p))\right]\left[\Gamma(\lambda(p))-\Gamma'(\lambda(p))\,\lambda(p)+\rho\left(1-\Gamma'(\lambda(p))\right)\right]}{p(1-p)\left(\Gamma(\lambda(p))-\lambda(p)\right)\Gamma''(\lambda(p))}.\label{eq:ODE_lamda}
	\end{equation}
	The statement continues to hold if we replace $V$ and $\lambda$,
	by $V_{+}$ and $\lambda_{+}$, or $V_{-}$ and $\lambda_{-}$. 
\end{lem}
\begin{proof}[Proof of Lemma \ref{lem:ODE_lambda}]
	If $\lambda(p)\neq\gamma$ satisfies \eqref{eq:FOC_Gamma}, then
	by Lemma \ref{lem:Solutions_exceed_lower_bound} 
	\begin{align*}
	V(p) & =A(\lambda(p))\overline{u}-B(\lambda(p))\,c,\\
	\text{and}\quad V'(p) & =A'(\lambda(p))\lambda'(p)\overline{u}-B'(\lambda(p))\lambda'(p)\,c.
	\end{align*}
	Inserting these two equations in \eqref{eq:ODE_V_Gamma} and solving
	for $\lambda'(p)$ we get \eqref{eq:ODE_lamda} 
\end{proof}
Next, we state a Lemma that identifies conditions under which the
solution to \eqref{eq:ODE_lamda} remains bounded away from $\lambda=0$
or $\lambda=1$. 
\begin{lem}
	\label{lem:lambdaprime_at_upper_bound}Suppose Assumptions \ref{assu:Gamma}
	and \ref{assu:symmetric_payoffs} are satisfied. Then there exists
	function $p^{1}(x)>1/2$ for $x>\gamma$ and $p^{0}(x)<1/2$ for $x<\gamma$
	such that 
	\begin{align*}
	\lambda(p)=\lambda_{+} & >\gamma\quad\Rightarrow\quad\left\{ \lambda'(p)<0\;\iff\;p<p^{1}(\lambda_{+})\right\} ,\\
	\lambda(p)=\lambda_{-} & <\gamma\quad\Rightarrow\quad\left\{ \lambda'(p)>0\;\iff\;p>p^{0}(\lambda_{-})\right\} .
	\end{align*}
\end{lem}
\begin{proof}
	Inserting $\lambda(p)=\lambda_{+}>\gamma$ in \eqref{eq:ODE_lamda}
	yields 
	\begin{align*}
	\\
	\lambda'(p) & <0\\
	\iff\left[p+(1-p)\Gamma'(\lambda_{+})\right]\frac{\Gamma(\lambda_{+})-\Gamma'(\lambda_{+})\,\lambda_{+}+\rho\left(1-\Gamma'(\lambda_{+})\right)}{p(1-p)\left(\Gamma(\lambda(p))-\lambda(p)\right)\Gamma''(\lambda(p))} & <0\\
	\iff p+(1-p)\Gamma'(\lambda_{+}) & <0\\
	\iff p<p^{1}(\lambda_{+}) & =\frac{\left|\Gamma'(\lambda_{+})\right|}{1+\left|\Gamma'(\lambda_{+})\right|}
	\end{align*}
	Since $\left|\Gamma'(\lambda_{+})\right|>1$ $p^{1}>1/2$. The proof
	for $\lambda(p)=\lambda_{-}<\gamma$ is similar. 
\end{proof}
Next, we show the following property that relates sufficiency of the
FOC \eqref{eq:FOC_Gamma} to convexity of the value function. 
\begin{lem}
	\label{lem:Vct_Gamma_convexity}Suppose Assumptions \ref{assu:Gamma}
	and \ref{assu:symmetric_payoffs} are satisfied. 
	\begin{enumerate}
		\item Let $W:[0,1]\rightarrow\mathbb{R}$ be weakly convex and satisfy $W(p)=U(p)$
		in neighborhoods of $0$ and 1. Then $H(p,W(p),W'(p),\lambda)$ is
		weakly concave in $\lambda$ for all $p$ and strictly concave whenever
		$W(p)>U(p)$. 
		\item Let $\lambda(p)$ be a solution to \eqref{eq:ODE_lamda} such that
		$\lambda(p)\in(0,1)$ at some $p$. Let 
		\[
		\pi(\ell)=\frac{\left(\rho+\ell\right)\Gamma'(\ell)}{\left(\rho+\ell\right)\Gamma'(\ell)-\left(\rho+\Gamma(\ell)\right)}.
		\]
		Then 
		\[
		\frac{\partial^{2}\left[A(\lambda(p))\overline{u}-B(\lambda(p))c\right]}{\partial p^{2}}\ge0\quad\text{if }\,\begin{cases}
		& \lambda(p)>\gamma\text{ and }p\le\pi(\lambda(p)),\\
		\text{or} & \lambda(p)<\gamma\text{ and }p\ge\pi(\lambda(p)).
		\end{cases}
		\]
		$\pi(\ell)>1/2$ if $\ell>\gamma$, and $\pi(\ell)<1/2$ if $\ell<\gamma$. 
	\end{enumerate}
\end{lem}
\begin{proof}
	(a) Some algebra yields 
	\[
	\frac{\partial^{2}H(p,W(p),W'(p),\lambda)}{\partial\lambda^{2}}\le0\quad\iff W(p)-pW'(p)\le U^{S}.
	\]
	The latter inequality is satisfied under the assumptions on $W$ and
	both are strict if $W(p)>U(p)$.
	
	(b) Differentiating $A(\lambda(p))\overline{u}-B(\lambda(p))c$ with
	respect to $p$, substituting $\lambda'(p)$ from \eqref{eq:ODE_lamda}
	and differentiating again yields (after some algebra): 
	\begin{align*}
	\frac{\partial^{2}\left[A(\lambda(p))\overline{u}-B(\lambda(p))c\right]}{\partial p^{2}} & <0\\
	\iff-\frac{\left(p^{2}-(1-p)^{2}\Gamma'(\lambda(p))\right)\left(\rho+\Gamma(\lambda(p))-(\rho+\lambda(p))\Gamma'(\lambda(p))\right)}{p(1-p)\left(\rho+p\lambda(p)+(1-p)\Gamma(\lambda(p))\right)\Gamma''(\lambda(p))} & >\lambda'(p).
	\end{align*}
	Substituting $\lambda'(p)$ from \eqref{eq:ODE_lamda} in the last
	line and rearranging we get 
	\[
	\left(\lambda(p)-\Gamma(\lambda(p))\right)\left(p\left[\rho+\Gamma(\lambda(p))\right]+(1-p)\left[\rho+\lambda(p)\right]\Gamma'(\lambda(p))\right)<0.
	\]
	Solving for $p$ this yields an upper bound if $\lambda(p)>\gamma$
	so that the first term is positive and a lower bound if $\lambda(p)<\gamma$.
	The bound is $\pi(\lambda(p))$ in both cases. If $\ell>\gamma>\Gamma(\ell)$
	we have 
	\begin{align*}
	\pi(\ell) & =\frac{\left(\rho+\ell\right)\left|\Gamma'(\ell)\right|}{\left(\rho+\ell\right)\left|\Gamma'(\ell)\right|+\left(\rho+\Gamma(\ell)\right)}\\
	& >\frac{\left(\rho+\ell\right)\left|\Gamma'(\ell)\right|}{\left(\rho+\ell\right)\left|\Gamma'(\ell)\right|+\left(\rho+\ell\right)}\\
	& =\frac{\left|\Gamma'(\ell)\right|}{\left|\Gamma'(\ell)\right|+1}\\
	& >1/2.
	\end{align*}
	where the last step follows because Assumption \ref{assu:Gamma} implies
	that $\left|\Gamma'(\ell)\right|>1$ if $\ell>\gamma$. Similarly
	we obtain $\pi(\ell)<1/2$ if $\ell<\gamma$. 
\end{proof}

\subsubsection{Solution Candidates}

\paragraph{Own-Biased Learning}

The first candidate is obtained by assuming that the DM uses an own-biased
attention strategy. In contrast to our original model, where we choose
$\lambda\in\{0,1\}$, we will now also use interior values for $\lambda$.
In an own-biased strategy, the DM may now receive breakthrough news
for both states but with a higher likelihood in the state that she
find relatively unlikely. For instance, for low posterior beliefs
$p$, the own-biased strategy involves $\lambda>\gamma$. At the
same time, the belief moves in the same direction as the initial bias
if now breakthrough arrives: $\dot{p}_{t}<0$ if $\lambda>\gamma$.
We have already identified the boundaries of the experimentation region. 
\begin{lem}
	\label{lem:contra_cutoffs_Gamma}Suppose \eqref{eq:EXPG} is satisfied.
	Then $\underline{p}^{*}$ and $\overline{p}^{*}$ satisfy 
	\begin{align}
	p^{*}= & \inf\left\{ p\in[0,\hat{p})\,\middle|\,c+\rho U_{\ell}(p)\le\max_{\lambda\in[\gamma,1]}\begin{Bmatrix}\begin{array}{l}
	\left(\lambda p+\Gamma(\lambda)(1-p)\right)\left(\overline{u}-U_{\ell}(p)\right)\\
	-p(1-p)(\lambda-\Gamma(\lambda))U_{\ell}'(p)
	\end{array}\end{Bmatrix}\right\} ,\label{eq:def_pl*_general_Gamma}\\
	p^{*}= & \sup\left\{ p\in(\hat{p},1]\,\middle|\,c+\rho U_{r}(p)\le\max_{\lambda\in[0,\gamma]}\begin{Bmatrix}\begin{array}{l}
	\left(\lambda p+\Gamma(\lambda)(1-p)\right)\left(\overline{u}-U_{r}(p)\right)\\
	-p(1-p)(\lambda-\Gamma(\lambda))U_{r}'(p)
	\end{array}\end{Bmatrix}\right\} ,\label{eq:def_ph*_general_Gamma}
	\end{align}
	and the maximizers on the right-hand side are given by $\lambda=1$
	and $\lambda=0$, respectively. Moreover, 
	\begin{align*}
	U_{\ell}(\underline{p}^{*}) & \ge A(1)\overline{u}-B(1)c,\\
	\text{and}\quad U_{r}(\overline{p}^{*}) & \ge A(1)\overline{u}-B(1)c.
	\end{align*}
	The first inequality is strict if and only if $\Gamma'(1)$ is finite.
	The second is strict if and only if $\Gamma'(0)<0$. 
\end{lem}
\begin{proof}[Proof of Lemma \ref{lem:contra_cutoffs_Gamma}]
	We only give the proof for $\underline{p}^{*}$, the other case is
	symmetric. Consider the maximization problem in \eqref{eq:def_pl*_general_Gamma}.
	The derivative of the objective function simplifies to $p\left(\overline{u}-\underline{u}\right)$.
	Therefore we can set $\lambda=1$ and \eqref{eq:def_pl*_general_Gamma}
	reduces to the definition via smooth pasting and value matching as
	in our original model.
	
	The first inequality is equivalent to 
	\[
	\frac{1}{(1+\rho)\Gamma'(1)-\rho}\le0,
	\]
	which holds under Assumption \ref{assu:Gamma}. The inequality is
	strict if and only if $\Gamma'(1)$ is finite. The second inequality
	is equivalent to 
	\[
	\frac{\Gamma'(0)}{1+\rho-\rho\Gamma'(0)}\le0,
	\]
	which is strict if and only if $\Gamma'(0)<0$. 
\end{proof}
We are now ready to define the opposite-biased strategy. Given that
we impose Assumption \ref{assu:Gamma}, we only describe the construction
for the left branch which is used for $p\le1/2$. There are up to
four intervals where the opposite-biased strategy takes a different
form. First, for $p\le\underline{p}^{*}$, the DM takes immediate
action. Then there is an interval $(\underline{p}^{*},\underline{q}^{b}]$
where the DM uses the own-biased strategy from our original model.
$\underline{q}^{b}$ is given by 
\begin{align*}
\frac{\partial H(\underline{q}^{b},\underline{V}_{own}(\underline{q}^{b}),\underline{V}'_{own}(\underline{q}^{b}),1)}{\partial\lambda} & =0.
\end{align*}
Rearranging this we get 
\[
\frac{(1+\rho)\Gamma'(\underline{q}^{b})}{\rho-(1+\rho)\Gamma'(\underline{q}^{b})}+\underline{q}^{b}+(1-\underline{q}^{b})\left(\frac{1-\underline{q}^{b}}{\underline{q}^{b}}\frac{\underline{p}^{*}}{1-\underline{p}^{*}}\right)^{\rho}=0,
\]
which is equivalent to 
\begin{align*}
\underline{V}_{opp}(\underline{q}^{b}) & =A(1)\overline{u}-B(1)c.
\end{align*}

By Lemma \ref{lem:contra_cutoffs_Gamma}, $\underline{q}^{b}=\underline{p}^{*}$
if $\Gamma'(1)$ is infinite and otherwise $\underline{q}^{b}>\underline{p}^{*}$.
If $\underline{q}^{b}\ge1/2$ we define the own-biased strategy
as in our original model. If $\underline{q}^{b}<1/2$, Lemma \ref{lem:lambdaprime_at_upper_bound}
implies that $\lambda'(\underline{q}^{b})<0$ if we impose the boundary
condition $\lambda(\underline{q}^{b})=1$. Denote the unique solution
for $p\ge\underline{q}^{b}$ to \eqref{eq:ODE_lamda} with $\lambda(\underline{q}^{b})=1$
by $\lambda(p;\underline{q}^{b},1)$. Since by Lemma \ref{lem:lambdaprime_at_upper_bound},
$\lambda'(p;p,1)<0$ for all $p\le1/2$, we have $\lambda(p;\underline{q}^{b},1)<1$
for $p\in(\underline{q}^{b},1/2)$. Finally we need to take care of
the possibility that there exists $\underline{q}^{s}\in(\underline{q}^{b},1/2]$
such that $\lambda(p;\underline{q}^{b},1)=\gamma$. If no such $\underline{q}^{s}$
exists we set $\underline{q}^{s}=1/2$. If Assumption \ref{assu:symmetric_payoffs}
is satisfied, a symmetric construction can be used for the right branch
with cutoffs $\overline{q}^{b}=1-\underline{q}^{b}$ and $\overline{q}^{s}=1-\underline{q}^{s}$.

We thus define the opposite biased strategy as follows. For $p\notin\left(\underline{p}^{*},\overline{p}^{*}\right)$:
take the optimal immediate action. For $p\in\left(\underline{p}^{*},\overline{p}^{*}\right)$,
experiment according to the following attention strategy: 
\[
\lambda_{own}^{\Gamma}(p)=\begin{cases}
1, & \text{if }p\in(\underline{p}^{*},\underline{q}^{b}],\\
\lambda(p;\underline{q}^{b},1), & \text{if }p\in(\underline{q}^{b},\underline{q}^{s}],\\
\gamma, & \text{if }p\in(\underline{q}^{s},\overline{q}^{s}),\\
\lambda(p;\overline{q}^{b},0), & \text{if }p\in[\overline{q}^{s},q^{b}),\\
0, & \text{if }p\in[q^{b},\overline{p}^{*}),
\end{cases}
\]
and take an action corresponding to the signal if one is received.\footnote{If $\underline{q}^{s}=\overline{q}^{s}$, $\lambda_{own}^{\Gamma}(\overline{q}^{s})\in\left\{ \lambda(\underline{q}^{s};\underline{q}^{b},1),\lambda(\underline{q}^{s};\overline{q}^{b},0)\right\} $
	with an arbitrary tie-breaking rule.} Note that by Lemma \ref{lem:lambdaprime_at_upper_bound}, $\lambda_{own}^{\Gamma}(p)$
is strictly decreasing if $p\in(\underline{q}^{b},\underline{q}^{s}]\cup[\overline{q}^{s},q^{b})$.
The value of this strategy is given by 
\[
V_{own}^{\Gamma}(p)=\begin{cases}
V_{own}(p), & \text{if }p\le\underline{q}^{b},\\
A\left(\lambda(p;\underline{q}^{b},1)\right)\overline{u}-B\left(\lambda(p;\underline{q}^{b},1)\right)c, & \text{if }p\in(\underline{q}^{b},\underline{q}^{s}],\\
U^{S}(p), & \text{if }p\in(\underline{q}^{s},\overline{q}^{s}),\\
A\left(\lambda(p;\overline{q}^{b},0)\right)\overline{u}-B\left(\lambda(p;\overline{q}^{b},0)\right)c, & \text{if }p\in[\overline{q}^{s},q^{b}),\\
V_{own}(p), & \text{if }p\ge q^{b},
\end{cases}
\]
where $V_{own}(p)$ denotes the value of the opposite-biased strategy
from our original model. Note that since we focus attention on the
symmetric case (Assumption \ref{assu:symmetric_payoffs}), the belief
that separates the ``left branch'' and the ``right branch'' of
the opposite-biased solution is given by $\check{p}$. Note also,
that in contrast to our original model, we defined the own-biased
strategy in a way that it is always weakly greater than $U^{S}(p)$.

The implied dynamics of the posterior as well as the attention strategy
are summarized by the following diagram: 
\[
|\underbrace{\text{------------------ }}_{\text{immediate action }b}\underline{p}^{*}\overbrace{\underbrace{\longleftarrow\longleftarrow}_{\lambda=1}\underline{q}^{b}\underbrace{\longleftarrow\longleftarrow}_{\lambda\in(\gamma,1)}\underline{q}^{s}\underbrace{\text{---}\check{p}\text{---}}_{\lambda=\gamma}\overline{q}^{s}\underbrace{\longrightarrow\longrightarrow}_{\lambda\in(0,\gamma)}\overline{q}^{b}\underbrace{\longrightarrow\longrightarrow}_{\lambda=0}}^{\text{information acquisition}}\bar{p}^{*}\underbrace{\text{------------------ }}_{\text{immediate action }a}|
\]
\begin{lem}
	\label{lem:Vct_satisfies_HJB_Gamma}Suppose Assumptions \ref{assu:Gamma}
	and \ref{assu:symmetric_payoffs} are satisfied. Then $V_{own}^{\Gamma}$
	is continuously differentiable and convex on $\left[0,\underline{q}^{s}\right)$
	and on $\left(\overline{q}^{s},1\right]$, respectively, and satisfies
	\eqref{eq:HJB_Gamma} on $\left[\underline{p}^{*},\underline{q}^{s}\right)$
	and on $\left(\overline{q}^{s},\overline{p}^{*}\right]$, respectively. 
\end{lem}
\begin{proof}
	We show the Lemma for $p\le1/2$. The remaining results follow from
	a symmetric argument.
	
	We need to show that $V_{own}^{\Gamma}$ is continuously differentiable
	at $\underline{q}^{b}$. For $r>0$, some algebra yields for $p\le1/2$\footnote{The derivation for $r=0$ is similar.}
	\begin{align*}
	V_{own}^{\Gamma} & =A(1)\overline{u}-B(1)c\\
	\iff\left(\frac{\underline{p}^{*}}{1-\underline{p}^{*}}\frac{1-p}{p}\right)^{\rho} & =1-\frac{r}{\left(1-p\right)\left(\rho-(1+\rho)\Gamma'[1]\right)}.
	\end{align*}
	Substituting this expression in $V_{own}^{\Gamma\prime}(p)$ yields
	\begin{align*}
	\left.V_{own}^{\Gamma\prime}(p)\right|_{V_{own}(p)=A(1)\overline{u}-B(1)c} & =\frac{\left(c+\rho U^{S}\right)\left(p+(1-p)\Gamma'[1]\right)}{\left(1-p\right)p\left(\rho-(1+\rho)\Gamma'[1]\right)}\\
	& =\left.\frac{\partial\left[A(\lambda(p))\overline{u}-B(\lambda(p))c\right]}{\partial\lambda}\right|_{\lambda(p)=1}.
	\end{align*}
	Convexity on $\left[\underline{p}^{*},\underline{q}^{s}\right]$ follows
	from strict convexity of $V_{own}$ (Lemma \ref{lem:properties_contradictory})
	and strict convexity of $A(\lambda(p))\overline{u}-B(\lambda(p))c$
	(Lemma \ref{lem:Vct_Gamma_convexity}.(b)) and continuous differentiability.
	
	Note that by Lemma \ref{lem:Unimprovability_General_Gamma}, it suffices
	to show that $V_{own}^{\Gamma}$ satisfies \eqref{eq:HJB+_Gamma} for
	all $\left[\underline{p}^{*},\underline{q}^{s}\right)$ since $V_{own}^{\Gamma}(p)>U^{S}(p)$
	for $p<\underline{q}^{s}$. We have derived $V_{own}^{\Gamma}$ from
	the first order-condition \eqref{eq:FOC_Gamma} and the respective
	Kuhn-Tucker condition of $p<\underline{q}^{b}$. Therefore it suffices
	to show that the maximization problem in the HJB equation is concave.
	By Lemma \ref{lem:Vct_Gamma_convexity}.(a), this is the case since
	we have shown that $V_{own}^{\Gamma}$ is weakly convex. 
\end{proof}

\paragraph{Opposite-Biased Learning}

The second candidate for the value function is obtained by assuming
that the DM uses an opposite-biased attention strategy. Specifically,
we define a ``reference belief'' $p^{*}$ such that the DM chooses
$\lambda<\gamma$ for lower beliefs $p<p^{*}$ and $\lambda>\gamma$
for higher beliefs $p>p^{*}$. The implied dynamics of the posterior
as well as the attention strategy are summarized by the following
diagram: 
\[
|\underbrace{\longrightarrow\longrightarrow\longrightarrow\longrightarrow}_{\lambda\in[0,\gamma)}p^{*}\underbrace{\longleftarrow\longleftarrow\longleftarrow\longleftarrow}_{\lambda\in(\gamma,1]}|
\]

The reference belief is absorbing and we assume that once $p^{*}$
is reached, the DM adopts the stationary attention strategy $\lambda=\gamma$.
Under Assumption \ref{assu:symmetric_payoffs}, we have $p^{*}=1/2$.
This can also be derived from value matching 
\begin{equation}
V(p^{*})=U^{S}(p^{*})(=U^{S}),\label{eq:boundary_pstar_general_Gamma}
\end{equation}
and the tangency condition 
\begin{equation}
V'(p^{*})=U^{S\prime}(p^{*})(=0).\label{eq:Vprime_pstar_general_Gamma}
\end{equation}
Substituting these two conditions together with $\lambda=\gamma$
in \eqref{eq:FOC_Gamma} yields $p^{*}=1/2$.\footnote{Note that in contrast to the linear model, we cannot use the HJB equation
	because for $\lambda=\gamma$, $V'(p)$ vanishes so that substituting
	\eqref{eq:Vprime_pstar_general_Gamma} has no bite.}

We would now like to construct the opposite-biased strategy in a similar
fashion as the own-biased solution, that is, we will identify two
types of regions. If $\lambda\in\left\{ 0,1\right\} $, we will use
solutions to \eqref{eq:V_a0} or \eqref{eq:V_a1} (with $\overline{\lambda}=1$,
$\underline{\lambda}=1$ and $\alpha$ replace by $\lambda$.) On
the other hand, if $\lambda\in(0,1)$ we will use solutions to \eqref{eq:ODE_lamda}
with a suitable boundary condition. A problem arises since we want
to impose the boundary condition $\lambda(p^{*})=\gamma$. Note that
this implies $\lambda'(p^{*})=0/0$. We therefore begin by identifying
a solution to \eqref{eq:ODE_lamda} that satisfies $\lambda(p^{*})=\gamma$
as well as $\lambda'(p^{*})>0$. 
\begin{lem}
	\label{lem:Solution_lambda}Suppose Assumptions \ref{assu:Gamma}
	and \ref{assu:symmetric_payoffs} are satisfied. Then there exists
	a unique continuously differentiable function $\hat{\lambda}_{opp}(p)$
	which satisfies \eqref{eq:ODE_lamda} for all $p$ in a neighborhood
	of $p^{*}=1/2$, such that $\lambda(p^{*})=\gamma$ and $\lambda'(p^{*})>0$.
	The derivative at $p^{*}$ is given by 
	\[
	\hat{\lambda}_{opp}'(p^{*})=-\left(\rho+\gamma\right)+\sqrt{\left(\rho+\gamma\right)^{2}-\frac{8\left(\rho+\gamma\right)}{\Gamma''(\gamma)}}.
	\]
\end{lem}
\begin{proof}[Proof of Lemma \ref{lem:Solution_lambda}]
	The ODE \eqref{eq:ODE_lamda} can be written as 
	\[
	\lambda'(p)=\frac{P(p,\lambda(p))}{Q(p,\lambda(p))},
	\]
	where 
	\begin{align*}
	P(p,\lambda) & =\left[\Gamma(\lambda)-\Gamma'(\lambda)\,\lambda+\rho\left(1-\Gamma'(\lambda)\right)\right]\mathbf{\times}\left[p+\Gamma'(\lambda)\,(1-p)\right],\\
	Q(p,\lambda) & =p(1-p)\Gamma''(\lambda)\left[\Gamma(\lambda)-\lambda\right].
	\end{align*}
	Since $P$ and $Q$ are both continuous and have continuous partial
	derivatives, the behavior of solutions that go through points in a
	neighborhood of $\left(p^{*},\gamma\right)$ is, under some conditions
	(see below), the same as for\footnote{See e.g. \citet{Bronshtein2007}.}
	\begin{equation}
	\lambda'(p)=\frac{a\,(p-p^{*})+b\,(\lambda(p)-\gamma)}{c\,(p-p^{*})+d\,(\lambda(p)-\gamma)},\label{eq:ODE_lambda_taylor}
	\end{equation}
	where 
	\begin{align*}
	a & =\partial_{p}P(p^{*},\gamma)=4\left(\rho+\gamma\right)>0,\\
	b & =\partial_{\lambda}P(p^{*},\gamma)=\left(\rho+\gamma\right)\Gamma''(\gamma)<0,\\
	c & =\partial_{p}Q(p^{*},\gamma)=0,\\
	d & =\partial_{\lambda}Q(p^{*},\gamma)=-\frac{1}{2}\Gamma''(\gamma)>0.
	\end{align*}
	The characteristic equation is 
	\begin{align*}
	x^{2}-bx-ad & =0.
	\end{align*}
	Since $ad>0$, the characteristic equation has two reals roots of
	opposite sign. This implies that $(p^{*},\gamma)$ is a saddle point
	and there are two continuously differentiable solutions $\lambda(p)$
	that pass through $(p^{*},\gamma)$. In the case of a saddle point,
	the behavior of the solutions of \eqref{eq:ODE_lambda_taylor} in
	a neighborhood of $(p^{*},\gamma)$ corresponds to the behavior of
	the solutions to \eqref{eq:ODE_lamda}. Hence there exist two continuously
	differentiable solutions $\lambda(p)$ that satisfy the boundary condition
	$\lambda(p^{*})=\gamma$.
	
	Next we want to obtain $\lambda'(p^{*})$ for these solutions, and
	show that only one of them has a positive derivative. We have 
	\begin{align*}
	\lambda'(p^{*})=\lim_{p\rightarrow p^{*}}\lambda'(p) & =\lim_{p\rightarrow p^{*}}\frac{P(p,\lambda(p))}{Q(p,\lambda(p))}\\
	& =\lim_{p\rightarrow p^{*}}\frac{\partial_{p}P(p,\lambda(p))+\partial_{\lambda}P(p,\lambda(p))\lambda'(p)}{\partial_{p}Q(p,\lambda(p))+\partial_{\lambda}Q(p,\lambda(p))\lambda'(p)}\\
	& =\frac{a+b\lambda'(p^{*})}{d\lambda'(p^{*})}.
	\end{align*}
	Hence $\lambda'(p^{*})$ solves 
	\begin{align*}
	x^{2}-\frac{b}{d}\,x-\frac{a}{d} & =0,\\
	\lambda'(p^{*}) & =\frac{b}{2d}\pm\sqrt{\left(\frac{b}{2d}\right)^{2}+\frac{a}{d}}.
	\end{align*}
	Since $a/d>0$, there is one positive and one negative solution. For
	the opposite-biased solution, we are interested in a solution that satisfies
	$\lambda'(p^{*})>0$. Hence we have 
	\begin{align*}
	\lambda'(p^{*}) & =\frac{b}{2d}+\sqrt{\left(\frac{b}{2d}\right)^{2}+\frac{4\left(\rho+\gamma\right)}{d}}\\
	& =-\left(\rho+\gamma\right)+\sqrt{\left(\rho+\gamma\right)^{2}-\frac{8\left(\rho+\gamma\right)}{\Gamma''(\gamma)}}.
	\end{align*}
\end{proof}
Lemma \ref{lem:Solution_lambda} provides the solution $\hat{\lambda}_{opp}$
which together with $V(p)=A(\hat{\lambda}_{opp}(p))\overline{u}$ defines
$V_{opp}$ in a neighborhood of $p^{*}$. To extend this definition
to $\left(0,1\right)$ we first extend $\hat{\lambda}_{opp}$ to the
maximal interval $(\underline{q},\overline{q})$ where $\hat{\lambda}_{opp}(p)\in(0,1)\setminus\{\gamma\}$
unless $p=p^{*}$. 
\begin{lem}
	\label{lem:extended_solution_lambda}Suppose Assumptions \ref{assu:Gamma}
	and \ref{assu:symmetric_payoffs} are satisfied. There exist two points
	$0\le\underline{q}<p^{*}<\overline{q}\le1$ such that 
	\begin{enumerate}
		\item $\hat{\lambda}_{opp}(p)$ is well defined as the unique $\mathcal{C}^{1}$-solution
		to \eqref{eq:ODE_lamda} that satisfies the properties in Lemma \ref{lem:Solution_lambda} 
		\item $\hat{\lambda}_{opp}(p)>\gamma$ if $p>p^{*}$ and $\hat{\lambda}_{opp}(p)<\gamma$
		if $p<p^{*}$. 
		\item Either $\underline{q}=0$ or $\hat{\lambda}_{opp}(\underline{q})=0$. 
		\item Either $\overline{q}=1$ or $\hat{\lambda}_{opp}(\overline{q})=1$. 
	\end{enumerate}
\end{lem}
Note that Properties (c) and (d) mean that the interval $(\underline{q},\overline{q})$
is the maximal interval where $\hat{\lambda}_{opp}(p)\in(0,1)$. 
\begin{proof}[Proof of Lemma \ref{lem:extended_solution_lambda}]
	Consider the interval $(\underline{q},p^{*})$. $\hat{\lambda}_{opp}(p)\in(0,\gamma)$
	in a neighborhood of $p^{*}$. Moreover, \eqref{eq:ODE_lamda} satisfies
	local Lipschitz continuity if $p\in(0,p^{*})$ and $\lambda\neq\gamma$.
	Hence, if there exists a $\mathcal{C}^{1}$ solution to \eqref{eq:ODE_lamda}
	with initial condition $\hat{\lambda}_{opp}(p^{*}-\varepsilon)\in(0,\gamma)$
	that satisfies $\hat{\lambda}_{opp}(p)\in(0,\gamma)$ for all $p\in(\underline{q},p^{*})$,
	then it is the unique such solution. We first show that by extending
	the interval from a neighborhood of $p^{*}$ to $(\underline{q},p^{*})$,
	we do not violate $\hat{\lambda}_{opp}(p)<\gamma$. Suppose by contradiction
	that there exists $p'<p^{*}$ such that $\lim_{p\searrow p'}\hat{\lambda}_{opp}(p)\nearrow\gamma$.
	Note that 
	\[
	p'+\Gamma'(\gamma)(1-p')<p^{*}+\Gamma'(\gamma)(1-p^{*})=0.
	\]
	Hence, since $\Gamma''<0$,$\lim_{p\searrow p'}\hat{\lambda}'_{opp}(p)\rightarrow\infty$
	which contradicts $\lim_{p\searrow p'}\hat{\lambda}_{opp}(p)\nearrow\gamma$.
	Therefore we can extend the domain of $\hat{\lambda}_{opp}(p)$ to
	the left until either $p=0$ or $\hat{\lambda}_{opp}(p)=0$. This completes
	the proof for $p<p^{*}$ and the argument for $p>p^{*}$ is similar. 
\end{proof}
If $\underline{q}>0$ and $\overline{q}<1$, respectively, then we
further extend $\lambda{}_{opp}(p)$ to $(0,1)$ by setting $\lambda=0$
for $p<\underline{q}$ and $\lambda=1$ for $p>\overline{q}$. We
define 
\[
\lambda^{\Gamma}{}_{opp}(p):=\begin{cases}
0, & \text{if }p\le\underline{q},\\
\hat{\lambda}{}_{opp}(p), & \text{if }p\in(\underline{q},\overline{q}),\\
1, & \text{if }p\ge\overline{q}.
\end{cases}
\]
The value of this strategy is given by 
\[
V_{opp}^{\Gamma}(p):=\begin{cases}
V_{0}\left(p;\underline{q},A(0)\overline{u}-B(0)c\right) & \text{if }p\le\underline{q},\\
A(\lambda{}_{opp}(p))\overline{u} & \text{if }p\in(\underline{q},\overline{q}),\\
V_{1}\left(p;\overline{q},A(1)\overline{u}-B(1)c\right) & \text{if }p\ge\overline{q}.
\end{cases}
\]
\begin{lem}
	Suppose Assumptions \ref{assu:Gamma} and \ref{assu:symmetric_payoffs}
	are satisfied. Then $V_{opp}^{\Gamma}(p)$ is a $\mathcal{C}^{1}$
	solution to \eqref{eq:HJB_Gamma} and $V_{opp}^{\Gamma}(p)$ is strictly
	convex on $\left(\underline{q},\overline{q}\right)$.
\end{lem}
\begin{proof}
	The proof has several steps. We give arguments for $p\ge1/2$. The
	Lemma then follows by symmetry (Assumption \ref{assu:symmetric_payoffs})
	and the fact that $V_{opp}^{\Gamma}(p)$ is constructed in a way that
	is continuously differentiable at $p^{*}$ (see \eqref{eq:Vprime_pstar_general_Gamma}).
	Suppose in the following that $p>1/2$.
	
	First we note that $V_{opp}^{\Gamma}(p)$ is continuously differentiable.
	This holds by construction for $p\neq\overline{q}$ and at $\overline{q}$
	it follows by the same argument as in the proof of Lemma \ref{lem:Vct_satisfies_HJB_Gamma}.
	
	Second, we shows that $V_{opp}^{\Gamma}(p)$ is strictly convex. For
	$p>1/2$, $\lambda_{opp}^{\Gamma}(p)>\gamma$ . Therefore, by Lemma
	\ref{lem:Vct_Gamma_convexity}, strict convexity on $(p^{*},\overline{q})$
	follows if $p<\pi(\lambda_{opp}^{\Gamma}(p))$ for all $p\in(p^{*},\overline{q})$.
	Note that $\pi(\lambda_{opp}^{\Gamma}(p^{*}))=\pi(\gamma)=1/2$. We
	show that whenever $p=\pi(\lambda_{opp}^{\Gamma}(p))$, then $\pi'(\lambda_{opp}^{\Gamma}(p))\lambda_{opp}^{\Gamma\prime}(p)>1$.
	This implies that $p<\pi(\lambda_{opp}^{\Gamma}(p))$ for all $p\in(p^{*},\overline{q})$.
	We have 
	\begin{align*}
	\pi'(\lambda_{opp}^{\Gamma}(p^{*}))\lambda_{opp}^{\Gamma\prime}(p^{*}) & >1\\
	\iff\frac{2-(\rho+\gamma)\Gamma''(\gamma)}{4(\rho+\gamma)}\left(\sqrt{\left(r+\gamma\right)^{2}-\frac{8\left(\rho+\gamma\right)}{\Gamma''(\gamma)}}-\left(\rho+\gamma\right)\right) & >1\\
	\iff\Gamma''(\gamma) & <0.
	\end{align*}
	for $p>p^{*}$, we substitute $p=\pi(\lambda_{opp}^{\Gamma}(p))$ in
	\eqref{eq:ODE_lamda}, which yields (after some algebra) 
	\[
	\pi'(\lambda_{opp}^{\Gamma}(p^{*}))\lambda_{opp}^{\Gamma\prime}(p^{*})=1+\frac{\Gamma'(\lambda_{opp}^{\Gamma}(p))\left(\rho+\Gamma(\lambda_{opp}^{\Gamma}(p))-\left(\rho+\lambda_{opp}^{\Gamma}(p)\right)\Gamma'(\lambda_{opp}^{\Gamma}(p))\right)}{\left(\rho+\lambda_{opp}^{\Gamma}(p)\right)\left(\rho+\Gamma(\lambda_{opp}^{\Gamma}(p))\right)\Gamma''(\gamma)}>1.
	\]
	This completes the proof of convexity on $(p^{*},\overline{q})$.
	For $p>\overline{q}$, convexity has been shown in Lemma \ref{lem:properties_confirmatory}.
	Since $V_{opp}^{\Gamma}(p)$ is continuously differentiable at $p=\overline{q}$,
	$V_{opp}^{\Gamma}(p)$ is strictly convex on $[0,1]$.
	
	Third, by Lemma \ref{lem:Vct_Gamma_convexity}.(a), convexity implies
	that the maximization problem in \eqref{eq:HJB+_Gamma} is concave
	so that the first-order condition is sufficient. Therefore, $V_{opp}^{\Gamma}(p)$
	satisfies \eqref{eq:HJB+_Gamma} or for $p>p^{*}$.
	
	Finally, convexity, together with \eqref{eq:boundary_pstar_general_Gamma}
	and \eqref{eq:Vprime_pstar_general_Gamma} implies that $V_{opp}^{\Gamma}(p)\ge U^{S}(p)$
	for $p\ge p^{*}$. Lemma \ref{lem:Unimprovability_General_Gamma}
	then implies that $V_{opp}^{\Gamma}(p)$ satisfies \eqref{eq:HJB_Gamma}. 
\end{proof}
Finally we show that $\lambda_{opp}^{\Gamma}(p)$ is strictly increasing.
\begin{lem}
	Suppose Assumptions \ref{assu:Gamma} and \ref{assu:symmetric_payoffs}
	are satisfied and let $\underline{q},\overline{q}$ be given as in
	Lemma \ref{lem:extended_solution_lambda}. Then $\lambda_{opp}^{\Gamma}(p)$
	is strictly increasing on $\left(\underline{q},\overline{q}\right)$. 
\end{lem}
\begin{proof}
	For $p\in(\underline{q},\overline{q})$, $V_{opp}^{\Gamma}(p)=A(\lambda{}_{opp}(p))\overline{u}$.
	Differentiating with respect to $p$ we get 
	\[
	V_{opp}^{\Gamma\prime}(p)=A'(\lambda{}_{opp}(p))\lambda'{}_{opp}(p)\overline{u}.
	\]
	Hence if $\lambda'(p)=0$ for $p\neq1/2$, we must have $V_{opp}^{\Gamma\prime}(p)=0$.
	Since $V_{opp}^{\Gamma\prime}(1/2)=0$, this violates strict convexity
	of $V_{opp}^{\Gamma}(p)$. Therefore $\lambda'(p)\neq0$ for all $p\in(\underline{q},\overline{q})$.
	Since $\lambda'(1/2)>0$, this implies that $\lambda'(p)>0$ if $p\in(\underline{q},\overline{q})$.
\end{proof}

\subsubsection{Optimal Solution}

As in our original model we show that the value function $V^{\Gamma}$
is the upper envelope of the two solution candidates. In contrast to our original model, the optimal policy is not a bang-bang solution. We show that inside the  own-biased region, $\alpha(p)=g^{-1}(\lambda(p))$ is decreasing whenever it is not a corner-solution. This means that more extreme beliefs lead to a more own-biased news-diet.  In the opposite-biased region, $\alpha(p)$ is strictly increasing. This implies that more moderate beliefs lead to a more balanced news-diet.
\begin{thm}
	\label{thm:Value_function_Gamma}Suppose Assumptions \ref{assu:Gamma}
	and \ref{assu:symmetric_payoffs} are satisfied. 
	\begin{enumerate}
		\item If \eqref{eq:EXPG} is violated then $V^{\Gamma}(p)=U(p)$ for all
		$p\in[0,1]$. 
		\item If \eqref{eq:EXPG} is satisfied and $V_{own}^{\Gamma}(p)>U^{S}(p)$
		for all $p\neq1/2$, then $V^{\Gamma}(p)=V_{own}^{\Gamma}(p)$ for
		all $p\in[0,1]$, and $\alpha(p)=g^{-1}(\lambda(p))$ is strictly decreasing if $V^{\Gamma}(p)>U(p)$ and $\alpha(p)=g^{-1}(\lambda(p))\in(0,1)$.
		\item If \eqref{eq:EXPG} is satisfied and $V_{own}^{\Gamma}(p)=U^{S}(p)$
		for some $p\neq1/2$, then $V^{\Gamma}(p)=\max\left\{ V_{own}^{\Gamma}(p),V_{opp}^{\Gamma}(p)\right\} $,
		and $\alpha(p)=g^{-1}(\lambda(p))$ is strictly decreasing if $V^{\Gamma}(p)=V_{own}^{\Gamma}(p)>U(p)$ and $\lambda(p)\in(0,1)$,
		and strictly increasing if $V^{\Gamma}(p)=V_{opp}^{\Gamma}(p)$.
	\end{enumerate}
\end{thm}
\begin{proof}[Proof of Theorem \ref{thm:Value_function_Gamma}]
	Follows from the same arguments as the proof of Theorem \ref{thm:optimal_attention_strategy}. 
\end{proof}

\subsection{Multiple Actions\label{sec:Third-Action}}

In this Appendix, we extend the model in Section \ref{sec:Model}
to include a third action $x=m$ which yields $u_{m}^{R}$ and $u^L_{m}$ in states $R$ and $L$. Up to relabeling of the actions it is without loss to assume that $u_{m}^{R}\in\left(u_{\ell}^{R},u_{r}^{R}\right)$. Further we assume $u^L_{m}< u_{\ell}^{L}$ which guarantees that action $m$ does not dominate action $\ell$ for all beliefs. 

The optimal policy will be affected by the availability of action $m$ if it is optimal to take this action for some beliefs. To identify when this is the case, we define a strategy that specifies a stopping region $[\underline{p}_m,\overline{p}_m]$ in which action $m$ is taken immediately. For $p>\overline{p}_m$, the strategy prescribes attention to the $L$-biased news source ($\alpha=1$) and for $p<\underline{p}_m$, the strategy prescribes attention to the $R$-biased news source ($\alpha=0$). We call this strategy the ``$m$-strategy.'' It has the following structure:
\[
\underset{p=}{\phantom{|}}\underset{0}{|}\underbrace{\longrightarrow\longrightarrow\vphantom{p^{*}}\negthickspace\negmedspace\longrightarrow\longrightarrow}_{\alpha=0}\underline{p}_m\underbrace{\text{\ensuremath{\vphantom{p^{*}}}---------------------}}_{\text{immediate action }m}\overline{p}_m\underbrace{\longleftarrow\vphantom{p^{*}}\negthickspace\negmedspace\longleftarrow\longleftarrow\longleftarrow}_{\alpha=1}\underset{1}{|}
\]

If this strategy is part of the optimal solution (for some range of belief), the boundary points $\underline{p}_m$ and $\overline{p}_m$ must satisfy value-matching and smooth-pasting
conditions that resemble those used to define $\underline{p}^{*}$
and $\overline{p}^{*}$. We will define $\overline{p}_m$ by imposing smooth
pasting and value matching with $U_{m}(p)$ in \eqref{eq:V_a1}:
\begin{align}
c+\rho U_{m}(p) & =\lambda p \left(u_{r}^{R}-U_{m}(p)\right)-\lambda p(1-p)U_{m}'(p).\label{eq:def_pmh}
\end{align}
Similarly we will define $\underline{p}_m$ by imposing smooth pasting and
value matching with $U_{m}(p)$ in \eqref{eq:V_a0}:
\begin{align}
c+\rho U_{m}(p)= & \lambda(1-p)\left(u_{\ell}^{L}-U_{m}(p)\right)+\lambda p(1-p)U_{m}'(p).\label{eq:def_pml}
\end{align}

The following lemma identifies when solutions to \eqref{eq:def_pmh} and \eqref{eq:def_pml} exist, and when these solutions can be used to define the cutoffs $\underline{p}_m$ and $\overline{p}_m$ in a way the $m$-strategy only prescribes information acquisition if it is not dominated by immediate action $m$ or by the stationary strategy. 

\begin{lem}\label{lem:m_cutoffs}~

	\begin{enumerate}
		\item Let $u^R_m \ge U^{FA}(1)$ or $c+\rho u^L_m \le 0$. If $q_1 \in (0,1)$ is a solution to \eqref{eq:def_pmh}, then $V_1(p;q_1,U_m(q_1))\le U_m(p)$ for all $p\in[q_1,1]$.

		\item If $u^R_m < U^{FA}(1)$ and $c+\rho u^L_m > 0$, then there exists a unique solution $q_1 \in (0,1)$ to \eqref{eq:def_pmh} given by
		\begin{equation}\label{eq:q1}
			q_1 = \frac{u_{m}^{L}\rho+c}{\rho\left(u_{m}^{L}-u_{m}^{R}\right)+\left(u_{r}^{R}-u_{m}^{R}\right)\lambda}.
		\end{equation}
		and $V_1(p;q_1,U_m(q_1))$ is strictly convex on $[q_1,1]$.

		\item Let $u^L_m \ge U^{FA}(0)$ or $c+\rho u^R_m \le 0$. If $q_2 \in (0,1)$ is a solution to \eqref{eq:def_pml}, then $V_0(p;q_2,U_m(q_2))\le U_m(p)$ for all $p\in[0,q_2]$.

		\item If $u^L_m < U^{FA}(0)$ and $c+\rho u^R_m > 0$, then there exists a unique solution $q_2 \in (0,1)$ to \eqref{eq:def_pml} given by 
		\begin{equation}\label{eq:q2}
			q_2 = \frac{\left(u_{\ell}^{L}-u_{m}^{L}\right)\lambda-u_{m}^{L} \rho -c}{\rho\left(u_{m}^{R}-u_{m}^{L}\right)+\left(u_{\ell}^{L}-u_{m}^{L}\right)\lambda}
		\end{equation}
		and $V_0(p;q_2,U_m(q_2))$ is strictly convex on $[0,q_2]$.
		\item Suppose $u^R_m < U^{FA}(1)$ and $c+\rho u^L_m > 0$, and $u^L_m < U^{FA}(0)$ and $c+\rho u^R_m > 0$. If $U_m(q_1)\ge U^S(q_1)$ and $U_m(q_2)\ge U^S(q_2)$, then $q_1\ge q_2$.
	\end{enumerate}
\end{lem}
	
\begin{proof}
	For (a) and (b) we note that the general solution to \eqref{eq:V_a1} is given by 
	\begin{equation*}
	 	V_1(p)=\underbrace{\frac{p \rho u^R_r \lambda-c \left(\rho +(1-p)\lambda \right)}{\rho(\rho + \lambda)}}_{=:z(p)}+\left(\frac{1-p}{p}\right)^{\frac{\rho}{\lambda}}(1-p)\,C,
	\end{equation*}
	there $C$ is the constant of integration. Clearly, the sign of $C$ determines whether the solution is convex or concave since 
	\[
	\frac{d^2}{dp^2}\left( \left( \frac{1-p}{p}\right)^{\frac{\rho}{\lambda}}(1-p) \right) > 0
	\]
	Moreover, we note the $V_1(1)=U^{FA}(1)$ regardless of the value of the constant $C$.

	For the proof of (a) we distinguish several cases: Case 1: If $u^R_m = U^{FA}(1)$ and $c+\rho u^L_m = 0$. In this case,  $U_m(p)=z(p)$ for all $p$. Hence any $q_1\in(0,1)$ satisfies \eqref{eq:def_pmh} and smooth pasting but $V(p;q_1,U_m(q_1))=U_m(p)$ for all $p \in [0,1]$.

	Case 2: $u^R_m > U^{FA}(1)$. We first show that if \eqref{eq:def_pmh} and smooth pasting is satisfied for $p'\in(0,1)$, then $U_m(p')<z(p)$. Suppose by contradiction that $U_m(p')\ge z(p')$. If \eqref{eq:def_pmh} is satisfies at $p'$, then $U_m(p)$ is tangent to $V_1(p;p',U_m(p')$at $p'$ and since $V_1(p;p',U_m(p'))\ge z(p)$, $V_1(p;p',U_m(p'))$ is weakly  convex as a function of $p$. But this implies that $V_1(1;p',U_m(p'))\ge U_m(1)=u^R_m>U^{FA}(1)$. This is a contradiction since we argued above that any solution to \eqref{eq:V_a1} satisfies $V_1(1)=U^{FA}(1)$. Hence \eqref{eq:def_pmh} or smooth pasting is vioalated at $p'$ if $U_m(p')\ge z(p')$. If  $U_m(p')\le z(p')$, \eqref{eq:def_pmh}, and smooth pasting is satisfied for $p'\in(0,1)$, then $V_1(p;p',U_m(p'))$ is strictly concave as a function of $p$ and tangent to $U_m(p)$ at $p'$. Hence $V_1(p;p',U_m(p'))< U_m(p)$ for all $p>p'$.

	Case 3: $u^R_m < U^{FA}(1)$. If $c + \rho u^L_m \le 0$, then $U_m(0)<z(0)$ and since $z(1)=U^{FA}(1)$ we have $U_m(p)<z(p)$ for all $p$. As in case 2, if $p'\in(0,1)$ satisfies \eqref{eq:def_pmh} and smooth pasting, then $V_1(p;p',U_m(p'))< U_m(p)$ for all $p>p'$ which contradicts $V_1(1;p',U_m(p'))=U^{FA}(1)$. Hence there is no solution to \eqref{eq:def_pmh} that satisdies smooth pasting. This concludes the proof of (a).

	For (b), note that if $u^R_m < U^{FA}(1)$ and $c + \rho u^L_m > 0$, then $U_m(p)$ crosses $z(p)$ from above. As in case 3 in the proof for part (a), $z(p')>U_m(p)$ implies that  \eqref{eq:def_pmh} and smooth pasting cannot be both satisfied. Next we identify a solution $q_1$ to \eqref{eq:def_pmh} for which $V'(q_1;q_1,U_m(q_1))=U'_m(q_1)$. If $U_m(q_1)=z(q_1)$ then $V'(q_1;q_1,U_m(q_1))=z'(q_1)=U'_m(q_1)$. On the hand $\lim_{q_1 \rightarrow 0} V'(q_1;q_1,U_m(q_1)) = -\infty$. Therefore, the intermediate value theorem implies that there exists $q_1\in (0,1)$ such that $V'(q_1;q_1,U_m(q_1))=U'_m(q_1)$ and simple algebra shows that is is given by \eqref{eq:q1}.

	The proofs of (c) and (d) follow from a similar argument. For part (e) suppose by contradiction that $q_1<q2$. Since both $V_1(p;q_1,U_m(q_1))$ and $V_0(p;q_0,U_m(q_0))$ are strictly convex on $[q_1,q_2]$ and coincide with $U_m(p)$ at $q_1$ and $q_2$, respectively, there exists $p'\in(q_1,q_2)$ such that $V_1(p';q_1,U_m(q_1))=V_0(p';q_0,U_m(q_0))>U_m(p')$ and $V'_1(p';q_1,U_m(q_1))>V'_0(p';q_0,U_m(q_0))$. Since $U_m(p)\ge U^S(p)$ for $p\in{q_1,q_2}$ and both $U_m$ and $U^S$ are linear, we have $V_1(p';q_1,U_m(q_1))=V_0(p';q_0,U_m(q_0))>U_S(p')$. By Lemma \ref{lem:branch_crossing} this implies $V'_1(p';q_1,U_m(q_1))<V'_0(p';q_0,U_m(q_0))$ which is a contradiction. Therefore we must have $q_1\ge q_2$. 
\end{proof}
Based on the results of this lemma, we define $\underline{p}_m$ and $\overline{p}_m$ as follows:
\begin{align*}
\overline{p}_m & =
\begin{cases}
	q_1,& \text{ if }u^R_m < U^{FA}(1),\: c+\rho u^L_m > 0 ,\text{ and }U_m(q_1)\ge U^S(q_1)\\
	1,& \text{ otherwise.}
\end{cases}\\
\underline{p}_m & = \begin{cases}
	q_2,& \text{ if }u^L_m < U^{FA}(0),\: c+\rho u^R_m > 0 ,\text{ and }U_m(q_2)\ge U^S(q_2)\\
	0,& \text{ otherwise.}
\end{cases}
\end{align*}

Consider $\overline{p}_m$. By Lemma \ref{lem:m_cutoffs}.(a)-(b) $u^R_m<U^{FA}(1)$ together with $c+\rho u^L_m > 0$ is a necessary and sufficient condition for the existence of a solution in $(0,1)$ to \eqref{eq:def_pmh} that satisfies smooth pasting and is not dominated by immediate action $m$. Hence if the necessary and sufficient condition is violated we set $\overline{p}_m=1$. Similarly, Lemma \ref{lem:m_cutoffs}.(c)-(d) motivates the definition of $\underline{p}_m=0$ if $u^B_m\ge U^{FA}(0)$ and $c+\rho u^R_m \le 0$. 

The requirements that $U_m(q_1)\ge U^S(q_1)$ in the definition of $\overline{p}_m$ and $U_m(q_2)\ge U^S(q_2)$ in the definition of $\underline{p}_m$, guarantee, respectively, that the $m$-strategy always has the structure depicted in the diagram above because it avoids defining $\overline{p}_m=q_1$ and $\underline{p}_m=q_2$ when $q_2>q_1$.

The value of the $m$-strategy is 
\[
V_{m}(p):=\begin{cases}
V_{0}(p;\underline{p}_m,U_{m}(\underline{p}_m)), & \mbox{ for }p<\underline{p}_m,\\
U_{m}(p), & \text{ for }p\in\left[\underline{p}_m,\overline{p}_m\right],\\
V_{1}(p;\overline{p}_m,U_{m}(\overline{p}_m)), & \mbox{ for }p>\overline{p}_m.
\end{cases}
\]

The Lemmas leading to the upper envelope characterization of the value
function in Proposition \ref{prop:envelope_characterization} depend on
the properties of branches defined by particular solutions to \eqref{eq:V_a0}
and \eqref{eq:V_a1}. Therefore the same steps can be applied in this
extension and we obtain that the value function of the extended problem
is given by: 
\[
V(p)=\max\left\{ V_{own}(p),V_{opp}(p),V_{m}(p)\right\} .
\]

It is straightforward to extend this to more than three actions. Suppose we have actions $\ell,r$ as well as additional actions $m_1,m_2,\ldots$, where for all $i=1,2,\ldots$, $(u^R_{m_i},u^A_{m_i})$ satisfy the conditions formulated for action $m$ at the beginning of this section. In this case we define an $m_i$-strategy for each of the actions in the same way as above. Denote the value of strategy $m_i$ by $V_{m_i}(p)$. The value function of the DM's problem is then given by 
\[
V(p)=\max\left\{ V_{own}(p),V_{opp}(p),V_{m_1}(p),V_{m_2}(p),\ldots \right\} .
\]

\end{document}